\documentclass{amsart}
\usepackage{amssymb,mathrsfs,amsmath}
\usepackage{color,umoline}
\usepackage[dvipsnames]{xcolor}
\usepackage{graphicx}
\usepackage{enumitem}
\usepackage[font=small,labelfont=bf]{caption}
\usepackage{tikz, caption}
\usepackage{relsize}
\usepackage[mathscr]{eucal}
\usepackage{dsfont}
\usepackage{verbatim}
\usepackage[draft]{todonotes}
\usepackage[
citecolor=black,  pagebackref,
hypertexnames=false]{hyperref}
\usepackage[dvipsnames]{xcolor}
\usetikzlibrary{arrows.meta}
\usepackage{subfig}
\usepackage{mathtools}
\usepackage{kotex}
\usepackage{amsthm}
\usepackage{parskip}

\makeatletter
\DeclareRobustCommand\widecheck[1]{{\mathpalette\@widecheck{#1}}}
\def\@widecheck#1#2{%
    \setbox\z@\hbox{\m@th$#1#2$}%
    \setbox\tw@\hbox{\m@th$#1%
       \widehat{%
          \vrule\@width\z@\@height\ht\z@
          \vrule\@height\z@\@width\wd\z@}$}%
    \dp\tw@-\ht\z@
    \@tempdima\ht\z@ \advance\@tempdima2\ht\tw@ \divide\@tempdima\thr@@
    \setbox\tw@\hbox{%
       \raise\@tempdima\hbox{\scalebox{0.9}[-1.2]{\lower\@tempdima\box
\tw@}}}%
    {\ooalign{\box\tw@ \cr \box\z@}}}
\makeatother

\DeclareMathOperator\supp{supp}
\DeclareMathOperator\dist{dist}
\DeclareMathOperator\rank{rank}

\newtheorem{thm}{Theorem}[section]

\newtheorem{conj}[thm]{Conjecture}
\newtheorem{prop}[thm]{Proposition}
\newtheorem{cor}[thm]{Corollary}
\newtheorem{lem}[thm]{Lemma}
\numberwithin{equation}{section}

\theoremstyle{definition}

\newcommand{\C}{\mathbb{C}}
\newcommand{\N}{\mathbb{N}}
\newcommand{\R}{\mathbb{R}}
\newcommand{\Z}{\mathbb{Z}}

\newcommand{\bs}{\mathbf{S}}
\newcommand{\bi}{\mathbf{I}}

\newcommand{\bfa}{\mathbf a}
\newcommand{\bfb}{\mathbf b}

\newcommand{\pz}{\partial_z^\intercal}
\newcommand{\pzp}{\partial_{z'}}

\newcommand{\baz}{\mathbf v(z,z')}

\newcommand{\px}{\partial_x^\intercal}
\newcommand{\py}{\partial_y}

\newcommand{\ep}{\epsilon}

\newcommand{\inp}[2]{\langle #1, #2\rangle}
\newcommand{\inpb}[2]{\Big \langle #1, #2 \Big\rangle}
\newcommand{\inpm}[2]{\big \langle #1, #2 \big\rangle}
\newcommand{\wt}[1]{\widetilde #1}

\newcommand{\clp}{\mathcal P_{\!\mathcal L} (t,z,z')}
\newcommand{\chp}{\mathcal P_{\!\,\mathcal H} (t,x,y)}

\newcommand{\epc}{\epsilon_0}

\newcommand{\tsupp}{\supp \chil\times \supp \chil' }

\newcommand{\Be}{\begin{equation}}
\newcommand{\Ee}{\end{equation}}
\newcommand{\Bea}{\begin{eqnarray}}
\newcommand{\Eea}{\end{eqnarray}}
\newcommand{\Bel}{\begin{align}}
\newcommand{\Eel}{\end{align}}
\newcommand{\Bels}{\begin{align*}}
\newcommand{\Eels}{\end{align*}}
\newcommand{\Beas}{\begin{eqnarray*}}
	\newcommand{\Eeas}{\end{eqnarray*}}
\newcommand{\Benu}{\begin{enumerate}}
	\newcommand{\Eenu}{\end{enumerate}}
\newcommand{\Bi}{\begin{itemize}}
	\newcommand{\Ei}{\end{itemize}}
	
\newcommand{\wchil}{\wchi_l}
\newcommand{\chil}{\chi_l}
\newcommand{\etar}{\eta_\rho}

\newcommand{\ipair}{\chi_l\HL{\eta_\rho}\chi_l'}
\newcommand{\ipairl}{\chi_l\LL{\eta_\rho}\chi_l'}
\newcommand{\LL}[1]{[#1]^{\mathcal L}_\lambda}
\newcommand{\cL}{\mathcal L}
\newcommand{\cH}{\mathcal H}
\newcommand{\HL}[1]{[#1]^\cH_\lambda}

\author[Lee]{Sanghyuk Lee}
\author[Ryu]{Jaehyeon Ryu}
\address
{Department of Mathematical Sciences and RIM, Seoul National University, Seoul 08826, Republic of  Korea}
\email{shklee@snu.ac.kr}
\email{miro21670@snu.ac.kr}


\begin{document}

\title[Bochner-Riesz means]{Bochner-Riesz means for the Hermite  and \\ special Hermite expansions}

\maketitle

\begin{abstract}  We consider the  Bochner-Riesz means for the Hermite  and special Hermite expansions and  study  their $L^p$ boundedness with the  sharp summability index in a local setting.  
In two dimensions we establish the boundedness on the optimal range of $p$ and  extend  the previously known range in higher dimensions. 
Furthermore, we prove a new lower bound  on the $L^p$ summability index for  the Hermite  Bochner-Riesz means in $\mathbb R^d$,  $d\ge 2.$  This invalidates the conventional conjecture which was expected to be true. 
\end{abstract}

\section{Introduction}
Let $\mathcal H$ denote the Hermite operator  
\[-\Delta+|x|^2=
-\sum_{i=1}^d  \partial_{i}^2 + x_i^2, \quad x=(x_1, \cdots, x_d), \quad d \ge 1\] 
which is non-negative and selfadjoint with respect to the Lebesgue measure
on $\mathbb R^d$. The spectrum of the operator $\mathcal H$  is  given by  the set    
$2\N_0+d$. Here $\N_0$ denotes the set of nonnegative integers. 
For each $k\in \N_0$, the Hermite polynomial $H_k(t) $ on $\mathbb R$ is given by  Rodrigues' formula $H_k(t)=(-1)^k e^{t^2} {d^k\over d t^k} \big(e^{-t^2}\big)$, and  the $L^2$ normalized  Hermite functions
$h_k(t):=(2^k k !  \sqrt{\pi})^{-1/2} H_k(t) e^{-t^2/2}$, $k\in \mathbb N_0$ form an orthonormal basis
of $L^2(\mathbb R)$.    
In higher dimensions  
 the $d$-dimensional Hermite functions are given by the tensor products of $h_k$:
\[
\Phi_{\alpha}(x)=\prod_{i=1}^d h_{\alpha_i}(x_i), \quad \alpha=(\alpha_1, \cdots, \alpha_d)\ {\in\N_0^d.}
\] 
The Hermite operator and functions respectively represent the Hamiltonian and  quantum states of the particle 
for the quantum harmonic oscillator. The functions $\Phi_\alpha$ can also be interpreted as basis functions for the bosonic Fock space via the Bargmann transform. For a detailed discussion regarding the matters, we refer the reader to \cite{Fo89}.  The Hermite operator also appears in the representation theory of the Heisenberg group $\mathbb H^d$
(see for example \cite{Th98b}).

The set $\{\Phi_{\alpha}\}_{\alpha\in \mathbb N_0^d}$ forms a complete orthonormal
system in $L^2(\mathbb R^d)$ and the functions $\Phi_{\alpha}$ are eigenfunctions for
the Hermite operator with eigenvalue  $2|\alpha|+d$ where $|\alpha|=\sum_{i=1}^d \alpha_i$. 
Thus, for  every $f\in L^2(\mathbb R^d)$  we have  the Hermite expansion
\[ 
f=\sum_{\alpha\in\N_0^d} \inp{f}{\Phi_\alpha}\Phi_\alpha=\sum_{\lambda\in 2\N_0+d}\Pi_\lambda^\mathcal H f,
\]
where $\Pi_\lambda^\cH$ denotes the Hermite spectral projection given by
\[
\Pi_\lambda^{\cH} f=\sum_{2|\alpha|+d=\lambda}\langle f, \Phi_{\alpha}\rangle\Phi_{\alpha}.
\]

\subsection{Bochner-Riesz means for the Hermite expansion}
The Hermite expansion  is convergent  in $L^2(\R^d)$ space, but  when $d\ge 2$ the expansion  $\sum_{\lambda\le N} \Pi_\lambda^{\cH} f$  does not converge to $f$ as $N\to \infty$ in $L^p(\R^d)$ unless $p= 2$. This can be shown making use of the transplantation theorem due to 
Kenig-Stanton-Tomas \cite{KST82} and  Fefferman's counterexample for $L^p$ boundedness of the ball multiplier  \cite{F71}  (also see \cite[Theorem 3.1.2]{Th93}). Thus we are naturally led to  consider the  Bochner-Riesz mean:   
\[
S_\lambda^\delta(\mathcal H) f(x) := 
\Big(1-\frac{\cH}{\lambda}\Big)_+^\delta f(x)
:=\sum_{\lambda'\in 2\N_0+d} \Big(1-\frac{\lambda'}{\lambda}\Big)_+^\delta  \Pi_{\lambda'}^\mathcal H f.
\]
The summability exponent $\delta$ mitigates the influence of  new summands $ \Pi_{\lambda'}^\mathcal H f$ which enter into the summation as $\lambda$ increases. 
So,  the operator $S_\lambda^\delta(\mathcal H)$  has more favorable behavior in perspective of $L^p$ summability  as $\delta$ becomes larger.  
The  classical Bochner-Riesz problem is to determine the optimal summability order $\delta$ for which $S_\lambda^\delta(-\Delta) f$ converges to $f$ in ${L^p}$ for a  given  $p\in [1,\infty]$. 
When $d=2$, the problem was settled  by Carleson-Sj\"olin {\cite{CS72}}. In higher dimensions progress has  been made, 
however  the problem is still left open. 
See  \cite{St93, TVV98, L04} and also see \cite{GHL, Wu} for most recent results and references therein. 

In this paper we are concerned with ${L^p}$ convergence of the Hermite Bochner-Riesz means, that is to say,  the problem of determining the optimal  $\delta$ for which $S_\lambda^\delta(\mathcal H)f$ converges to $f$ in $L^p$.
By the uniform boundedness principle, this problem is equivalent to that of characterizing the optimal $\delta$ for which the estimate
\begin{align}\label{conj:brhm}
    \|S_\lambda^\delta(\mathcal H)\|_{p}\le C
\end{align}
holds with a uniform constant $C$ where $\|T\|_{p}:=\sup_{\|f\|_p\le 1} \|Tf\|_p$.

When $d=1$, the problem is almost completely settled except some endpoint cases. Askey and Wainger \cite{AW65} proved that  \eqref{conj:brhm} holds with $\delta=0$ if and only if $4/3<p<4$.
When $p\le 4/3$ or $p\ge 	4$, combining this with the result due to Thangavelu \cite{Th89a}, one can show that  $S_\lambda^\delta(\mathcal H)$ is uniformly bounded on $L^p(\R^d)$   
if $\delta>\max\{\frac23|\frac1p-\frac12|-\frac16,0\}$. On the other hand,  Thangavelu \cite{Th89a} showed that \eqref{conj:brhm} fails to hold if $\delta<\max\{\frac23|\frac1p-\frac12|-\frac16,0\}$.
However,  it looks that the estimate \eqref{conj:brhm} (or its weaker variants) with $\delta=\frac23|\frac1p-\frac12|-\frac16$ still remains open when  $p< 4/3$ or $p>	4$.

In higher dimensions,  unlike one dimension,  only partial results are known.
By the transplantation theorem due to Kenig, Stanton, and Tomas \cite{KST82}, the bound \eqref{conj:brhm} implies 
that the classical Bochner-Riesz means $S_\lambda^\delta(-\Delta)$ is uniformly bounded on ${L^p(\R^d)}$ (see Proposition \ref{prop:necessary} and its proof). Thus, by the well known necessary condition for ${L^p}$ boundedness of $S_\lambda^\delta(-\Delta)$ (see for example \cite{Herz, F71}) we have
\Be
\label{del-br}
\delta> \delta(d,p):=\max\Big\{d\,\Big|\frac1p-\frac12\Big|-\frac12,0\Big\}, \quad p\neq 2
\Ee
 if the uniform bound \eqref{conj:brhm} holds.  This naturally leads to the following conjecture. 
 
 \begin{conj}\label{hbr}   Let $p\in [1, \infty]\setminus\{2\}$. The uniform estimate \eqref{conj:brhm} holds if and only if \eqref{del-br} holds. 
  \end{conj}
 
  Karadzhov \cite{Kar94}  verified  Conjecture \ref{hbr} for  $\max(p, p') \ge 2d/(d-2)$. His result was based on  
 the optimal $L^2$--$L^p$ spectral projection estimate 
 $$\|\Pi_\lambda^\mathcal H\|_{2\to p}\le C\lambda^{-\frac12+\frac d2(\frac12-\frac1p)}, \quad {2d}/(d-2)\le p\le \infty.$$
 The estimate  plays the role of $L^2$--$L^p$ restriction estimate for the sphere  in Stein's  argument  \cite{F71} which deduces the sharp ${L^p}$ bound on $S_\lambda^\delta(-\Delta)$
 from the $L^2$ restriction estimates. However, as shown by Koch and Tataru \cite{KT05},  the range of $p$ where  the  above estimate  is valid can not be extended any further. 
We refer the reader to  \cite{KT05, JLR1} and references therein  for more about the Hermite spectral projection operator.  This means the approach in \cite{Kar94} relying on the $L^2$--$L^p$ 
spectral estimate  is no longer viable when one tries to prove $L^p$ boundedness with $\delta$ satisfying \eqref{del-br} when $\max(p, p') \in (2, {2d}/(d-2))$.

\subsubsection*{Local ${L^p}$ estimate for $S_\lambda^\delta(\mathcal H)$}  Meanwhile,  the kernel of $S_\lambda^\delta(\mathcal H)$ is expressed  on  a critical region as an Airy type integral and such phenomenon does not occur in the case of the classical Bochner-Riesz operator  $S_\lambda^\delta(-\Delta)$. Taking this into account,  Thangavelu \cite{Th98} speculated that  Conjecture \ref{hbr} may fails \footnote{That is to say, the summability index for \eqref{conj:brhm} may be bigger than that for the classical Bochner-Riesz means $S_\lambda^\delta(-\Delta)$ when 
$d\ge 2$.}  when $\max(p, p') \in  (2d/(d-2), 2).$
Instead of  the global estimate \eqref{conj:brhm}  he considered a local variant of \eqref{conj:brhm}.
To be specific, let us consider the estimate 
\begin{align}\label{conj:localhm}
\|\chi_E S_\lambda^\delta(\mathcal H)\chi_F\|_p\le C
\end{align}
with a constant $C$ independent of $\lambda$  where  {$E,F$ are measurable subsets of $\R^d$.  It was shown  by Thangavelu  \cite{Th98} that  \eqref{conj:localhm}  holds with a compact set $E$ and $F=\R^d$ for $\frac{2(d+1)}{d-1}\le p\le\infty$ 
when  \eqref{del-br} holds.  The result is clearly sharp in that the estimates fail if $\delta<\delta(d,p)$ because of the aforementioned transplantation \cite{KST82}.  
 In analogy to Karadzhov's approach,  a form of local $L^2$--$L^{2(d+1)/(d-1)}$ estimate for $\Pi_\lambda^\mathcal H$ was utilized. 
 As was shown {in \cite{JLR1}, the local spectral projection estimate does  not extend for $p<{2(d+1)}/{(d-1)}$, so we can not expect any progress using 
 $L^2$--$L^p$ estimate for $\Pi_\lambda^\mathcal H$.

We shall show  that Conjecture \ref{hbr} is generally not true when $\max(p, p') \in  (2d/(d-2), 2).$ In fact, on a certain range of $p$ we obtain a new lower bound on the summability index $\delta$  (see Proposition \ref{prop:necessary})  for the uniform bound \eqref{conj:brhm}.  This invalidates Conjecture \ref{hbr}. Thus, in order to prove boundedness for $\delta>\delta(d,p)$ 
one has to consider a weaker alternative as was done \cite{Th98}.  It would be interesting to determine whether  \eqref{conj:brhm} holds up to the new lower bound  but for the present  the problem seems to be beyond reach. 
Instead, we first look into $L^p$ convergence of $S_\lambda^\delta(\mathcal H)$ 
in a local setting  
to make progress on the current state regarding $L^p$ boundedness of the Hermite Bochner-Riesz means.

As far as the authors are aware,  concerning  on $L^p$ {boundedness} of 
 the Hermite Bochner-Riesz means  no further progress has been made  beyond Thangavelu's result (\cite{Th98}) until now.  
  In this paper, we extend the range of $p$ for which \eqref{conj:localhm} holds under a suitable condition on $E$ and $F$. 
Even if  \eqref{conj:localhm} is a weaker variant of the global estimate, the local estimate  is still strong enough to imply the 
sharp $L^p$  bound on $S_\lambda^\delta(-\Delta)$. 
 More precisely,  if  the estimate \eqref{conj:localhm}    holds  with $E,F=  B(0,\ep)$ for any $\ep>0$, from the transplantation theorem (\cite{KST82}) we see that  the Bochner-Riesz operator $(\lambda+\Delta)_+^\delta$ is uniformly bounded on $L^p$.

 To state  our first result, we introduce some notations.
Let us set
\begin{align*}
p_0(d)=\begin{cases} \ 2\cdot\dfrac{3d+2}{3d-2} & \text{if } d\equiv 0 \,(\text{mod } 2), \\[2ex]
\ 2\cdot\dfrac{3d+1}{3d-3} & \text{if } d\equiv 1 \,(\text{mod } 2),
\end{cases}
\end{align*}
and 
\[
\mathcal D(x,y):=1+(\inp xy)^2-|x|^2-|y|^2,  \quad  (x,y)\in \mathbb R^d\times \mathbb R^d.
\]
The following is our first result.

\begin{thm}\label{thm:hmrsz}
 Suppose that $E,F\subset\R^d$ are compact sets such that 
$E\times F\subset  \mathfrak D(c_0):=\{ (x,y)\in \mathbb R^d: |x|,\, |y|\le 1-c_0, \, \mathcal D(x,y)>c_0^2\}$ for some $0<c_0<1$. 
Then there is a constant {$C$ independent of $\lambda$} such that
\Be
\label{hbr-lambda}
\|\chi_{E_\lambda}S_\lambda^\delta(\mathcal H)\chi_{F_\lambda}\|_{p} \le C,
\Ee
provided that $p>p_0(d)$ and $\delta>\delta(d,p)$ where $E_\lambda$, $F_\lambda$ denote the dilated set $\sqrt\lambda E$, $\sqrt\lambda F$, respectively. 
\end{thm}

When $d=2$,  Theorem \ref{thm:hmrsz} establishes the estimate \eqref{hbr-lambda} on the optimal range  $p$, that is to say, 
\eqref{hbr-lambda} holds if and only if  $\delta>\delta(d,p)$, $p\neq 2$. 
From Theorem \eqref{thm:hmrsz}, we have the following $L^p$ convergence result.

\begin{cor}\label{cor:hmconv}
Suppose that $p>p_0(d)$ and $\delta>\delta(d,p)$.
Then for any compact set $K\subset\R^d$ and compactly supported $f\in {L^p(\R^d)}$, we have 
\[
\lim_{\lambda\to\infty} \int_K |S_\lambda^\delta(\mathcal H) f(x) - f(x)|^p dx = 0.
\]
\end{cor}

\subsection{Bochner-Riesz means for the special Hermite expansion}
Now we  consider the  twisted Laplacian which is closely related to the {Hermite} operator. 
The twisted Laplacian $\mathcal L$ on $\C^d\cong \R^{2d}$ which is defined by 
\[
\mathcal L = - \sum_{j=1}^d \Big(\big(\frac\partial{\partial x_j}-\frac12iy_j\big)^2 +\big(\frac\partial{\partial{y_j}}+\frac12 ix_j\big)^2\Big),\quad x,y\in \R^d
\]
has the same discrete spectrum $2\N_0+d$ as  $\mathcal H$. The associated eigenfunctions are  the special Hermite functions $\Phi_{\alpha,\beta}$ which 
are given by the Fourier-Wigner transform of the Hermite functions.
Indeed, for any multi-index $\alpha,\beta\in\N_0^d$,  
\[
\Phi_{\alpha,\beta}(z):=(2\pi)^{-\frac d2}\int_{\R^d} e^{i\inp{x}{\xi}} \Phi_\alpha(\xi-\frac12 y)\Phi_\beta(\xi+\frac12 y)d\xi, \quad z=x+iy.
\]
Then it follows that $\mathcal L\Phi_{\alpha,\beta} = (2|\beta|+d)\Phi_{\alpha,\beta}$, thus $\Phi_{\alpha,\beta}$ is an eigenfunction of $\mathcal L$ with the eigenvalue $2|\beta|+d$ and  
the eigenspaces of $\cL$ are  infinite dimensional.
Additionally, $\Phi_{\alpha,\beta}$ satisfies $(-\Delta_z+\frac14|z|^2)\Phi_{\alpha,\beta} = (|\alpha|+|\beta|+d)\Phi_{\alpha,\beta}$, which means $\Phi_{\alpha,\beta}$ is an eigenfunction of the Hermite operator $-\Delta_z+\frac14|z|^2$. 
This is the reason that $\Phi_{\alpha,\beta}$ is called the special Hermite function.

For $\lambda\in 2\N_0+d$, by  $\Pi_\lambda^\mathcal L$   we denote the projection to the eigenspace of $\cL$ with the eigenvalue $\lambda$, i.e.,  
\Be 
\label{pi-L}
\Pi_\lambda^\mathcal L f(z) := \sum_{\beta:2|\beta|+d=\lambda} \sum_{\alpha\in\N_0^d} \inp{f}{\Phi_{\alpha,\beta}}\Phi_{\alpha,\beta} (z), \quad f\in L^2(\C^d).
\Ee 
Since $\{\Phi_{\alpha,\beta}\}_{\alpha,\beta}$ is a orthonormal  basis of $L^2(\C^d)$, so one can expand $f$ into the series of special Hermite functions. In fact, 
we have
\[
f(z) = \sum_{\lambda\in 2\N_0 +d} \Pi_\lambda^\mathcal L f(z) 
,  \quad f\in L^2(\C^d).
\]
As seen before in the case of Hermite expansion,  by the transplantation in  \cite{KST82} and Fefferman's counter example  \cite{F71}) 
this series fails to converge  in  $L^p$ unless $p=2$.  So, we need to  consider the Bochner-Riesz mean for the special Hermite expansion which is defined by
\[
S_\lambda^\delta(\cL) f(z) = \sum_{\mu\in 2\N_0+d} \Big(1-\frac{\mu}{\lambda}\Big)_+^\delta  \Pi_\lambda^\mathcal L f(z). 
\]
By the uniform boundedness principle, the $L^p$ convergence of $S_\lambda^\delta(\cL) f$ for all $f\in L^p$  is equivalent to  the uniform estimate
\begin{align}\label{conj:rsztl}
\|S_\lambda^\delta(\cL) f\|_p\le C\|f\|_p.
\end{align}  
The problem has been studied by many authors. By the transplantation theorem in \cite{KST82}  we see 
the estimate \eqref{conj:rsztl} implies $L^p$ boundedness of the classical Bochner-Riesz operator in  $\mathbb R^{2d}$. Thus 
 \eqref{conj:rsztl} holds only if 
 $\delta>\delta(2d,p)$, $p\neq 2$.  
 It seems to be plausible to conjecture that the uniform estimate  \eqref{conj:rsztl} holds if   $\delta>\delta(2d,p)$.  
In \cite{Th93}, Thangavelu verified the conjecture for $\frac{2d}{d-1}\le p\le\infty$.
Later, the range was extended to $\frac{2(3d+1)}{3d-2}<p\le\infty$ by  Ratnakumar, Rawat, and Thangavelu \cite{RRS97}.
Further progress was made by Thangavelu \cite{Th98} who showed 
the local estimate 
$  \|\chi_{E}S_\lambda^\delta(\cL)\|_{p}\le C $
holds provided that $\delta>\delta(2d,p)$ and $\frac{2(2d+1)}{2d-1}<p\le\infty$.
The corresponding global estimate was later established by Stempak and Zienkiewicz \cite{SZ98}, i.e., they showed that the estimate  \eqref{conj:rsztl} holds for $\frac{2(2d+1)}{2d-1}<p\le\infty$ and $\delta>\delta(2d,p)$. 
The common key ingredient of the previous results  is  the $L^2$--$L^p$ projection estimate of the form
\begin{align}\label{est:proj2p-L}
    \|\Pi_\lambda^\mathcal L\|_{2\to p}\le C\lambda^{\frac{d-1}{2}-\frac{d}{p}}
\end{align}
with $C$  independent of $\lambda$ which was combined with Stein's argument \cite{F71}. 
The projection estimate \eqref{est:proj2p-L} was shown by Stempak and Zienkiewicz \cite{SZ98}   for  $\frac{2(2d+1)}{2d-1}<p\le\infty$. 
Later, Koch and Ricci \cite{KR07}  proved that \eqref{est:proj2p-L} holds if and only if $\frac{2(2d+1)}{2d-1}\le p\le\infty$ (also, see \cite{JLR2} for $L^p$-$L^q$ estimates for $\Pi_\lambda^\mathcal L$). 
So, further improvement is no longer possible via the estimate \eqref{est:proj2p-L} when $\frac{2(2d+1)}{2d-1}> p$.

\subsubsection*{Local $L^p$ convergence of $S_\lambda^\delta(\cL)$}
Currently,  no result with the sharp summability exponent $\delta(2d,p)$  is known when $\frac{2(2d+1)}{2d-1}> p$.  Following the approach in \cite{Th98}, we consider a local variant  of 
\eqref{conj:rsztl} and 
prove new estimate with the sharp summability exponent  outside the aforementioned range of $p$.  The following is our result 
regarding the Bochner-Riesz means for the special Hermite expansion.

\begin{thm}\label{thm:tlrsz}
Let $0<c_0<2$, $p>p_0(2d)$, and $\delta>\delta(2d,p)$.
Suppose that $E,F\subset \R^{2d}$ are compact sets satisfying $|z-z'|\le 2-c_0$ for all $z\in E$, $z'\in F$. 
Then there exists a constant $C$ independent of $\lambda$ such that
\begin{align}\label{est:unif-tl}
    \|\chi_{E_\lambda}S_\lambda^\delta(\cL)\chi_{F_\lambda}\|_{p} \le C,
\end{align}
where $E_\lambda$, $F_\lambda$ denote the dilated set $\sqrt\lambda E$, $\sqrt\lambda F$, respectively. 
\end{thm}

\begin{cor}\label{cor:tlconv}
Let $p,\delta$ be as in Theorem \ref{thm:tlrsz}.  Then for any compact set $K\subset\R^{2d}$ and compactly supported $f\in {L^p(\C^d)}$ we have
\begin{align}\label{est:conv-tl}
\lim_{\lambda\to\infty}\int_B|S_\lambda^\delta(\cL)f(x) - f(x)|^p dx = 0.
\end{align}
\end{cor}

In the case of $d=1$, from this result we have a complete characterization of $(p,\delta)$ for which the local convergence \eqref{est:conv-tl} holds.
The assumption that $|z-z'|\le 2-c_0$ for all $z\in E$, $z'\in F$ is made for technical reason and we don't know whether the condition is necessary for the 
estimate \eqref{est:unif-tl} with the sharp summability index.  The assumption can be regarded as a counterpart of the assumption on $\mathcal D(x,y)$ in Theorem \ref{thm:hmrsz}.
{If $|z-z'|\ge 2+c_0$ for all $z\in E$, $z'\in F$, the kernel of  the operator $\chi_{E_\lambda}S_\lambda^\delta(\cL)\chi_{F_\lambda}$ rapidly decays (See \eqref{for:derivpl}).
Thus the assumption can be relaxed so that  $||z-z'|-2|\ge c_0$ for all $z\in E$, $z'\in F$.}
In order to prove the global bound we need to understand the behavior of the kernel of $S_\lambda^\delta(\cL)(z,z')$ when $|z-z'|$ is close to $2$.

\subsubsection*{\bf Our approach}
The proofs of  Theorem \ref{thm:tlrsz} and \ref{thm:hmrsz} follow a similar strategy which is inspired by the recent works  by Jeong and the authors  \cite{JLR1, JLR2} on the spectral projection operators 
$\Pi_\lambda^\mathcal H$ and $\Pi_\lambda^\mathcal L$.  First, we  obtain a explicit expression for the kernels    of  the operators $S^\delta_\lambda(\mathcal H)$,  $S^\delta_\lambda(\mathcal L)$
using the Schr\"odinger propagators  $e^{-it\cH}$,  $e^{-it\cL}$ (see \eqref{rpn:Hpropa}, \eqref{rpn:Lpropa}) of which kernel representation is well known. Secondly, combining the expression of kernel with the method of stationary phase we obtain asymptotic expansions of the kernels. This reduces the matter to obtaining the sharp estimate {for} the oscillatory integral operators which satisfy   the \textit{Carleson-Sj\"olin condition} \cite{CS72, H73, St86}. In two dimensions this allows us to obtain the optimal results. However, in higher dimensions  the Carleson-Sj\"olin condition alone is not enough, as was shown by Bourgain {\cite{B91}},  to give the sharp bound for $p<2(d+1)/(d-1)$ (or  $p<2(2d+1)/(2d-1)$).  However, an  additional ellipticity assumption on the 
second fundamental form of the phase  allows to get the bound on improved range.  This observation was first  made by one of the author \cite{L06}. 
Thirdly, we show  that the phases in the asymptotic expansion satisfy the ellipticity condition (see Lemma  \ref{lem:funda-L} and \ref{lem:phase-H}).  
We combine this with the results regarding the oscillatory integral operator  \cite{L06, GHL}. 
We refer the reader forward to Section \ref{cs-operator} for more regarding the oscillatory integral operators.

\subsubsection*{\bf Notation}  Throughout the rest of the paper, we identify $\C^d$ with $\R^{2d}$.  
\begin{enumerate}
[itemsep=-2.5mm,  leftmargin=.5cm, labelsep=0.3 cm, topsep=0pt]
\item[$\bullet$]  For given  $A, B>0$,  we write $B \lesssim A$ if there is a constant $C>0$ such that $B\le CA $. 
Here, if $C$ has to  be taken to be small enough, we use the notation $B\ll A$ to mean that $A$ is sufficiently larger than $B$.
Furthermore, $A\sim B$  denotes that $A\lesssim B$ and $B\lesssim A$.
\item[$\bullet$]    $B_d(x,r)=\{y\in \mathbb R^d: |y-x|<r\}$.
\item[$\bullet$]  For an operator $T$ we denote by $T(x,y)$ (or $T(z,z')$) the kernel of $T$.


\item[$\bullet$]   $\partial_x:= ( \partial_{x_1}, \dots, \partial_{x_d})^\intercal$, $\partial_x^\intercal : = ( \partial_{x_1}, \dots,  
\partial_{x_d})$,  so $\partial_x \partial_y^\intercal=(\partial_{x_i}\partial_{y_j})_{1\le i,j\le d}$.  In  particular, if $a(x) = (a_1(x),\dots,a_d(x))$ is a $\C^d$-valued differentiable function on $\R^d$, then we have 
$\partial_x^\intercal a(x)= (\partial_{x_j}a_i(x))_{1\le i,j\le d}$.

\item[$\bullet$]  By $\mathbf I_d$ we denote the $d\times d$  identity matrix.
If the value of $d$ is clear from the context, we simply denote $\mathbf I_d$ by $\bi$.

\item[$\bullet$] For $1\le i\le d$,   $\mathbf e_i$ denotes  the $i$-th standard basis in $\R^d$.

\item[$\bullet$]  For $S\subset \mathbb R^d$ and a constant $a>0$, we denote  $aS=\{ ax:  x\in S\}$.
\end{enumerate}

\section{Bochner-Riesz means for the Hermite expansion: \\ Proof of Theorem \ref{thm:hmrsz}}
\label{Sec2}

In this section we prove Theorem \ref{thm:hmrsz}. We begin by  obtaining  an explicit expression of  the kernel of the Bochner-Riesz means. For the 
purpose we make  use of Mehler's formula for the Schr\"odinger propagator.

\subsection{Decomposition  of $S_\lambda^\delta(\cH)$}

We start by considering the Hermite-Schr\"odinger propagator  {$e^{-it\cH}$} which is the solution to 
the {Cauchy} problem $(i\partial_t-\cH)u=0$ and $u(0,x)=f(x)$. 
The  propagator $e^{-it\cH}$ can be expressed by {the} spectral projection operators $\Pi_\lambda^\mathcal H$: 
\[
e^{-it\cH} = \sum_{\lambda\in 2\N_0+d} e^{-it\lambda}\Pi_\lambda^\mathcal H.
\]
On the other hand, by virtue of  Mehler's formula we  have an explicit  expression of the kernel of $e^{-it\cH}$. 
It is well known that 
\begin{align}\label{rpn:Hpropa}
e^{-it\mathcal H} f(x) = (2\pi i \sin 2t)^{-\frac d2}e^{i\pi d/4}\int e^{\frac{i}{2}((|x|^2+|y|^2)\cot{2t}-2\inp xy \csc{2t})} f(y) dy
\end{align}
for $f\in\mathcal S(\R^d)$. For example, see \cite{Th93} and also see \cite{Sj10} for a detailed discussion regarding derivation of \eqref{rpn:Hpropa}.

To study $L^p$ boundedness of the Bochner-Riesz means  we need to properly decompose  the operator  $S_\lambda^\delta(\cH)$. 
Let $\psi\in C^\infty_c([\frac14,1])$ be a smooth bump function such that $\sum_{j\in \Z} \psi(2^j t)=1$ for all $t>0$. Then, setting  $ \psi^\delta(t):=t^\delta \psi(t)$, 
we denote  
\begin{align*}
 \quad  \psi_j(t)&=\psi^\delta(2^{-j}t), \quad j\ge 1,
 \\[2pt]
 \psi_0(t)&= t^\delta_+ \sum_{j\ge0}\psi(2^j t),
\end{align*}
so that   $t^\delta=  \sum_{1\le  2^j \le 4\lambda}  2^{\delta j} \psi_j(t) $  for $0\le t\le \lambda$. For any bounded continuous function $m$ on $\mathbb R$, we define an operator $m(\cH)$ by setting 
\[  m(\cH)= \sum_{\lambda\in 2\N_0+d} m(\lambda) \Pi_{\lambda}^\cH. \]  Then, since $S_\lambda^\delta(\cH)=\lambda^{-\delta}\big(\lambda-\cH \big)_+^\delta$ and since the above summation is taken over 
the set   $2\N_0+d$,  we can write 
\begin{equation}
\label{eq:decompositionH} 
S_\lambda^\delta(\cH)=\lambda^{-\delta}\big(\lambda-\cH \big)_+^\delta
=\lambda^{-\delta}\sum_{1\le  2^j \le 4\lambda}  2^{\delta j}\psi_j(\lambda-\cH). 
\end{equation}
Though  $\psi_0$ is not smooth, we may assume  $\psi_0$ is smooth replacing it with a suitable smooth function because $\cH$ has the spectrum $2\N_0+d$. 
Thus, the proof of  Theorem \ref{thm:hmrsz} is reduced to obtaining  the sharp $L^p$ estimate for each operator $\psi_j(\lambda-A)$. 
That is to say, by \eqref{eq:decompositionH} Theorem \ref{thm:hmrsz} follows  if we show
\Be
\label{goal}
\|\chi_{E_\lambda}  \psi_j(\lambda-\cH) \chi_{F_\lambda}\|_{p} \lesssim   (\lambda 2^{-j})^{\delta(d,p)}, \quad 1\le  2^j \le 4\lambda
\Ee
for $p>p_0(d)$.

We now relate the operators $\psi_j(\lambda-\cH)$  to the propagator $e^{it \cH}$  via  Fourier inversion. 
In fact, for $\eta\in \mathcal S(\mathbb R)$ we have
\[ \widecheck{ \eta}  (\lambda-\cH) = \frac 1{2\pi} \int \eta(t) e^{i{t}(\lambda-\cH)} dt. \] 
Combining this with 
\eqref{rpn:Hpropa} and changing variables $t\to t/2$, we get 
\[ \widecheck{ \eta}  (\lambda-\cH) (x,y)= C_d \int \eta(t/2) (\sin t)^{-\frac d2} e^{i \phi_\lambda (t,x,y)} dt,
 \]
where 
\[
\phi_\lambda (t,x,y) := \frac{\lambda t}{2}+\frac{|x|^2+|y|^2}{2}\cot t-\inp xy \csc t.
\]
Instead of dealing with $\widecheck{ \eta}  (\lambda-\cH)$  it is more convenient to work with the rescaled operator. So, 
for $\zeta\in \mathcal S(\mathbb R)$ we define an operator $ [\zeta]^\cH_\lambda$ whose kernel is given by 
\begin{align}
\label{for:ih00}
 [\zeta]^\cH_\lambda(x,y)& :=  \widecheck{ \zeta(2\cdot)}  (\lambda-\cH)(\sqrt{\lambda}x,\sqrt{\lambda}y)
 \\
 &\,=  C_d \int \zeta(t) (\sin t)^{-\frac d2} 
e^{i \lambda\chp} dt 
\label{for:ih0}, 
\end{align}
where  $C_d$ is a constant depending on $d$ and 
\Be 
\label{def:phaseH} \chp := \frac t2 +\frac{(|x|^2+|y|^2)\cos t}{2\sin t}-\frac{\inp xy}{\sin t}.
\Ee
By scaling,  the estimate  \eqref{goal} is now equivalent to 
\Be
\label{goal1}
   \|\chi_{E}   \HL{   \widehat{\psi_j}(\cdot/2) } \chi_{F}\|_{p}\lesssim \lambda^{-\frac d2}   (\lambda 2^{-j})^{\delta(d,p)}, \quad 1\le  2^j \le 4\lambda   .
   \Ee
We occasionally use  the following lemma which is a simple consequence of  the formula \eqref{for:ih0} for the kernel.

\begin{lem}\label{lem:redH}
Let $\zeta\in C_c^\infty((0,\infty))$. Then,  for any measurable sets $E, F\subset \R^d$, 
\begin{align}
       \|\chi_{E}  [\zeta]^\cH_\lambda \,\chi_{F}\|_{p} &= 
    \|\chi_{E} \HL{\zeta(-\,\cdot)}\, \chi_{F}\|_{p}, 
    \label{hnorm1}\\
     \|\chi_{E}  [\zeta]^\cH_\lambda\,\chi_{F}\|_{p} & =
     \|\chi_{E} \HL{\zeta(\cdot+ \pi)}\,\chi_{(-F)}\|_{p}.
     \label{hnorm2}
\end{align}
\end{lem}

\begin{proof}
We observe that
$
\mathcal P_{\mathcal H}(-t,x,y) = -\phi(t,x,y)$ and $ \mathcal P_{\mathcal H}(t+\pi,x,y) = \frac{\pi}{2}+\mathcal P_{\mathcal H}(t,x,-y).
$
Using this, \eqref{for:ih0} and changing  of variables, we have
\begin{align}
\label{for:ihsym}
\HL{\zeta(-\cdot)} (x,y) &= C_1 \overline{\HL{\zeta}(x,y)}, 
\\
\label{for:ihsym2}
\HL{\zeta(\cdot+\pi)}(x,y) &= C_2  \HL{\zeta}(x,-y),
\end{align}
where  $C_1,C_2\in \C$ such that $|C_1|=|C_2|=1$.
Thus, \eqref{for:ihsym}, \eqref{for:ihsym2} give  \eqref{hnorm1}, \eqref{hnorm2}, respectively. 
\end{proof}

In the expression \eqref{for:ih0} there are  singularities at $t={k\pi}, $ $k\in \mathbb Z$, so we need to make further decomposition by breaking 
$\widehat{\psi_j}(\cdot/2)$. To do so,  we  set
\Be
\label{cutoff-nu}
\begin{aligned}
\nu^l (t) &= \sum_{\pm} \psi(\pm 2^lt)+\psi( 2^l(\pi\pm t)), \quad l\ge1, 
\\[-5pt]
 \nu^0(t) &= \chi_{(-\pi,\pi)}(t) - \sum_{l=1}^\infty \nu^l (t),
\end{aligned}
\Ee
so that  $\supp \nu^l\subset (-\pi, \pi)$ and {$\sum_{l=0}^\infty \nu^l (t-2k\pi)=1$ a.e. on $(-\pi, \pi)$.}  We also set 
\[
{\nu_j^{l, k}}(t) =  \widehat{\psi_j}(t/2) {\nu^l (t-2k\pi),}\]
so we have $\sum_{k=-\infty}^\infty \sum_{l=0}^\infty {\nu^l (t-2k\pi)=1}$ a.e. Thus  it  follows 
\Be\label{eq:decomp}\HL{   \widehat{\psi_j}(\cdot/2) } =\sum_{k=-\infty}^\infty \sum_{l=0}^\infty\HL{\nu_j^{l, k}}.\Ee
Now the proof  of \eqref{goal1} (and Theorem \ref{thm:hmrsz}) is  essentially reduced to showing the following.

\begin{thm}\label{thm:keyestH}
Let $0<\rho<\pi-2^{-5}$ and $\lambda\in2\N_0+d$. Let $\eta_\rho$ be a smooth function  such that 
$\supp \eta_\rho\subset [2^{-2}\rho,\rho]$ and $|\eta_\rho^{(n)}(t)|\le C_n \rho^{-n}$,  $n\in\N_0$. Suppose  that the sets $E,F$ satisfy the condition in Theorem \ref{thm:hmrsz}, i.e., $E\times F\subset \mathfrak D(c_0)$. 
{Then, for $p>p_0(d)$ we have}
\begin{align}\label{est:keyH}
    \|\chi_{E}\HL{\eta_\rho}\chi_{F}\|_{p}\lesssim
    \begin{cases}
  \    \lambda^{-\frac d2} \rho, &  \rho\le\lambda^{-1},
     \\
 \   \lambda^{-\frac d2}  \lambda^{\delta(d,p)}\rho^{\delta(d,p)+1}, &  \rho> \lambda^{-1}. 
    \end{cases}
    \end{align}
\end{thm}

Once we have  Theorem \ref{thm:keyestH}, the proof of \eqref{goal1} is rather straightforward.  
Assuming Theorem \ref{thm:keyestH} for the moment, we prove \eqref{goal1}.

\begin{proof}[Proof of \eqref{goal1}]   
Recalling \eqref{eq:decomp} we need to obtain bounds on  $\|\chi_{E} \HL{{\nu_j^{l, k}}} \chi_{F}\|_{p}$. 
By \eqref{hnorm2}  we need only to consider, instead of $\nu_j^{l, k}$,  the multipliers operators given  by 
$ \widehat{\psi_j}(k\pi+t/2) \nu^l (t)$. 
Then by \eqref{hnorm2}  and \eqref{hnorm1}  the matter reduces 
to dealing with the multipliers given by 
\[  \nu_\pm^{l,k}:=\widehat{\psi_j}(k\pi\pm t/2)\psi(2^lt), \quad  
 \nu_{\pm\pi}^{l,k}:= \widehat{\psi_j}(k\pi\pm (t-\pi)/2)\psi(2^lt), \]
which are supported in $[2^{-2-l}, 2^{-l}]$. 
One can easily see that  the estimates 
\begin{align*}
\bigg|\frac{d^n}{dt^n}\nu_\theta^{l,k}(t)\bigg|\le C 2^{j}\max(2^{nj},2^{nl})(1+2^{j-l})^{-N} (1+2^j|k|)^{-N}, 
\quad \theta\in \{\pm, \pm\pi\},
\end{align*}
hold for any $N\in\N$ with a constant $C=C(N)$.  Recalling $1\le  2^j \le 4\lambda$,  after normalizing the functions $\nu_\theta^{l,k}$,   we apply  Theorem \ref{thm:keyestH}  to  
$\|\chi_{E} \HL{{\nu_\theta^{l, k}}} \chi_{F}\|_{p}$. Thus, for any $N$ we get
\[ 
\|\chi_{E} \HL{\nu_j^{l, k}} \chi_{F}\|_{p} \lesssim 
\begin{cases}
\lambda^{-\frac d2}2^{-l} 2^{j}(1+2^j|k|)^{-N}, &   4\lambda< 2^l, \\
\lambda^{-\frac d2}\lambda^{\delta(d,p)}\,2^{-l(\delta(d,p)+1)}\,2^{j}(1+2^j|k|)^{-N}, &\, 2^j\le 2^l\le 4\lambda, \\
\lambda^{-\frac d2}\lambda^{\delta(d,p)}\,2^{-j\delta(d,p) }\,2^{-(j-l)N} (1+2^j|k|)^{-N},  & \ 2^l<2^j.
\end{cases}
\]
Using \eqref{eq:decomp} and taking summation over $l, k$ give \eqref{goal1}. In fact, taking $N$ large enough, it is sufficient to consider the case $k=0$. 
\end{proof}

\newcommand{\cD}{\mathcal D}
\newcommand{\cP}{\mathcal P}

\subsection{Further decomposition and reduction}

From now on we extensively utilize  the expression  \eqref{for:ih0}. To obtain estimates from the oscillatory integral we need to  collect 
some properties of the phase function {$\mathcal P_{\mathcal H}$.}

Using \eqref{def:phaseH},  a computation 
gives \begin{align}
\label{eq:derivp-H0} 
\partial_t\mathcal P_{\mathcal H}(t,x,y) &= -\frac{\cos^2 t-2\inp xy \cos t+|x|^2+|y|^2-1}{2\sin^2 t}.
\end{align}
Since $\mathcal D(x,y)\ge c_0^2$ for $(x,y)\in \mathfrak D(c_0)$,  we can factor
$\cos^2 t-2\inp xy \cos t+|x|^2+|y|^2-1=(\cos t-\inp xy - \cD(x,y))(\cos t-\inp xy +\cD(x,y))$. 
One can easily see  $-1<\inp xy \pm \sqrt{ \cD(x,y)}  < 1$ for $(x,y)\in \mathfrak D(c_0)$ because  $|\inp xy|<1$ if $(x,y)\in \mathfrak D(c_0)$. 
Thus  we can define smooth functions  $S_c(x,y),$ $S_*(x,y)\in (0,\pi)$  by  setting 
\begin{align}
\label{coss}
\cos S_c(x,y) &= \inp xy + \sqrt{\mathcal D(x,y)},  \quad  (x,y)\in \mathfrak D(c_0),
\\
\label{coss1}
\cos S_*(x,y) &= \inp xy - \sqrt{\mathcal D(x,y)}, \quad  (x,y)\in \mathfrak D(c_0).
\end{align}
So, we may also write
\begin{align}
\partial_t\mathcal P_{\mathcal H}(t,x,y)=  -\frac{(\cos t-\cos S_c(x,y))(\cos t-\cos S_*(x,y))}{2\sin^2 t} \label{for:derivp-H}\,.
\end{align}
Also,  since $ \sin^2 S_c(x,y) = (1-\inp xy-\sqrt{\mathcal D(x,y)})(1+\inp xy+\sqrt{\mathcal D(x,y)})$ and $(1-\inp xy)^2-\cD(x,y)=|x-y|^2$, we have
\begin{align}\label{for:sinsc-H}
    \sin S_c(x,y) 
    =|x-y|\bigg(\frac{1-\inp xy+\sqrt{\mathcal D(x,y)}}{1+\inp xy+\sqrt{\mathcal D(x,y)}}\bigg)^{\frac12}.
\end{align}

Note that $1\pm\inp xy+\sqrt{\mathcal D(x,y)}\gtrsim  c_0$ for $(x,y)\in \mathfrak D(c_0)$.
Thus we have $S_c(x,y)\sim |x-y|$.  A similar computation shows $S_\ast(x,y)\sim |x+y|$.

 Since the integral in \eqref{for:ih0} has slower decay  near the critical points, the kernel $\HL\etar$ becomes singular on  the sets 
$\{x= y\}$, $\{x=-y\}$, as $S_c(x,y)$, $S_\ast(x,y)\to 0$ while $S_c(x,y)$, $S_\ast(x,y)\in \supp \etar$.
We note that 
\Be\label{cond:sepSc}
|S_c(x,y)-S_*(x,y)|\ge 2c_0
\Ee
for any $(x,y)\in E\times F$. 
This follows by the mean value theorem and the fact $\cos S_*(x,y)-\cos S_c(x,y) = 2\sqrt{\mathcal D(x,y)}\ge 2c_0$.

{As a first step of the proof of \eqref{est:keyH}, we decompose $E\times F$ into smaller balls to localize the values of $S_c(x,y)$, $S_*(x,y)$.
We choose a constant $c>0$ small enough  so that
\begin{align}\label{cond:locH}
\begin{aligned}
|\cos S_c(x_1,y_1) - \cos S_c(x_2,y_2)|\le c_0/10,
\\
|\cos S_*(x_1,y_1) - \cos S_*(x_2,y_2)|\le c_0/10,
\end{aligned}
\end{align}
for  $(x_1,y_1)$, $(x_2,y_2)\in \mathfrak D(c_0)$ if $|(x_1,y_1)-(x_2,y_2)|\le 4cc_0$.
Since $E$ and $F$ are compact, so there are  finite collections of balls  $\{B(x_j, cc_0)\}_{j=1}^N$, $\{B(y_{j'}, cc_0)\}_{j'=1}^{N'}$ which cover $E$, $F$, respectively. 
Let $\{\varphi_j\}_{1\le  j\le  N}$, $\{\varphi'\}_{1\le  j'\le  N}$   be  collections of smooth functions such that  $\sum_{j=1}^N \varphi_j=1$ on $E$ and $\supp \varphi_j\subset  B(x_j, 2cc_0)$ and $\sum_{j'=1}^{N'} \varphi'_{j'}=1$ on $F$ and $\supp \varphi'_{j'}\subset  B(y_{j'}, 2cc_0)$. 
In order to prove \eqref{est:keyH} it suffices to prove the corresponding estimates for $\varphi_j[\eta_\rho]_\lambda^{\mathcal H}\varphi'_{j'}$ with the same bound.

If $\rho\sim 1$, then the support of $\eta_\rho$ may contain both critical points $S_c(x,y)$ and $S_*(x,y)$ for some $(x,y)\in\supp{\varphi_j}\times \supp{\varphi'_{j'}}$.
However, if we use the decomposition in the above we can exclude the case $S_*(x,y)\in\supp{\eta_\rho}$ breaking $\eta_\rho$ into  smooth functions supported in small intervals. In fact,  combining \eqref{cond:locH} with the symmetric properties \eqref{hnorm1}, \eqref{hnorm2} and the separation condition \eqref{cond:sepSc}, we may assume that 
%
\[ \dist( S_*(x,y), \supp \etar)\ge c_0/2\]
for all $(x,y)\in\supp{\varphi_j}\times \supp{\varphi'_{j'}}$. 

To see this, we fix $j,j'$ and decompose $\eta_\rho$ into $O(\rho/c_0)$ many cutoff functions  $\eta_{\rho, k}$ which are supported in finitely overlapping intervals of length $c_0/2$.
Then, from the construction (\eqref{cond:locH}) we have either $\dist( S_*(x,y), \supp \eta_{\rho,k})\ge c_0/2$ for all $(x,y)\in B(x_j,2cc_0)\times B(y_{j'},2cc_0)$ or  $\dist( S_*(x,y), \supp \eta_{\rho,k})\le c_0$ for all $(x,y)\in B(x_j,2cc_0)\times B(y_{j'},2cc_0)$.
The former case is acceptable.
For the latter case, by \eqref{cond:sepSc} it follows that $\dist(S_c(x,y), \supp \eta_{\rho,k})\ge c_0/2$ for all $(x,y)\in B(x_j,2cc_0)\times B(y_{j'},2cc_0)$.
We note that   $S_c(x,-y)+S_*(x,y) = \pi$ for any $x,y$ since  $\cos S_*(x,y) = -\cos S_c(x,-y)$. Thus  $\dist(\eta_{\rho,k}(\pi-\cdot), S_\ast(x,-y))\ge c_0/2$ for all $(x,y)\in B(x_j,2cc_0)\times B(y_{j'},2cc_0)$.
From  \eqref{hnorm1} and \eqref{hnorm2} we note that 
\[  \|\chi_{B(x_j,2cc_0)}\HL{\eta_{\rho,j}}\chi_{B(y_{j'},2cc_0)}\|_{p}= \|\chi_{B(x_j,2cc_0)}\HL{\eta_{\rho,k,\pi}}\chi_{(-B(y_{j'},2cc_0))}\|_{p}, \]
 where $\eta_{\rho,k,\pi}(t):=\eta_{\rho,k}(\pi-t)$. 
Clearly, $B(x_j,2cc_0)\times (-B(y_{j'},2cc_0))\subset \mathfrak D(c_0)$ and we have $\dist(\eta_{\rho,k,\pi}, S_\ast(x,y))\ge c_0/2$ for all $(x,y)\in B(x_j,2cc_0)\times (-B(y_{j'},2cc_0))$ as desired.}

Now, the proof of  \eqref{est:keyH} reduces to showing,  for $p> p_0(d)$, 
\Be
\label{eq:est-etaH} 
 \| \varphi_j  \HL{\eta_\rho}  {\varphi'_{j'}} \|_p  \lesssim
\begin{cases}
\lambda^{-\frac d2}\rho, &  \rho\le\lambda^{-1}, \\
\lambda^{-\frac d2+\delta(d,p)}\rho^{\delta(d,p)+1}, &  \rho>\lambda^{-1}
\end{cases}
\Ee
\newcommand{\wchi}{\widetilde \chi}
%
for each $j, j'$. 
To show this we   break the kernel dyadically away from the diagonal $\{x=y\}$:
\Be
\label{decomp-l}  ( \varphi_j  \HL{\eta_\rho}  {\varphi'_{j'}})(x,y) = \sum_{l}   \mathcal T^l(x,y) :
= \sum_{l} \HL{\eta_\rho}(x,y)  \psi(2^l|x-y|) \varphi_j(x){\varphi'_{j'}}(y). \Ee
For each $l$ we further decompose  $\varphi_j$, ${\varphi'_{j'}}$ into smooth functions which are supported in finitely overlapping balls of radius $c2^{-l}$. 
Since $|x-y|\sim 2^{-l}$, by this  additional decomposition 
it now suffices to consider the operator 
\[ \chi_l\HL{\eta_\rho}\chi_l'\] 
while $\chil, \chil'$ satisfy the following: 
\begin{align}
\supp \chi_l\times \supp \chi_l'\subset \mathfrak D(c_0/2), & 
\label{eq:chis1}
\\
\supp \chi_l, \ \supp \chi_l'\subset B(x_l, 2^{2-l}), \ \text{ for some } &x_0\in\mathfrak D(c_0), 
 \label{eq:chis2}
 \\
 2^{-l-2}\le \dist(\supp \chi_l,  \supp \chi_l')\le 2^{-l}, &
 \label{eq:chis3}
\\
|\partial_x^\alpha  \chi_l|, \ |\partial_x^\alpha  \chi_l'|\le C_\alpha 2^{|\alpha|l}.&
\label{eq:chis4}
\end{align}
The following is  clear.

\begin{lem}\label{lem:Tl-sum} Let $\chil, \chil'$ be smooth functions satisfying \eqref{eq:chis1}--\eqref{eq:chis4}.  
Suppose that  $\|\chi_l\HL{\eta_\rho}\chi_l'\|_p\le B$ holds,  then 
we have $\|\mathcal T^l\|_p\le CB$ for a constant $C$. 
\end{lem}

 \subsection{Estimates for $\chil\HL{\eta_\rho}\chil'$} 
 Therefore, we only need to obtain estimate for $\chi_l\HL{\eta_\rho}\chi_l'$.   To do this, we consider the kernel of $\HL{\eta_\rho}$ which is given by  the integral in \eqref{for:ih0}.  
So, the decay property of the kernel of $\HL{\eta_\rho}$ largely depends on  whether the support of $\eta_\rho$ contains the critical point $S_c(x,y)$ or not.
As be seen below, it is easy to handle $\ipair$ when $\dist(S_c(x,y), \supp(\eta_\rho))\ge c>0$.  

From  \eqref{for:sinsc-H} and \eqref{eq:chis3} we have
\Be 
\label{Scl} S_c(x,y)\sim 2^{-l},\quad (x,y)\in \supp \chil\times \supp \chil'.
\Ee
 We also occasionally make use of the following elementary lemmas.

\begin{lem}\label{lem:EFpq}
Let $1\le p\le \infty$ and let $E,F$ be measurable subsets of $\R^d$.
Suppose that  the kernel $T(x,y)$ of $T$  satisfies  $|T(x,y)|\le B$ for $(x,y)\in E\times F$.
Then
$
\|\chi_E T\chi_F\|_{p}\le B|E|^{\frac1p}|F|^{\frac{1}{p'}}.
$
\end{lem}

\begin{lem}[{\cite[Lemma 2.8]{JLR1}}] \label{lem:vander} 
Let $\phi\in C^\infty(I)$ and $A$ be a smooth bump function on $\R$ supported in an interval $I$ of length $0<\rho<2^2$.
Suppose that $| \phi'(t)|\ge L$, $|\phi^{(n)}(t)|\le  C L\rho^{1-n}$, and $ |A^{(n)}(t)|\le C \rho^{-n}$ for any $n\in\N_0$ and $t\in I$.
Then, for $\lambda\ge 1$ and $N=0,1,2,\dots$, 
\[
\bigg|\int A(t)e^{i\lambda\phi(t)}dt\bigg|\le C_N\rho(1+\lambda \rho L)^{-N}.
\]
\end{lem}

Lemma \ref{lem:EFpq} follows from interpolation between the trivial estimates 
$\|\chi_E T\chi_F\|_{\infty}\le B|F|$  and $\|\chi_E T\chi_F\|_{1}\le B|E|$. Lemma \ref{lem:vander}   can be shown by routine integration by parts, also see 
 {\cite[Lemma 2.8]{JLR1}} for a proof based on a scaling argument.

In order to show \eqref{eq:est-etaH}  we  separately handle  the cases  $\rho\le\lambda^{-1}$ and   $\rho>\lambda^{-1}$. 
The former case is easier to show.

\begin{lem}
\label{smallrho} Let $1\le p\le \infty$. If $\rho\le\lambda^{-1}$, then 
\Be
\label{eq:smallrho}
\|\ipair\|_p
\lesssim 
\begin{cases} 
\lambda^{-N}\rho^{1-\frac d2+N}2^{(2N-d)l}, 
& \quad 2^{-l}\gg\lambda^{-\frac12}\rho^{\frac12}, 
\\
\rho^{1-\frac d2}2^{-dl}, &  \quad 2^{-l}\lesssim\lambda^{-\frac12}\rho^{\frac12}. 
\end{cases} 
\Ee
\end{lem} 

\begin{proof} The bound \eqref{eq:smallrho} follows from estimates for the kernel of $\ipair$, so we assume  $(x,y)\in \supp \chi_l\times \supp \chi_l'$ and $t\in \supp \eta_\rho$. 
We first deal with the bound  \eqref{eq:smallrho}
 in the case   $2^{-l}\gg\lambda^{-\frac12}\rho^{\frac12}$.

From \eqref{Scl} 
we have $1-\cos S_c(x,y) \sim 2^{-2l}$. Thus it follows   $|\cos t-\cos{S_c(x,y)}| \ge  |1-\cos S_c(x,y)|-|1-\cos t|\gtrsim 2^{-2l}$ because $2^{-l}\gg \rho$.
Using this and \eqref{for:derivp-H}, we get the lower bound
\begin{align}\label{for:1derivH}
    |\partial_t\mathcal P_{\mathcal H}(t,x,y)|\sim 2^{-2l}\rho^{-2}
\end{align}
since $t\sim \rho$. We now obtain bounds on $\partial_t^n \mathcal P_{\mathcal H}$, $n\ge 2$. From 
\eqref{for:derivp-H} we  note  
\begin{align}\label{for:nderiv} \!\!
 \partial^n_t\mathcal P_{\mathcal H}(t,x,y) = -\frac12  \!\!\!\!\sum_{n_1+n_2+n_3=n-1} \!\!\!\!\!\!\! \partial^{n_1}_t\! (\cos t - \cos S_c) \partial^{n_2}_t\! (\cos t - \cos S_*) \partial^{n_3}_t(\sin t)^{-2}.
\end{align}
From now on,  we occasionally denote   $S_c(x,y)$, $S_*(x,y)$  by $S_c$, $S_*$, respectively, concealing $x,y$ for simplicity.
The absolute value of the summand $\partial^{n_1}_t (\cos t - \cos S_c) \partial^{n_2}_t (\cos t - \cos S_*) \partial^{n_3}_t(\sin t)^{-2}$ is bounded by  $C\rho^{1-n}$ if $n_3\le n-2$, or $C2^{-2l} \rho^{-n-1}$ if $n_3=n-1$. 
Since $2^{-l}\gg \rho$, we thus obtain 
$ |\partial^n_t\mathcal P_{\mathcal H}(t,x,y)|\lesssim 2^{-2l}\rho^{-n-1}$. 
Also, by a computation it is clear that  
\begin{align} 
\label{aa-d}
&\quad |\partial_t^n(\eta_\rho(t)(\sin t)^{-\frac d2})|\lesssim \rho^{-\frac d2-n}, \quad n\in\N_0, 
\end{align}
Putting these estimates together and using Lemma \ref{lem:vander}, we get
\[
|(\ipair)(x,y)|\lesssim \lambda^{-N}\rho^{1-\frac d2+N}2^{2Nl}, \quad N\in \N.
\]
Then by \eqref{eq:chis2} and Lemma \ref{lem:EFpq} we get the desired estimate \eqref{eq:smallrho}  when $2^{-l}\gg\lambda^{-\frac12}\rho^{\frac12}$.

Now we turn to the case $2^{-l}\lesssim\lambda^{-\frac12}\rho^{\frac12}$. 
In this case, we shall make use of the following well known estimate (see  \cite[p.70]{Th93}):
\begin{align}\label{est:l1inf-H}
|\Pi_\lambda^\mathcal H(x,y)|\lesssim \lambda^{\frac d2-1}, \quad \lambda\in2\N_0+d,
\end{align}
for any $x,y\in\R^d$.  From \eqref{for:ih00}  we note that 
$\HL{\eta_\rho}(\lambda^{-1/2}x, \lambda^{-1/2}y) $ can be expressed as $ \sum_{\lambda'\in 2\N_0+d}\int \eta_\rho(2t) e^{it(\lambda-\lambda')} dt\,  \Pi_{\lambda'}(\mathcal H)(x,y)$. 
Thus we have 
\begin{align*}
\HL{\eta_\rho}(\lambda^{-1/2}x, \lambda^{-1/2}y)
    = 2^{-1} \sum_{\lambda'\in 2\N_0+d} \widecheck{\eta_\rho}((\lambda-\lambda')/2) \Pi_{\lambda'}^\mathcal H(x,y).
\end{align*}
Since  $|\widecheck{\eta_\rho}(\tau)|\lesssim \rho(1+\rho|\tau|)^{-N}$ for any $N\in\N$,  we get 
\[
\big | \HL{\eta_\rho}(\lambda^{-1/2}x, \lambda^{-1/2}y)\big|\lesssim \sum_{\lambda'\in 2\N_0+d}\rho(1+\rho|\lambda-\lambda'|)^{-N}{(\lambda')}^{\frac d2-1}\lesssim \rho^{1-\frac d2}
\]
for $x,y\in\R^d$. 
Thus,  $|\HL{\eta_\rho}(x,y)|\lesssim  \rho^{1-\frac d2}.$ Applying Lemma \ref{lem:EFpq} with \eqref{eq:chis2}, we get the second  estimate in \eqref{eq:smallrho}.  
\end{proof}

\textit{Estimates for $\chil\HL{\eta_\rho}\chil'$ when $\rho>\lambda^{-1}$.}
We distinguish the three cases 
   \[ 2^{-l}\ll \rho,\quad  2^{-l}\gg \rho,\quad 2^{-l}\sim \rho.\]
Unlike the first and second cases, the support of $\eta_\rho$ may contain $S_c(x,y)$ in the third case.
So handling the case is the main part of the proof (see Proposition \ref{medrho}).  We first deals with the two easier cases.

\begin{lem}
\label{largerho} Let $1\le p\le \infty$. If $\rho>\lambda^{-1}$, then 
\Be
\label{eq:largerho}
\|\ipair\|_p
\lesssim 
\begin{cases} 
\lambda^{-N} \rho^{1-\frac d2-N}2^{-dl}, & \quad 2^{-l}\ll \rho, 
\\
\lambda^{-N}\rho^{1-\frac d2+N}2^{(2N-d)l}, 
 &  \quad 2^{-l}\gg \rho. 
\end{cases} 
\Ee
\end{lem}

\begin{proof}
As before  we show the bound \eqref{eq:largerho} by obtaining estimates for the kernel of $\ipair$. Throughout the proof we assume  $(x,y)\in \supp \chi_l\times \supp \chi_l'$ and $t\in \supp \eta_\rho$. 

We first deal with the case   $\rho\gg 2^{-l}$. 
By \eqref{eq:chis3} and \eqref{for:sinsc-H}
we have $1-\cos S_c(x,y) \sim 2^{-2l}$, so  $|\cos t-\cos{S_c(x,y)}| \ge |1-\cos t|- |1-\cos S_c(x,y)|\gtrsim \rho^2$ because $\rho\gg 2^{-l}$.
Using this and \eqref{for:derivp-H} and \eqref{for:nderiv}, we get  $|\partial_t\mathcal P_{\mathcal H}(t,x,y)|\gtrsim 1$ and 
\begin{align*}
       &|\partial^n_t\mathcal P_{\mathcal H}(t,x,y)|\lesssim \rho^{1-n}, \quad
    n\in\N.
\end{align*}
Since $|\partial^n_t (\eta_\rho(t)(\sin t)^{-\frac d2})|\lesssim \rho^{-\frac d2-n}$, by Lemma \ref{lem:vander} we get 
$
|(\ipair)(x,y)|\lesssim \lambda^{-N} \rho^{1-\frac d2-N},$ $N\in\N_0. $
Applying Lemma \ref{lem:EFpq} yields the first case estimate in \eqref{eq:largerho}.

 The proof of the second case in  \eqref{eq:largerho}  similar to that of  the first case of \eqref{eq:smallrho}.  
 Note that  $|\cos t-\cos{S_c(x,y)}| \gtrsim 2^{-2l}$ because $\rho\ll 2^{-l}$.  Thus,  from \eqref{for:derivp-H} and \eqref{for:nderiv},  we have 
$|\partial_t\mathcal P_{\mathcal H}(t,x,y)|\gtrsim 2^{-2l}\rho^{-2}$ and $|\partial^n_t\mathcal P_{\mathcal H}(t,x,y)|\lesssim 2^{-2l}\rho^{-1-n},$ $n\in\N.$ The rest of proof is identical, so we omit the detail.
\end{proof}

The following is the key estimate which we need  to prove Theorem \ref{thm:keyestH}.

\begin{prop}
\label{medrho}  Let  $\rho>\lambda^{-1}$ and $\rho\sim 2^{-l}$. Then  we have 
\Be 
\label{eq:medrho}
\| \ipair\|_{p}\lesssim \lambda^{-\frac d2+\delta(d,p)}\rho^{1+\delta(d,p)}
\Ee
provided that $p_0(d)< p\le\infty$.
\end{prop}

Assuming Proposition \ref{medrho},  we show the estimate \eqref{eq:est-etaH}. We combine  Lemma \ref{smallrho}, Lemma \ref{largerho} and Proposition \ref{medrho}.

\begin{proof}[Proof of \eqref{eq:est-etaH}] 
Let us first consider the case $\rho\le \lambda^{-1}$.  By \eqref{decomp-l}  we see 
\Be \label{eq:break} \| \varphi_j  \HL{\eta_\rho}  \varphi_{j'}\|_p\le  \sum_{l}   \|  \mathcal T^l\|_p.\Ee 
Splitting the sum $ \sum_{l} = \sum_{2^{-l}\lesssim\lambda^{-\frac12}\rho^{\frac12}} + \sum_{2^{-l}\gg\lambda^{-\frac12}\rho^{\frac12}}$, 
we combine Lemma \ref{smallrho} with a large $N$ and Lemma \ref{lem:Tl-sum}. Thus,  we see that   $\| \varphi_j  \HL{\eta_\rho}  \varphi_{j'}\|_p$ 
is bounded by  a constant times 
\begin{align*}
               \sum_{2^{-l}\lesssim\lambda^{-\frac12}\rho^{\frac12}}  \rho^{1-\frac d2}2^{-dl}
+  
\sum_{2^{-l}\gg\lambda^{-\frac12}\rho^{\frac12}}   \lambda^{-N}\rho^{1-\frac d2+N}2^{(2N-d)l}. 
\end{align*}
Thus,  we get  $\| \varphi_j  \HL{\eta_\rho}  \varphi_{j'}\|_p  \lesssim  \ \lambda^{-\frac d2}\rho$ as desired.

We now consider the case $\rho> \lambda^{-1}$. Similarly,  using Lemma \ref{largerho} with a large $N$  and Proposition \ref{medrho}  together with 
Lemma \ref{lem:Tl-sum},   we have
\begin{align*}
 \big( \sum_{2^{-l}\ll \rho } +\sum_{2^{-l}\sim \rho } + \sum_{2^{-l}\gg\rho} \big) \|\mathcal T^l\|_p
                                \lesssim 
(\rho\lambda)^{-N}\rho^{1+\frac d2}+ 
\lambda^{-\frac d2+\delta(d,p)}\rho^{1+\delta(d,p)}
\end{align*}
for $p_0(d)<p\le\infty$. By  \eqref{eq:break} it follows $\| \varphi_j  \HL{\eta_\rho}  \varphi_{j'}\|_p   \lesssim \lambda^{-\frac d2+\delta(d,p)}\rho^{1+\delta(d,p)}$
 since  $\rho>\lambda^{-1}$. Therefore we get  \eqref{eq:est-etaH}.
\end{proof} 


To complete the proof of  Theorem  \ref{thm:keyestH} 
it now remains to prove Proposition \ref{medrho}.  
The rest of this section is devoted to the proof of  Proposition \ref{medrho}. 

\subsection{Asymptotic expansion of the kernel when $2^{-l}\sim \rho>\lambda^{-1}$}
As  discussed before,  in this case the support of $\eta_\rho$ may
contain the critical point $S_c(x,y)$.

\textit{Additional decomposition of  $\chi_l$, $\chi'_l$ and $\eta_\rho$.}
If we further break the cutoff functions $\chi_l$, $\chi'_l$ and $\eta_\rho$  into finitely many smooth functions, then in addition to \eqref{eq:chis1}--\eqref{eq:chis4}  we may assume that 
\begin{align}
\supp \chil\subset B(x_0,\epc\rho),  &\quad   \supp \chil'\subset B(y_0,\epc\rho),
\label{chis1}
\\
\label{chis2} S_c(x,y)\in S_c(x_0,y_0)+( -\epc \rho, \epc \rho), &\quad  (x,y)\in \supp \chil\times \supp \chil',
\\
\label{chis3} \supp \etar \subset  S_c(x_0,y_0&) +( -2\epc \rho, 2\epc \rho),
\end{align}
for  a small $\epsilon_0>0$ and $(x_0,y_0)\in \mathfrak D(c_0)$. 
Indeed, just breaking $\chi_l$, $\chi'_l$ into cutoff functions which are supported in finitely overlapping balls of radius $c\epc \rho$ with a small $c>0$, trivially we have \eqref{chis1} and \eqref{chis2}
and this gives rise to only $O(\epc^{-2d})$ many such functions.
However, the third condition is not completely trivial.  
This can be achieved by breaking $\chi_l\HL{\eta_\rho}\chi_l' $ into major and minor parts. 
In fact, let $\wt{\eta_\rho}$ be a smooth bump function such that  ${\supp \widetilde{\etar}}\subset S_c(x_0,y_0) +( -3\epc \rho, 3\epc \rho)$, 
$ \widetilde{\etar}(t)=1$ if $t\in S_c(x_0,y_0) +( -2\epc \rho, 2\epc \rho)$, and $ |\widetilde{\etar}^{(n)}|\le C \rho^{-n}$ for any $n\in\N_0$. 
Then we split 
\[
\ipair =\chi_l\HL{\eta_\rho\wt{\eta_\rho}}\chi_l'   + \chi_l\HL{\eta_\rho(1-\wt{\eta_\rho})}\chi_l' .\]
The  second term can be handled in the same manner as before. From \eqref{chis2} we note that 
  $|\cos t-\cos{S_c(x,y)}|\gtrsim \rho^2$ on the support of $ \eta_\rho(1-\wt{\eta_\rho})$. Thus, 
using this, \eqref{for:derivp-H} and \eqref{for:nderiv}, we have $|\partial_t\mathcal P_{\mathcal H}(t,x,y)|\gtrsim 1$ and 
$|\partial^n_t\mathcal P_{\mathcal H}(t,x,y)|\lesssim \rho^{1-n},$  $n\in\N.$ Following the argument in the proofs of Lemma \ref{smallrho} and 
\ref{largerho}  we get 
$ \|\chil\HL{\eta_\rho(1-\wt{\eta_\rho})}\chil' \|_p\lesssim (\lambda\rho)^{-N}\rho^{1+\frac d2}$  for $\rho> \lambda^{-1}$. 
Therefore we may disregard the contribution from $\chi_l\HL{\eta_\rho(1-\wt{\eta_\rho})}\chi_l'$ and  we may therefore assume 
\eqref{chis3}.

We  obtain an asymptotic development of the kernel $\HL{\eta_\rho}(x,y)$ via the method of stationary phase. 
We recall the formula \eqref{for:nderiv} with $n=2$. Then, using \eqref{coss} and \eqref{coss1} we have
\begin{align}
\begin{aligned}\label{for:2nd-H}
    \partial_t^2\mathcal P_{\mathcal H}(S_c(x,y),x,y) 
    = \frac{\cos S_*(x,y)-\cos S_c(x,y)}{2\sin S_c(x,y)} = \frac{\sqrt{\mathcal D(x,y)}}{\sin S_c(x,y)}.
\end{aligned}
\end{align}
Since $c_0\le\sqrt{\mathcal D(x,y)}\le 1$, from \eqref{eq:chis3} and \eqref{Scl} we have
\begin{align}\label{est:2lobd-H}
\partial_t^2\mathcal P_{\mathcal H}(S_c(x,y),x,y) \sim \rho^{-1}, \quad (x,y)\in \supp \chil\times\supp\chil'.
\end{align}
One can easily see 
\begin{align}\label{est:phampbd-H}
|\partial^n_t \partial_x^\alpha\partial_y^\beta  \chp|  \lesssim \rho^{1-n-|\alpha|-|\beta|},
\end{align}
where $(t, x,y)\in \supp \etar\times\supp \chil\times\supp\chil'$.   To show the estimate  \eqref{est:phampbd-H}, 
recalling \eqref{def:phaseH}, we only need to consider the case $|\alpha|+|\beta|\le 2$. When $|\alpha|+|\beta|=0$,   \eqref{est:phampbd-H} can be shown using   \eqref{for:nderiv} similarly  as before.  The other cases  can be handled  in a more straightforward manner. For example, note that 
$\partial_x   \chp  =(\cos t x-y)/\sin t$.  Thus $\partial_x   \chp=O(1) $ because $\cos t\,x-y=(\cos t-1)\,x+x-y=O(\rho)$ and  we also have  $\partial_t^n\partial_x   \chp  =O(\rho^{-n}).$  
The remaining cases can be handled similarly.

We make change of variables so that the associated phase and amplitude functions have uniformly  bounded derivatives in $\rho$. 
Let us set
\begin{align*}
   \sigma(t,x,y) &= S_c(\rho x+ x_0,\rho y+y_0)+\rho t, 
    \\
         \widetilde{\mathcal P}_h (t,x,y) &= \rho^{-1}\mathcal P_{\mathcal H}(\sigma(t,x,y),\rho x +x_0,\rho y+y_0),
    \\
    \wchil(x) &=  \chil(\rho x+x_0), \quad   \wchil'(y)=  \chil(\rho y+y_0), 
    \\
     a(t,x,y) &= \rho^{\frac d2}(\sin \sigma(t,x,y))^{-\frac d2}   \eta_\rho(\sigma(t,x,y))  \wchil(x)  \wchil'(y).
\end{align*}
As to be seen later,  all of these functions have uniformly  bounded derivatives and   the support of $a$ is contained in  $(-2\epc,2\epc)\times B(0,\epc)\times B(0,\epc)$. 
Let us set 
\[ I_\rho^\cH(x,y)=\rho^{1-\frac d2} \int a(t,x,y) e^{i\lambda\rho\widetilde{\mathcal P}_h(t,x,y)}dt.\]
By the change of variables $t\to \sigma(t,x,y)$, we  have
\Be
\label{integral}
(\chil \HL{\eta_\rho}\chil') (\rho x+x_0 ,\rho y+y_0 ) =C_d    I_\rho^\cH(x,y),
\Ee
where $C_d$ is a constant depending on $d$. 
We  obtain an asymptotic expansion of the integral $  I_\rho^\cH(x,y)$ using the method of stationary phase. 

To do so,  
we first note 
\Be
\label{phi-dH}
|\partial_t^n \partial_x^\alpha\partial_y^\beta  \sigma(t,x,y)|\le C \rho
\Ee
for $(t, x,y)\in \supp a\times \supp \wchil\times \supp \wchil'$. 
This can be shown using  the following.

\begin{lem}
\label{lem:dd}  If $(x,y)\in  \supp \chil\times \supp \chil',$ then 
\begin{align}
\label{scc-dH}
|\partial_x^\alpha\partial_y^\beta  S_c(x,y)|  \lesssim \rho^{1-|\alpha|-|\beta|}.
\end{align} 
\end{lem}

\begin{proof}   From \eqref{for:sinsc-H} we see  \eqref{scc-dH} is trivially true if $|\alpha|+|\beta|=0$. We may assume  $|\alpha|+|\beta|\ge 1$ and by symmetry we may also assume $|\alpha|\ge 1$. 
Since $\partial_x S_c(x,y)=-(\sin S_c(x,y))^{-1} \partial_x \cos S_c(x,y)$
we note that 
\[ 
 \partial_x^{\alpha} \partial_y^\beta ( \partial_x S_c(x,y))  
= - \sum_{\alpha_1+\alpha_2=\alpha,\,|\alpha_2|=1}\partial_{x}^{\alpha_1}\partial_y^\beta (\sin S_c(x,y))^{-1} \partial_x^{\alpha_2} \cos S_c(x,y)+  \mathcal E\]
where $\mathcal E$ is  a linear combination of the terms 
\[
\partial_{x,y}^{\kappa} (\sin S_c(x,y))^{-1}  \partial_{x,y}^{\gamma}  \cos S_c(x,y)
\]
with $|\gamma|\ge 2$ and  $|\kappa|=|\alpha|+|\beta|-|\gamma|$.  On the other hand,  for any multi-indices $\alpha', \beta'$,  
using \eqref{coss},  we get $| \partial_x^{\alpha'}\partial_y^{\beta'} \cos S_c(x,y)|\le C_{\alpha', \beta'}$ because $\cD(x,y) \ge c_0^2/4$ and from \eqref{for:sinsc-H}  it similarly follows that 
$\partial_x^{\alpha'}\partial_y^{\beta'}   (\sin S_c(x,y))^{-1}=O(|x-y|^{-|\alpha'|-|\beta'|-1}). $  Thus, we see $|\mathcal E|\lesssim |x-y|^{1-|\alpha|-|\beta|}$ and  $|\partial_{x}^{\alpha-1}\partial_y^\beta( (\sin S_c(x,y))^{-1}) \partial_x \cos S_c(x,y)|\lesssim    |x-y|^{1-|\alpha|-|\beta|}$.  Hence we get \eqref{scc-d}. 
\end{proof}

Using \eqref{phi-dH},  \eqref{aa-d},  and  \eqref{scc-dH}, one can easily see 
\begin{align}
 \label{a-d}
 |\partial_t^n \partial_x^\alpha\partial_y^\beta  a(t,x,y)|  \le C_{n,\alpha, \beta}
 \end{align}
 for $(t, x,y)\in \supp a\times \supp \wchil\times \supp \wchil'.$
Similarly, combining   \eqref{est:phampbd-H}  with \eqref{scc-dH}, we also have
 \begin{align}
 \label{phih-d}
 |\partial_t^n \partial_x^\alpha\partial_y^\beta \widetilde{\mathcal P}_h (t,x,y)|  \le C_{n,\alpha, \beta}
 \end{align}
 for $(t, x,y)\in \supp a\times \supp \wchil\times \supp \wchil'.$
From  \eqref{est:2lobd-H}  it follows that  
\begin{align}
    \label{phi-ddH}
    \partial_t^2\widetilde{\mathcal P}_h(t,x,y) \sim 1
\end{align}  
for $(t, x,y)\in \supp a\times \supp \wchil\times \supp \wchil'.$ 
Since  $\partial_t \widetilde{\mathcal P}_h (0,x,y)=0$, the map  $t\to \widetilde{\mathcal P}_h (t,x,y)$ has a nondegenerate critical point at $t=0$. 
Thus we may now apply the method of stationary phase. In fact,  applying  \cite[Theorem 7.7.5]{H90} together with  \eqref{a-d}, \eqref{phih-d}, and \eqref{phi-ddH}, 
we obtain the following.

\begin{lem} 
\label{asymptotic} For $N\in\N$, we have 
\begin{align}\label{for:asymih}
 I_\rho^\cH(x,y)
= &\rho^{\frac{2-d}{2}} \sum_{n=0}^{N-1} (\lambda\rho)^{-\frac12-n}A_n(x,y)e^{i\lambda\rho\widetilde{\mathcal P}_h(0,x,y)}+E_N(x,y),
\end{align}
where $\supp A_n, \supp E_N \subset B(0,\epc)\times B(0,\epc)$ and 
\begin{align*}
|\partial_x^\alpha\partial_y^\beta A_n(x,y)|\le C_{\alpha, \beta}, \quad   |E_N(x,y)|\le C_N\rho^{\frac{2-d}{2}} (\lambda \rho)^{-N} \end{align*}
with $C_{\alpha, \beta}$ and $C_N$ independent of $\lambda,\rho$.
\end{lem}

 In the expansion \eqref{for:asymih}, the error term $E_N(x,y)$ is negligible if we take $N$ large enough.
Indeed, from Lemma \ref{lem:EFpq} we have $\|\chi_l E_N\chi_l'\|_{p}\lesssim \rho^{\frac{2-d}{2}} (\lambda \rho)^{-N} \rho^d$. With a large $N$ it is clear that
$\rho^{\frac{2-d}{2}} (\lambda \rho)^{-N} \rho^d\lesssim    \lambda^{-\frac d2+\delta(d,p)}\rho^{1+\delta(d,p)}$.
Thus, to obtain the estimate \eqref{eq:medrho}  we are led to consider the operators with the oscillatory kernels
$A_n(x,y)e^{i\lambda\rho\widetilde{\mathcal P}_h(0,x,y)}$.  Boundedness of properties of such an operator are determined by 
the phase function $\widetilde{\mathcal P}_h(0,x,y)$. So, we need to take a close look at it.

Let us define 
\Be 
\label{Phih}\Phi_{\mathcal H}(x,y)=\mathcal P_{\mathcal H}(S_c(x,y),x,y).
\Ee
  From  \eqref{def:phaseH}, using  \eqref{est:qsin} and  \eqref{coss1}, we see
\begin{align*}
        \Phi_{\mathcal H}(x,y)& = \frac12\Big(S_c(x,y)+\frac{(|x|^2+|y|^2)\cos S_c(x,y)-2\inp xy}{\sin S_c(x,y)}\Big)
        \\
        &=\frac12(S_c(x,y)-\cos S_*(x,y)\sin S_c(x,y)).
        \end{align*}
Combining this with  \eqref{for:sinsc-H} one can easily  see  $\Phi_{\mathcal H}(x,y)-2^{-1}|x-y|=O(|x-y|^2)$ as $|x-y|\to 0$.  Thus, this suggest 
the phase function $\Phi_{\mathcal H}(x,y)$ 
can be viewed as a small perturbation of the function $|x-y|$ when $(x,y)$ is contained in a small ball with center $0$.
This is natural in view of the transplantation theorem due to Kenig, Stanton, and Tomas \cite{KST82} because the phase function $|x-y|$ arises as a counterpart of $\Phi_{\mathcal H}(x,y)$ for the classical Bochner-Riesz operator $(1+\Delta)_+^\delta$. 

\subsection{Carleson-Sj\"olin type operator}
\label{cs-operator}
Let $A\in C_c^\infty(\R^d\times\R^{d-1})$ and $\phi\in C^\infty(\supp A)$.
Let $T_\lambda[\phi,A]$ denote the operator defined by
\[
T_\lambda[\phi,A]f(x) = \int_{\R^{d-1}} e^{i\lambda\phi(x,\xi)}A(x,\xi)f(\xi)d\xi.
\]
We assume that the mixed Hessian $\partial_\xi\partial^\intercal_x\phi$ has the maximal rank, {\it i.e.,}
\begin{align}
\label{cs1}
    \tag{C1} \rank(\partial_\xi\partial^\intercal_x\phi(x,\xi)) = d-1, \quad  (x,\xi)\in \supp A
\end{align}
This means that the image of  $\xi\to \partial_x\phi(x_0,\xi)$ is a smooth immersed surface in $\R^d$. The condition guarantees that, for any $(x_0,\xi_0)\in \supp A$, there is a unique vector $\nu(x_0,y_0)\in \mathbb{S}^{d-1}$ modulo  $\pm$  such that
\[
\partial_\xi\inp{\partial_x\phi(x_0,\xi)}{\nu(x_0,\xi_0)}\big|_{\xi=\xi_0} = 0.
\]
We further assume that the parameterized surface $\xi\to \partial_x\phi(x_0,\xi)$ has nonvanishing Gaussian curvature. 
Equivalently,  the second fundamental form of the surface parameterized by $\xi\to \partial_x\phi(x_0,\xi)$ is not singular, {\it i.e.,}
\[
\label{cs2}
\tag{C2}  \rank  (\partial_\xi\partial_\xi^\intercal\inp{\partial_x\phi(x_0,\xi)}{\nu(x_0,\xi_0)}\big|_{\xi=\xi_0})= d-1,  \ \  (x_0,\xi_0)\in\supp A.
\]
The conditions \eqref{cs1} and \eqref{cs2} together  are called the {\it Carleson-Sj\"olin condition}, and 
if an oscillatory integral operator $T_\lambda[\phi,A]$ satisfies both  \eqref{cs1} and \eqref{cs2}, we say  $T_\lambda[\phi,A]$ is a {\it Carleson-Sj\"olin type operator}.
Concerning the {Carleson-Sj\"olin type operators}, the estimate of the form
\begin{align}\label{est:CSpq}
\|T_\lambda[\phi,A] f\|_{L^q(\R^d)}\lesssim \lambda^{-\frac dq} \|f\|_{L^p(\R^{d-1})}
\end{align}
has been studied by various authors  \cite{CS72, H73, St86, B91, L06, GHL}.
H\"ormander \cite{H73}  conjectured that \eqref{est:CSpq} holds for $\frac1q<\frac{d-1}{2d}$ and $\frac1q\le \frac{d-1}{(d+1)p'}$
\footnote{The range of $p,q$ is the best possible one. This follows from the well known necessity condition for the restriction estimate to the surfaces with nonzero curvature since the estimate \eqref{est:CSpq} implies the adjoint restriction estimate to such surfaces. }
 if  $T_\lambda[\phi,A]$ is a Carleson-Sj\"olin type operator.  
This was verified  by  H\"ormander when $d=2$. He also showed that this range of $p,q$ is optimal for \eqref{est:CSpq}.

In  higher dimensions,  Stein \cite{St86} obtained \eqref{est:CSpq} for $\frac{2(d+1)}{d-1}\le q\le\infty$ and $\frac1q\le \frac{d-1}{(d+1)p'}$ but 
 Bourgain \cite{B91} essentially disproved the H\"ormander's conjecture by constructing a phase function $\phi$ which satisfies the Carleson-Sj\"olin condition but
the  estimate \eqref{est:CSpq} does fails for any $q>\frac{2(d+1)}{d-1}$ if $d$ is odd.
Nevertheless, one of the authors \cite{L06}  proved that the range of $p,q$ for \eqref{est:CSpq} can be improved to $q>\frac{2(d+2)}{d}$ under the additional condition: 
\Be
\tag{C3}
\label{elliptic} 
\begin{aligned}
\text{ For $(x_0,\xi_0)\in \supp  A$, all nonzero eigenvalues of the matrix}
\\ 
 \partial_\xi\partial_\xi^\intercal\inp{\partial_x\phi(x_0,\xi)}{\nu(x_0,\xi_0)}|_{\xi=\xi_0}  \text{  have the same sign.}
 \end{aligned}
\Ee
It was also shown that the range is optimal when $d=3$. Recently, Guth, Hickman and Iliopoulou \cite{GHL} obtained the sharp result for \eqref{est:CSpq} with $p=q$ under the assumptions \eqref{cs1}\,--\,\eqref{elliptic} for $d\ge 4$.
The main ingredients of their result are multilinear estimates due to Bennett, Carbery and Tao \cite{BCT06} and the method of polynomial partitioning.  
For our purpose, we summarize the previously  known results when $p=q$ as follows (\cite{CS72, H73, L06, GHL}). 

\begin{thm}\label{thm:osc}
Suppose that $d\ge 2$ and $\phi$ satisfies the conditions \eqref{cs1}\,--\,\eqref{elliptic}. Then we have \eqref{est:CSpq} whenever $p=q>p_0(d)$.
\end{thm}

Furthermore, the range of $p$ is sharp up to the endpoint in that there exist counterexamples of $T_\lambda[\phi,A]$ whose phase function $\phi$ satisfies the conditions C1)\,--\,C3) such that \eqref{est:CSpq} with $p=q$ fails whenever $p<p_0(d)$ (see \cite{GHL}).

\subsection{The phase function $\Phi_{\mathcal H}$}  
In this section we investigate the phase function $\Phi_{\mathcal H}$ and its curvature condition. 
We begin with considering the vectors $\bfa(x,y),\bfb(x,y)$ which are given by
\begin{align}
\bfa(x,y)&:= \cos S_c(x,y) x - y,
\label{ba}
\\
\bfb(x,y)&:= x - \cos S_c(x,y) y.
\label{bb}
\end{align}
The following  show how the vectors $\bfa(x,y), \bfb(x,y)$, and  $x,y$ are related.

\begin{lem}\label{lem:abfor} Let $\cD(x,y)>0$. Then,  
we have the following: 
\begin{align}
  \label{a}   |\bfa(x,y)|^2 &= (1-|x|^2)\sin^2 S_c(x,y), 
    \\
   \label{b} |\bfb(x,y)|^2 & = (1-|y|^2)\sin^2 S_c(x,y),
\\
  \label{ab}  \inp{\bfa(x,y)}{\bfb(x,y)} &= \sqrt{\mathcal D(x,y)} \sin^2 S_c(x,y).
\end{align}
\end{lem}

In particular, \eqref{a}  shows  that the map $y\to \partial_x\Phi_{\mathcal H}(x,y)$ can not have the maximal rank $d$.  To see this,  
note that
\[
\partial_x\Phi_{\mathcal H}(x,y) = \partial_x\mathcal P_{\mathcal H}(S_c(x,y),x,y)
=\frac{\cos S_c(x,y) x-y}{\sin S_c(x,y)}.
\]
For the first equality we use $\partial_t \mathcal P_{\mathcal H}(S_c(x,y),x,y) = 0$. 
Thus, by \eqref{a} we have
$|\partial_x\Phi_{\mathcal H}(x,y)| = \sqrt{1-|x|^2}$. So, the image of $y\to \partial_x\Phi_{\mathcal H}(x,y)$ is contained in the sphere of radius $\sqrt{1-|x|^2}$.

\begin{proof}[Proof of Lemma \ref{lem:abfor}] From \eqref{eq:derivp-H0}  we note that  $ -2\inp xy \cos S_c +|y|^2=1-|x|^2-\cos^2 S_c$. Thus, it follows
\begin{align*}
    |\cos S_c\mspace{1mu}x -y|^2 &= |x|^2\cos^2 S_c -2\inp xy\cos S_c +|y|^2 =(1-|x|^2)\sin^2 S_c.
\end{align*}
This gives \eqref{a}. 
Since $\bfb(x,y) =- \bfa(y,x)$, \eqref{b} follows from \eqref{a}. It remains to show \eqref{ab}. 
Let us set $\mathcal Q(t,x,y)= \inp xy\cos^2 t-(|x|^2+|y|^2)\cos t+\inp xy$. Then we note 
\begin{align*}
    \inp{\bfa(x,y)}{\bfb(x,y)} 
    =-\mathcal Q(S_c,x,y).
\end{align*}
We also note that 
$
\partial_t^2\mathcal P_{\mathcal H}(t,x,y) = -\frac{\mathcal Q(t,x,y)}{\sin^3 t}. 
$
Combining this with  \eqref{for:2nd-H} we get 
\Be
\label{est:qsin}
    \mathcal Q(S_c(x,y),x,y) = - \sqrt{\mathcal D} \sin^2 S_c  
\Ee and hence \eqref{ab}.
\end{proof}

Let us set 
    \Be\label{M0}
    \mathbf M(x,y) = \partial_z\partial_z^\intercal  \inpb{\partial_x\Phi_{\mathcal H}(x,z)}{\frac{\bfa(x,y)}{|\bfa(x,y)|}}\Big|_{z=y}.   
    \Ee
What follows is crucial   in showing that  the phase $ \Phi_{\mathcal H}$ satisfies C1)\,--\,C3) after rescaling and freezing a suitable coordinate. 
    
\begin{lem}\label{lem:phase-H} 
Let $(x,y)\in \supp \chil\times \supp\chil'$. Then,  
$(i)$ the matrix $\partial_y\partial_x^\intercal \Phi_{\mathcal H}(x,y)$ has rank  $d-1$ and 
\Be
\label{phixy}
\partial_y\partial_x^\intercal \Phi_{\mathcal H}(x,y) \bfa(x,y)=0.
\Ee
   Additionally, $(ii)$ if $(x,y)\in \supp \chil\times \supp\chil'$ satisfies 
   \Be 
\label{rot} 
\frac{\bfb(x, y)}{|\bfb(x, y)|}=\mathbf e_d, 
\Ee
then the submatrix $\widetilde{\mathbf M}(x,y) := \{\mathbf M(x,y)_{i,j}\}_{1\le i,j\le d-1}$  
   of $\mathbf M(x,y)$ has negative eigenvalues $\lambda_1,\dots,\lambda_{d-1}$ such that 
    \[
    -\lambda_i\sim |x-y|^{-2},\quad 1\le i\le d-1.
    \]
\end{lem}

From now on, to simplify  the notation, we denote  by $\bfa,\bfb$ the vectors $\bfa(x,y)$, $\bfb(x,y)$, respectively, and
we also drop the variables $x,y$ from  $\mathcal D(x,y)$ as long as no ambiguity arises.
From \eqref{Phih} a computation shows 
\Be
\label{xyPhih} 
\begin{aligned}
    \partial_y\partial_x^\intercal\Phi_{\mathcal H}(x,y) &= \partial_y\partial_x^\intercal \mathcal P_{\mathcal H}(S_c,x,y) + \partial_y\partial_t\mathcal P_{\mathcal H}(S_c,x,y) \partial_x^\intercal S_c \\
    &\qquad +\partial_y S_c  \partial_x^\intercal \partial_t \mathcal P_{\mathcal H}(S_c,x,y)
    +(\partial_y S_c  \partial_x^\intercal S_c)\partial_t^2\mathcal P_{\mathcal H}(S_c,x,y). 
\end{aligned}
\Ee
Here we use $\partial_t\mathcal P_{\mathcal H}(S_c,x,y)=0$. Using \eqref{def:phaseH}  it is easy  to show   
\begin{align*}  
     \py\px\mathcal P_{\mathcal H}(S_c,x,y) &= -\frac{\bi}{\sin S_c}, \\
     \partial_x^\intercal\partial_t\mathcal P_{\mathcal H}(S_c,x,y) & = -\frac{\bfb^\intercal}{\sin^2 S_c}, \\ 
      \py\partial_t\mathcal P_{\mathcal H}(S_c,x,y) &= \frac{\bfa}{\sin^2 S_c}. 
     \end{align*}
     Since $\px S_c = -{\px(\cos S_c)}/{\sin S_c}$ and $\py S_c=-{\py(\cos S_c)}/{\sin S_c}$,  using \eqref{coss} we also have  
\begin{align}
   \px S_c=\frac{\bfb^\intercal}{\sin S_c\sqrt{\mathcal D}}, \quad   \py S_c  = -\frac{\bfa}{\sin S_c\sqrt{\mathcal D}}.  \label{syH}
\end{align}
Thus, putting these identities  and \eqref{for:2nd-H} into \eqref{xyPhih}, we have
\begin{align}
\label{Phixy}
    \py\px\Phi_{\mathcal H}(x,y) = -\frac{\bi_d}{\sin S_c}+\frac{\bfa \bfb^\intercal}{\sin^3 S_c\sqrt{\mathcal D}}
    =\frac{1}{\sin^3 S_c\sqrt{\mathcal D}}\big( \bfa \bfb^\intercal-\bfa^\intercal \bfb\,\bi_d\big).
\end{align}
The second equality follows from   \eqref{ab}. This  allows us to obtain an explicit expression for the matrix  $\mathbf M(x,y) .$

\begin{lem} Let $(x,y)\in \mathfrak D(c_0)$ and $\omega(x,y)=\sqrt{(1-|x|^2)\mathcal D(x,y)}\sin^4 S_c(x,y)$. Then we have 
\Be
\label{M}
\mathbf M(x,y) 
=   \frac{\bfa^\intercal \bfb\,\bi - \bfa\bfb^\intercal}{ \omega(x,y) \bfa^\intercal \bfb} \Big({\bfb \bfa^\intercal -\cos S_c \bfa\bfa^\intercal -\bfa^\intercal \bfb\,\bi}\Big).
\Ee
\end{lem}
\begin{proof}
From \eqref{M0} and \eqref{Phixy}, using \eqref{a},  we see 
\begin{align*}
\mathbf M(x,y) 
&= \frac{-1}{\sqrt{1-|x|^2}\sin S_c(x,y)}\partial_z^\intercal\bigg(\frac{\mathcal G(x,y,z)}{\sin^3 S_c(x,z)\sqrt{\mathcal D(x,z)}}\bigg)\bigg|_{z=y},
\end{align*}
where 
\[ \mathcal G(x,y,z)= \inp{\bfa(x,z)}{\bfb(x,z)} \bfa(x,y)-\inp{\bfa(x,y)}{\bfb(x,z)}\bfa(x,z). \] 
Since $\mathcal G(x,y,y) = 0$, it follows that 
\begin{align}\label{for:MG}
\mathbf M(x,y) = -\frac{\partial_z^\intercal \mathcal G(x,y,z)|_{z=y}}{\omega(x,y)}.
\end{align}
Via a straightforward calculation we have
\begin{align*}
&\partial_z^\intercal(\inp{\bfa(x,z)}{\bfb(x,z)}\bfa(x,y)) = \bfa(x,y) \bfa(x,z)^\intercal \partial_z^\intercal\bfb(x,z)+  \bfa(x,y) \bfb(x,z)^\intercal \partial_z^\intercal\bfa(x,z),\\
&\partial_z^\intercal(\inp{\bfa(x,y)}{\bfb(x,z)}\bfa(x,z)) =
\inp{\bfa(x,y)}{\bfb(x,z)}\partial_z^\intercal\bfa(x,z) + \bfa(x,z)\bfa(x,y)^\intercal\partial_z^\intercal\bfb(x,z).
\end{align*}
Thus we have
\begin{align*}
    \partial_z^\intercal \mathcal G(x,y,z)|_{z=y}
    =(\bfa\bfb^\intercal - \bfa^\intercal \bfb\,\bi)\, \partial_y^\intercal \bfa.
\end{align*}
Differentiating both sides of  the equations \eqref{ba} and   \eqref{coss}, one can easily  see  
$\partial_y^\intercal \bfa(x,y)= \mathcal D^{-\frac12} x \bfa^\intercal-\bi$.
Combining this with an identity $
x = \frac{\bfb-\cos S_c \bfa}{\sin^2 S_c}
$
gives 
\[
\partial_y^\intercal \bfa(x,y) = \frac{(\bfb-\cos S_c \bfa)\bfa^\intercal}{\sin^2 S_c\sqrt{\mathcal D}}-\bi
=\frac{\bfb \bfa^\intercal -\cos S_c \bfa\bfa^\intercal -\bfa^\intercal \bfb\,\bi}{\bfa^\intercal \bfb}.
\]
For the last inequality we use \eqref{ab}. 
Hence, combining the above identities with \eqref{for:MG} we obtain 
\eqref{M}.
\end{proof}

We are now ready to prove Lemma \ref{lem:phase-H}. 

\begin{proof}[Proof of Lemma \ref{lem:phase-H}]  
We begin by noting  from Lemma \ref{lem:abfor} that $|\mathbf a|,   |\mathbf b|,  |\mathbf a^\intercal \mathbf b|\gtrsim \rho$  since $(x,y)\in   
\supp \chil\times \supp\chil'\subset 
\mathfrak D(c_0/2)$. 
Using \eqref{Phixy} we have
\begin{align*}
    \py\px\Phi_{\mathcal H} v =
    \begin{cases}
   \  \qquad \quad 0, &   v= \bfa, 
   \\
   \ \dfrac{\inp \bfb v \bfa-  \inp {\bfa}{\bfb}  v}{\sin^3 S_c\sqrt{\mathcal D}}, &  v\perp \text{span}\{\bfa\}.
    \end{cases}
\end{align*}
This proves $(i)$  in Lemma \ref{lem:phase-H} because $\inp \bfb v \bfa-  \inp {\bfa}{\bfb}  v\neq 0$ by \eqref{ab}.

We now verify the next assertion $(ii)$. Let us set 
\[ \mathbf G(x,y)= (\bfa^\intercal \bfb\,\bi - \bfa\bfb^\intercal)({\bfb \bfa^\intercal -\cos S_c \bfa\bfa^\intercal -\bfa^\intercal \bfb\,\bi}). \] 
 Then we note that  $\bfb^\intercal  \mathbf G(x,y)=0$ and 
$
    \mathbf G(x,y)  \bfb = (\bfa^\intercal \bfb\,\bi - \bfa\bfb^\intercal) (-\cos S_c\,\bfa) = 0.
$
Since $\bfb(x,y) = |\bfb(x,y)|\mathbf{e}_d$, it follows that
\begin{align*}
    \mathbf G(x,y) = \begin{pmatrix}
   \widetilde{\mathbf G}(x,y)  & 0 \\
    0 & 0
    \end{pmatrix}
\end{align*}
where $\widetilde{\mathbf G}(x,y)$ is a $d-1\times d-1$ submatrix of  $\mathbf G(x,y)$.
Also, we observe that
\begin{align*}
    \mathbf G(x,y) v =
    \begin{cases}
    - {|\bfa|^2|\bfb|^2} v, & \text{if } v= \bfa-\dfrac{\bfa^\intercal \bfb}{|\bfb|^2}\bfb,\text{\footnotemark}
    \\ 
    -(\bfa^\intercal \bfb)^2\,v, & \text{if } v\perp \text{span}\{\bfa,\bfb\}.
    \end{cases}
\end{align*} 
A computation gives  $\mathbf G(x,y) \bfa = -|\bfa|^2|\bfb|^2(\bfa-\frac{\bfa^\intercal \bfb}{|\bfb|^2}\bfb)$. Thus, the first follows since $\mathbf G(x,y)  \bfb=0$ and 
the second is easy to see. 
Thus,  $\widetilde{\mathbf G}(x,y)$ has two eigenvalues
$
{|\bfa|^2|\bfb|^2}$,  $(\bfa^\intercal \bfb)^2$ 
of  multiplicity $1$, $d-2$, respectively.   Therefore, from \eqref{for:MG} we see that
$\widetilde {\mathbf M}(x,y)$ has  eigenvalues 
$\lambda_1, \lambda_2=\dots=\lambda_{d-1}, $ where
\begin{align*}
    -\lambda_1=\frac{|\bfa|^2|\bfb|^2}{\bfa^\intercal \bfb \sqrt{(1-|x|^2)\mathcal D}\sin^4 S_c}, \quad 
    -\lambda_2= \frac{\bfa^\intercal \bfb}{\sqrt{(1-|x|^2)\mathcal D}\sin^4 S_c} .
\end{align*}
Using \eqref{a}--\eqref{ab},  and  \eqref{for:sinsc-H}, we see 
\begin{align*}
|\lambda_1|
& = \frac{\sqrt{1-|x|^2}(1-|y|^2)}{\sin^2 S_c(x,y)\mathcal D(x,y)}\sim \frac{1}{|x-y|^2}, 
 \\
  |\lambda_2|&= \frac{1}{\sqrt{(1-|x|^2)}\sin^2 S_c(x,y)}\sim \frac{1}{|x-y|^2} 
 \end{align*}
because $(x,y)\in  \mathfrak D(c_0/2)$.
Hence $\widetilde{\mathbf M}(x,y)$ satisfies $(ii)$. This completes the proof.
\end{proof}

\subsection{Proof of Proposition \ref{medrho}}
We note that $\bfa(\mathbf Rx,\mathbf Ry) = \mathbf R \bfa(x,y)$, $\bfb(\mathbf Rx,\mathbf Ry) = \mathbf R \bfb(x,y)$ for any rotation matrix $\mathbf R$ (see \eqref{ba} and \eqref{bb}) 
because $S_c(x,y)$ is invariant under the simultaneous rotation, {\it i.e., $S_c(x,y) = S_c(\mathbf Rx,\mathbf Ry).$} It is clear  that  $\cP_h(t,x,y)=\cP_h(t,\mathbf R x,\mathbf Ry)$. Thus,  from  \eqref{for:ih0} we see $\HL{\etar}(\mathbf R x,\mathbf R y)=\HL{\etar}( x, y)$, so  we have 
\[\|\ipair\|_{p} = \|\chil(\mathbf R\cdot) \HL{\etar} \chil'(\mathbf R\cdot)\|_{p}\] 
for any rotation $\mathbf R$. 
In order to show \eqref{eq:medrho}, performing  simultaneous rotation for $x,y$, we may assume  
   \Be 
\label{rot2} 
\frac{\bfb(x_0, y_0)}{|\bfb(x_0, y_0)|}=\mathbf e_d, 
\Ee
 after replacing $\mathbf R x_0, \mathbf R y_0$ with  $ x_0, y_0$.

By \eqref{integral} and \eqref{for:asymih} the matter is reduced to obtaining estimates for  the operator  
\[ \mathcal T f(x)=\int A(x,y)e^{i\lambda\rho\widetilde{\mathcal P}_h(0,x,y)} f(y)  dy,\]
where $A\in C_c^\infty( B(0, \epc)\times B(0,\epc))$.   Since  $\| \ipair\|_{p}=\rho^d \|(\chil \HL{\eta_\rho}\chil') (\rho\, \cdot+x_0 ,\rho\, \cdot+y_0 )\|_p $, 
 taking a large enough $N$ in \eqref{for:asymih}  and using  Lemma \ref{asymptotic},  the estimate \eqref{eq:medrho} follows if we show
\Be
\label{final}
\|\mathcal T f\|_p\le C(\lambda\rho)^{-\frac dp} \|f\|_p
\Ee  
for $p_0(d)< p\le\infty$. We note that 
\Be
\label{ho}
  \widetilde{\mathcal P}_h(0,x,y)=  \rho^{-1} \Phi_{\mathcal H}( \rho x+x_0, \rho y+y_0).\Ee
Let us write $y=(\xi, y_d)\in \mathbb R^{d-1}\times \mathbb R$ and  set 
\[\phi_{y_d} (x,\xi)=  \widetilde{\mathcal P}_h(0,x,\xi, y_d).\] 
Then by \eqref{ho} and \eqref{phixy} it follows that 
\[
\partial_\zeta  \Big \langle \partial_x \phi_{y_d} (x,\zeta), \frac{\bfa_\rho (x,\xi, y_d)}{|\bfa_\rho(x,\xi, y_d)|}\Big\rangle\Big|_{\zeta=\xi}
=0,
\]
for ${(x,\xi, y_d)}\in B(0, \epc)\times B(0,\epc)$ where $ \bfa_\rho (x,y)= \bfa(\rho x+ x_0,\rho y+ y_0)$. 
From \eqref{M0} we also see 
\[
\partial_\xi \partial_\xi^\intercal \Big \langle  \partial_x \phi_{(y_0)_d}(x,\xi), \frac{\bfa(x_0,y_0)}{|\bfa(x_0,y_0)|}\Big\rangle\Big|_{(x,\xi)=(x_0,\xi_0)}= \rho^2
\widetilde {\mathbf M}(x_0,y_0).
\]
Recalling \eqref{rot2}, 
by Lemma \ref{lem:phase-H} we see   the matrix has negative eigenvalues $-\lambda_1, \dots, -\lambda_{d-1}$ while $\lambda_i\sim 1$, $i=1, \dots, d-1$.  
We now note that 
\[    |\partial_x^\alpha\partial_y^\beta  \widetilde{\mathcal P}_h(0,x,y) |  \lesssim 1,  
\quad  \Big|\partial_x^\alpha\partial_y^\beta \Big(\frac{\bfa_\rho (x,y)}{|\bfa_\rho(x,y)|}\Big)\Big |\lesssim 1.  \] 
The former  follows from  \eqref{phih-d} and Lemma \ref{lem:dd}. The latter can be shown  similarly. 
Therefore, taking small enough $\epc>0$,  by continuity  we see that the matrix 
\[
\partial_\zeta \partial_\zeta^\intercal \Big \langle  \partial_x \phi_{y_d}(x,\zeta), \frac{\bfa_\rho (x,\xi, y_d)}{|\bfa_\rho(x,\xi, y_d)|}\Big\rangle\Big|_{\zeta=\xi}
\]
has negative eigenvalues $-\lambda_1, \dots, -\lambda_{d-1}$ with $\lambda_i\sim 1$, $i=1, \dots, d-1$ for $(x,\xi, y_d)\in B(0, \epc)\times B(0,\epc)$. 
Therefore the phase function $\phi_{y_d}$ satisfies 
the elliptic Carleson-Sj\"olin condition, i.e., \eqref{cs1}--\eqref{elliptic}. We now set 
\[   \mathcal T_{y_d} g(x)= \int A(x,\xi, y_d)e^{i\lambda\rho\widetilde{\mathcal P}_h(0,x,\xi, y_d)} g(\xi) d\xi.   \] 
Thus we can use Theorem \ref{thm:osc} to obtain 
\[ \|   T_{y_d} g \|_p  \le C(\lambda\rho)^{-\frac dp}\|g\|_p.\] 
Since $\mathcal Tf=\int \mathcal T_{y_d} f(\cdot, y_d)  dy_d$,  by Minkowski's inequality followed by H\"older's inequality gives \eqref{final}. 
This completes the proof.  \qed

\section{Bochner-Riesz means for the special Hermite expansion: \\ Proof of Theorem \ref{thm:tlrsz}}

In this section we  consider the Bochner-Riesz means for the  special Hermite expansion. Basically, we follow the same strategy  for the Hermite expansion. 
We begin with noting that the  expression  \eqref{pi-L}  can be simplified by making use of the twisted convolution which is defined as follows:  
\[
f\times g(z)= \int_{\R^{2d}} f(z-w)g(w)e^{\frac{i}{2}\inp{z}{\bs z'}}dw,
\]
where $\bs$ is a skew-symmetric $2d\times 2d$ matrix given  by
\[
\bs = \begin{pmatrix}
0 & -\bi_d \\
\,\,\bi_d & \,0
\end{pmatrix}.
\]
By the argument using the Weyl transform, it can be shown that
$\Pi_\lambda^\mathcal L f = f\times \varsigma_k$, where $k=\frac{\lambda-d}{2}$ and $\varsigma_k(z) = (2\pi)^{-d}L_k^{d-1}(\frac12|z|^2)e^{-\frac14|z|^2}$, and $L_k^\alpha$ is the $k$-th Laguerre polynomial of type $\alpha$.  
Thus 
the Schr\"odinger propagator  $e^{-itL} f$ is given by  
$$
e^{-itL} f = \sum_{k=0}^\infty e^{-it(2k+d)} f\times \varphi_k.
$$
If $f\in\mathcal S(\R^{2d})$, by Lemma \ref{lem:rpnL} below we see that the sum in the right-hand side is uniformly and absolutely convergent.
Using the kernel formula for the heat operator $e^{-tL}$  by replacing $t$ with $it$ (see
[23, p.37]),\footnote{In fact, one can make it rigorous via analytic continuation using the operator  $e^{-zL}$ with $\Re z>0$.    } 
we have
\begin{align}\label{rpn:Lpropa}
e^{-itL}f (z)= \frac{1}{(-4\pi i \sin t)^d}\int_{\R^{2d}}e^{i(\frac{|z-z'|^2}{4}\cot t+\frac12\inp{z}{\bs z'})}f(z')dz',\quad f\in\mathcal S(\mathbb R^{2d}).
\end{align}
Here and henceforth, we regard the variables  $z,z'$ as real variables, i.e.,  $z,z' \in \mathbb R^{2d}$.

\begin{lem}\label{lem:rpnL}
Let $f\in \mathcal S(\R^{2d})$ and $N\in\N$. 
Then, there exists a constant $C=C(N,f,d)$ that satisfies the estimate
$\|f\times \varphi_k\|_\infty\le C |k|^{-N}$ for any $k\in\N_0$.
\end{lem}

\begin{proof}
We recall the known identity $\varphi_k(z)=(2\pi)^{-d/2}\sum_{|\alpha|=k}\Phi_{\alpha,\alpha}(z)$ (see \cite[p. 30]{Th93}).
From this, we have
\begin{align}\label{est:pkl2}
    \|\varphi_k\|_2=(2\pi)^{-\frac d2}(|\{|\alpha|=k\}|)^{\frac12}
    \sim\lambda^{\frac{d-1}{2}},
\end{align}
where $\lambda=2k+d$.  Clearly we have $f\times   \mathcal L^N  \varphi_k = \mathcal L^N (f\times \varphi_k)=  \lambda^N f\times \varphi_k $  and, on the other hand, 
by routine integration by parts   we also have  $f\times \mathcal Lg=(L f)\times g$ for a second order differential operator $L$ whose 
coefficients are $O(|z|^2)$. 
Thus we obtain 
$f\times \varphi_k=\lambda^{-N}(f\times \mathcal L^N\varphi_k)=\lambda^{-N}\big(L^N f\times\varphi_k\big)$. 
By H\"older's inequality and \eqref{est:pkl2} we get 
$|f\times \varphi_k(z)|\le\lambda^{-N}\|L^N f\|_2\|\varphi_k\|_2\le C_N \lambda^{-N+\frac{d-1}{2}}$ for any $N$ since 
$f\in \mathcal S(\C^d)$. 
\end{proof}

Using the same notation as in Section \ref{Sec2},  we decompose (c.f. \eqref{eq:decompositionH})
\begin{equation}
\label{eq:decomposition} 
S_\lambda^\delta(\cL)=\Big(1-\frac{\cL}{\lambda}\Big)_+^\delta
=\lambda^{-\delta}\sum_{1\le  2^j \le 4\lambda}  2^{\delta j}\psi_j(\lambda-\cL).  
\end{equation}
As before, Theorem \ref{thm:tlrsz} follows if we show 
\Be
\label{goal2}
\|\chi_{E_\lambda}  \psi_j(\lambda-\cL) \chi_{F_\lambda}\|_{p} \lesssim   (\lambda 2^{-j})^{\delta(2d,p)} 
\Ee
for $p>p_0(d)$.  As before,  for any bounded continuous function $m$ on $\mathbb R$, by $m(\cL)$  we denote the operator defined by 
$  m(\cL)= \sum_{\lambda\in 2\N_0+d} m(\lambda) \Pi_{\lambda}^\cL.$

For $\eta\in \mathcal S(\mathbb R)$ 
we define the scaled  operator $\LL{\eta}$ of which kernel is given by 
\Be
\label{for:il}
\LL{\eta} (z,z') =  \int \eta(t) (\sin t)^{-\frac d2} e^{i\lambda\cP_\cL(t,z,z')} dt,
\Ee
where
\begin{align}
\label{def:phaseL}
&\clp := t+\frac{|z-z'|^2\cos t}{4\sin t} +\frac{\inp{z}{\bs z'}}{2}.
\end{align}
The explicit kernel form \eqref{for:il} gives the following periodic and symmetric property  of $\LL{\eta}$. 

\begin{lem}\label{lem:redL} Let $\eta\in C_c^\infty((0,\infty))$. Let $\mathbf L$  denote a $2d\times 2d$ rotation matrix 
\[\mathbf L=\frac{1}{\sqrt2}\begin{pmatrix}
\bi_d &  -\bi_d \\
\bi_d & \ \ \bi_d
\end{pmatrix}.\]
Then we have
\begin{align}
    &  \|\chi_{E}\LL{\eta}  \chi_{F}\|_{p} = 
    \|\chi_{\mathbf L(E)}\LL{\eta(-\,\cdot)} \chi_{\mathbf L(F)}\|_{p}, 
    \label{lnorm1}
    \\
    &   \|\chi_{E}\LL{\eta}  \chi_{F}\|_{p} = 
    \|\chi_{E}\LL{\eta(\cdot+ k\pi)} \chi_{F}\|_{p},\quad  k\in\Z.
    \label{lnorm2}
\end{align}
\end{lem}

\newcommand{\rt}{\mathbb R^{2d}}

\begin{proof}[Proof of Lemma \ref{lem:redL}]
The first  identity \eqref{lnorm1}  follows from 
\[
\LL{\eta(-\, \cdot)} (z,z') = c\,\overline{\LL \eta (\mathbf L z,\mathbf L z')},
\quad z,z'\in \rt,
\]
where $c$ is a constant such that $|c|=1$.
This can be shown changing  variables $t\to -t$ and using the fact that $\bs = -\mathbf L^\intercal\bs \mathbf L$.
For \eqref{lnorm2},  the change of variable $t\to t+k\pi$ gives 
$
\LL{\eta(\cdot+ k\pi)}(z,z')  =  c \LL{\eta } (z,z')
$
with $|c|=1$.  Hence we get \eqref{lnorm2}. 
\end{proof}

\subsection{Reductions}
Similarly as before, we observe that 
\[\psi_j(\lambda-\cL)(x,y)=   \LL{\widehat  \psi_j }(\sqrt \lambda x, {\sqrt{\lambda} y}).\] Thus, by scaling and \eqref{rpn:Lpropa} we see  the estimate \eqref{goal2} is equivalent to 
\begin{equation}
\label{goal3} 
\|\chi_{E} \LL{\widehat  \psi_j }   \chi_{F}\|_{p} \lesssim  \lambda^{-d}   (\lambda 2^{-j})^{\delta(2d,p)}.
\end{equation}
To prove the estimate \eqref{goal3}, we proceed in the similar manner as in the Hermite case. We decompose $\LL{\widehat \psi_j }$ 
using the cutoff functions $\nu_j^{l,k}$  in Section 2 (see \eqref{cutoff-nu}). Thus we have $ \LL{   \widehat{\psi_j} } =\sum_{k=-\infty}^\infty \sum_{l=0}^\infty\LL{\psi_j^{l, k}} $.
If we combine this with  the symmetric and periodic properties in  Lemma \ref{lem:redL}, the proof of \eqref{goal3} is basically reduced to showing the following. (See  \emph{Proof of \eqref{goal1}.})

\begin{thm}\label{thm:keyest-L}
Let $ 0<\rho<\pi-\epc$ for some $\epc>0$. Let $\etar$ be a smooth function such that $\supp \etar \subset [2^{-2}\rho,\rho]$ and $|\etar^{(n)}(t)|\le C_n \rho^{-n}$ for $t>0$, $n\in\N_0$.
Then for any $\lambda\in2\N_0+d$ and the sets $E,F$ satisfying the condition in Theorem \ref{thm:tlrsz}, we have 
    \begin{align}\label{est:keyL}
    \|\chi_{E}\LL\etar \chi_{F}\|_{p}\lesssim
    \begin{cases}
   \  \lambda^{-d} \rho , &  \rho\le\lambda^{-1}, 
   \\
   \  \lambda^{-d} \lambda^{\delta(2d,p)}\rho^{\delta(2d,p)+1}, & \rho\ge \lambda^{-1}
    \end{cases}
    \end{align}
whenever $p>p_0(2d)$.
\end{thm}

In order to show Theorem \ref{thm:keyest-L}  we make a simple observation. 
From the assumption on $E$, $F$ in Theorem \ref{thm:tlrsz}, there exists $z_0\in\R^{2d}$ such that $E,F\subset B(z_0,4)$.
Using \eqref{for:il} it is easy to see  that
\[
[\eta]_\lambda^{\mathcal L}(z+z_0,z'+z_0) = [\eta]_\lambda^{\mathcal L}(z,z')e^{i\frac{\lambda}{2}(\inp{z_0}{\bs z'}+\inp{z}{\bs z_0})}.
\]
Clearly the oscillating term $e^{i\frac{\lambda}{2}(\inp{z_0}{\bs z'}+\inp{z}{\bs z_0})}$ does not affect the $L^p$ norm of $\chi_{E}\LL\etar \chi_{F}$, so we may assume that $E,F\subset B(0,4)$ after making a change of variable $(z,z')\to (z+z_0,z'+z_0)$.

We now recall  \eqref{for:il}  and  \eqref{def:phaseL}. To obtain estimates for the kernel of $\chi_E\LL\etar\chi_F$ we consider the  phase function $\clp$. A computation gives  
\begin{align}\label{for:derivpl}
\partial_t \clp = 1-\frac{|z-z'|^2}{4\sin^2 t}, \quad t\in (0,\pi). 
\end{align}
Since $|z-z'|\le 2-c_0$, $\clp$ 
has two critical points  $\mathfrak S_c(z,z')$ and $\mathfrak S_c^*(z,z')$  such that  $\mathfrak S_c(z,z')\in (0, \pi/2 -2\epc)$ and $\mathfrak S_c^*(z,z')\in  (\pi/2 +2\epc,  \pi)$ for a small $\epc>0$ and 
\begin{align}\label{def:scl}
\sin \mathfrak S_c(z,z')=\sin \mathfrak S_c^*(z,z') = \frac{|z-z'|}{2}.
\end{align}
So we have $\pi-\mathfrak S_c^*(z,z') = \mathfrak S_c(z,z')$.  

 As in the Hermite case which we have handled, we may further assume that 
 \Be
 \label{exclude}
\dist( \mathfrak S_c^*(z,z'),  \supp{\eta_\rho})\ge \epsilon_0,   \quad (z,z')\in E\times F.
\Ee 
To justify this, let us define $\eta_1,\eta_2,\eta_3\in C_c^\infty([0,\pi])$ such that $\supp \eta_1\subset [0,\pi/2-\epc)$, $\supp \eta_2\subset (\pi/2-3\epc/2,\pi/2+3\epc/2)$, $\supp \eta_3\subset (\pi/2+\epc,\pi]$ and $\eta_1+\eta_2+\eta_3 = 1$ on $[0,\pi]$.
Then we split $\LL\etar$ into the sum
\[
\LL{\etar\eta_1}+\LL{\etar\eta_2}+\LL{\etar\eta_3}.
\]
Since $\supp{(\etar\eta_1)}\subset[0,\pi/2-\epc)$, so the desired assumption is valid in this case.
Also, from \eqref{for:derivpl} it follows that $|\partial_t\mathcal P_{\mathcal L}(t,z,z')|\gtrsim \epc^3$ on $\supp{(\etar\eta_2)}$.
So, applying Lemma \ref{lem:vander} yields that $|(\chi_E\LL{\etar \eta_2 } \chi_F)(z,z')|\lesssim \lambda^{-N}$ for any $(z,z')\in E\times F$.
Hence we may disregard the contribution from  $\LL{\etar \eta_2 }.$
For the remaining term $\LL{\etar\eta_3}$, by utilizing \eqref{lnorm1}, \eqref{lnorm2} we note
\[
\|\chi_E\LL{\etar\eta_3}\chi_F\|_p = \|\chi_{\mathbf L (E)}\LL{(\etar\eta_3)(\pi-\cdot)}\chi_{\mathbf L (F)}\|_p
\]
where $\mathbf L$ is the rotation matrix in Lemma \ref{lem:redL}. 
Obviously $\mathbf L (E)$, $\mathbf L (F)$ satisfy the assumption of Theorem \ref{thm:keyest-L} and $\supp((\etar\eta_3)(\pi-\cdot))\subset[0,\pi/2-\epc)$.
Hence it is sufficient to show the estimate  \eqref{est:keyL} under the additional assumption \eqref{exclude}.

In order to show \eqref{est:keyL}, by additional decomposition and the standard argument as in  Section 2   it is sufficient to show 
\Be
\label{eq:est-eta} 
 \| \varphi \LL{\eta_\rho}  \varphi' \|_p  \lesssim
\begin{cases}
\lambda^{-d}\rho, &  \rho\le\lambda^{-1}, \\
\lambda^{-d+\delta(2d,p)}\rho^{\delta(2d,p)+1}, &  \rho>\lambda^{-1},
\end{cases}
\Ee
where $\varphi$, $\varphi'$ are smooth functions such that $\supp \varphi\times  \supp \varphi'\subset \{(z,z'): |z|, |z'|<4,  |z-z'|< 2- c_0 \} $. 
For the purpose  we decompose the integral kernel  in dyadic manner:
\Be
\label{decomp-lL}  ( \varphi \LL{\eta_\rho}  \varphi')(z,z') = \sum_{l}  \mathfrak T^l (z,z') :
= \sum_{l} \LL{\eta_\rho} (z,z') \psi(2^l|z-z'|) \varphi(z)\varphi'(z'). \Ee
The matter now  is reduced to showing the estimate for $\mathfrak T^l$. For the purpose it is sufficient to consider 
the operator which is given by the kernel
\[   (\chi_l\LL{\eta_\rho}\chi_l')(z,z')=\LL{\eta_\rho}(z,z')\chi_l(z)\chi_l'(z'), \] 
where $\chil, \chil'$ satisfy \eqref{eq:chis2}--\eqref{eq:chis4} and  $\tsupp \subset \{(z,z'): |z|, |z'|<  4,  |z-z'|< 2- c_0 \} $.  
Thus, by \eqref{def:scl}
we have 
\Be  
\label{sc-ll}
\mathfrak S_c(z,z') \sim 2^{-l}, \quad (z,z')\in \tsupp.
\Ee

\subsubsection*{Estimate for $\chi_l\LL{\eta_\rho}\chi_l'$ when $\rho\le \lambda^{-1}$ } 

We first handle the case $\rho\le\lambda^{-1}$.

\begin{lem}
\label{smallrhol} Let $1\le p\le \infty$. If $\rho\le\lambda^{-1}$, then  for any $N\ge 0$ we have
\Be
\label{eq:smallrhol}
\|\ipairl\|_p
\lesssim 
\begin{cases} 
\lambda^{-N}\rho^{1-d+N}2^{(2N-2d)l}, 
& \quad 2^{-l}\gg\lambda^{-\frac12}\rho^{\frac12}, 
\\
\rho^{1-d}2^{-2dl}, &  \quad 2^{-l}\lesssim\lambda^{-\frac12}\rho^{\frac12}. 
\end{cases} 
\Ee
\end{lem}

\begin{proof}
We assume $(t,z,z')\in \supp \etar\times \tsupp$. 
We first consider the case $2^{-l}\gg\lambda^{-\frac12}\rho^{\frac12}$. 
Since $2^{-l}\gg\rho$, by \eqref{sc-ll} it follows that  $0<\sin t\ll |z-z'|$.
From \eqref{for:derivpl}, we have
\begin{align}\label{est:lplb1}
|\partial_t\clp| \sim \frac{|z-z'|^2}{4\sin^2 t}\sim 2^{-2l}\rho^{-2}.
\end{align}
Also a simple calculation shows 
\begin{align}\label{est:lplb2}
&|\partial_t^n\clp|\lesssim \frac{|z-z'|^2}{(\sin t)^{n+1}}\lesssim 2^{-2l}\rho^{-n-1}, \quad n\ge 2, \\
&|\partial_t^n(\etar(t)(\sin t)^{-d})|\lesssim \rho^{-d-n}, \quad n\ge 1. \label{est:lpam}
\end{align}
Hence Lemma \ref{lem:vander} gives  $|(\ipairl)(z,z')|\lesssim_N \lambda^{-N}\rho^{1-d+N}2^{2Nl}$.
Using Lemma \ref{lem:EFpq}, we get $\| \ipairl\|_{p}\lesssim_N \lambda^{-N}\rho^{1-d+N}2^{2(N-d)l}$ for any   $1\le p\le \infty$.

We now handle the case $2^{-l}\lesssim \lambda^{-\frac12}\rho^{\frac12}.$
We first recall 
\begin{align}\label{est:l1inf-L}
\|\Pi_\lambda^\mathcal L\|_{1\to \infty}\lesssim \lambda^{d-1}, \quad \lambda\in2\N_0+d.
\end{align}
We refer the reader to \cite[Section 2.6]{Th93} for the proof of \eqref{est:l1inf-L}.
Then, by following the same argument in the proof of Lemma \ref{smallrho}, we have 
$$|\LL\etar(z,z')|\lesssim \sum_{\nu\in 2\N_0+d} \rho(1+\rho|\lambda-\nu|)^{-N} \nu^{d-1}
\lesssim \rho^{1-d}, \quad z,z'\in \R^{2d}.$$
This and Lemma \ref{lem:EFpq} give the desired estimate. 
\end{proof}

\subsubsection*{Estimate for $\chi_l\LL{\eta_\rho}\chi_l'$ when $\rho> \lambda^{-1}$ } 

Now we deal with the case {$\rho>\lambda^{-1}$}. 
We separately consider  three sub-cases $ 2^{-l}\ll \rho,$ $2^{-l}\gg \rho,$ and $2^{-l}\sim \rho.$ 
The first two cases are easy  to handle.

\begin{lem}
\label{largerhol} Let $1\le p\le \infty$. If $\rho>\lambda^{-1}$, then 
\Be
\label{eq:largerhol}
\|\ipairl\|_p
\lesssim 
\begin{cases} 
\lambda^{-N} \rho^{1-\frac d2-N}2^{-dl}, & \quad 2^{-l}\ll \rho, 
\\
\lambda^{-N}\rho^{1-\frac d2+N}2^{(2N-d)l}, 
 &  \quad 2^{-l}\gg \rho. 
\end{cases} 
\Ee
\end{lem}

\begin{proof}  We assume $(t,z,z')\in \supp \etar\times \tsupp$ and we first consider the case $2^{-l}\ll \rho$. The condition $2^{-l}\ll \rho$ implies 
\begin{align*}
|\partial_t \clp| = 1-\frac{|z-z'|^2}{4\sin^2 t} \gtrsim 1.
\end{align*}
Thus, the estimates \eqref{est:lplb2}, \eqref{est:lpam} remains valid.
So,  from Lemma \ref{lem:vander}  it follows that $|\ipairl(z,z')|\lesssim_N \lambda^{-N}\rho^{1-d-N}$ for any $N\in\N$.
This estimate combined with Lemma \ref{lem:EFpq} yields
$
    \|\ipairl\|_{p}\lesssim \lambda^{-N}\rho^{1-d-N}2^{-2dl}
$
for  $2\le p\le\infty$.

Now we consider the case $2^{-l}\gg \rho$.  In this case we have $0<\rho\ll \mathfrak S_c(z,z')<\pi/2$ because of \eqref{sc-ll}, so as in the previous case
the derivatives of $\mathcal P_{\cL}$ satisfy \eqref{est:lplb1} and \eqref{est:lplb2}.
Also the amplitude function $\eta_\rho(t)(\sin t)^{-d}$ satisfies the bound \eqref{est:lpam}.
With these estimates, we obtain the second case estimate in \eqref{eq:largerhol}
following the same argument as before.
\end{proof}

What follows  is the key estimate which we need for the proof of Theorem \ref{thm:keyest-L}.

\begin{prop}
\label{medrhol}  Let  $\rho>\lambda^{-1}$ and $\rho\sim 2^{-l}$. Then  we have 
\Be 
\label{eq:medrhol}
\| \ipairl\|_{p}\lesssim \lambda^{-d+\delta(2d,p)}\rho^{1+\delta(2d,p)}
\Ee
provided that $p_0(2d)\le p\le\infty$.
\end{prop}

Now, putting Lemma \ref{smallrhol}, Lemma \ref{largerhol}, and Proposition \ref{medrhol} together, we can prove Theorem \ref{thm:keyest-L} by establishing \eqref{eq:est-eta}.  Let us consider the case $\rho>\lambda^{-1}$ first.  By  \eqref{decomp-lL}
we have
\[ \| \varphi  \LL{\eta_\rho}  \varphi'\|_p\le  \sum_{l}   \| \mathfrak T^l\|_p.\]
 Using Lemma \ref{largerhol} and Proposition \ref{medrhol}, by a standard argument   (for example, Lemma \ref{lem:Tl-sum})  we have
\[   \| \mathfrak T^l\|_p 
 \lesssim 
\begin{cases} 
\lambda^{-N} \rho^{1-\frac d2-N}2^{-dl}, & \quad 2^{-l}\ll \rho, 
\\
 \lambda^{-d+\delta(2d,p)}\rho^{1+\delta(2d,p)}, & \quad 2^{-l}\sim \rho,
\\
\lambda^{-N}\rho^{1-\frac d2+N}2^{(2N-d)l}, 
 &  \quad 2^{-l}\gg \rho. 
\end{cases}\]  
Taking summation over $l$  as in the \emph{Proof of \eqref{eq:est-etaH}}, we get the desired bound \eqref{eq:est-eta}. The remaining case $\rho>\lambda^{-1}$
can be handled similarly using Lemma \ref{smallrhol}, so we omit the detail. 

\subsection{Asymptotic expansion of the kernel when $2^{-l}\sim \rho$}
Unlike the previous cases, the support of $\eta_\rho$ may
contain the critical point $\mathfrak S_c(z,z') $.

\textit{Additional decomposition of  $\chi_l$, $\chi'_l$ and $\eta_\rho$.}
Let $\epc>0$ be a small number. 
If we further break  $\chi_l$, $\chi'_l$ and $\eta_\rho$  into finitely many the cutoff functions, then in addition to \eqref{eq:chis2}--\eqref{eq:chis4}  we may assume that 
\begin{align}
\tsupp \subset \{(z,z'): |z|, \ & |z'|<2^2,  |z-z'|< 2- 2^{-1}c_0 \},
\label{chis0l}
\\
\supp \chil\subset B(z_0,\epc\rho),  &\quad   \supp \chil'\subset B(z_0',\epc\rho),
\label{chis1l}
\\
\label{chis2l} \mathfrak S_c(z,z')\in \mathfrak S_c(z_0,z_0')+( -\epc \rho, \epc \rho), &\quad  (z,z')\in \supp \chil\times \supp \chil',
\\
\label{chis3l} \supp \etar \subset  \mathfrak S_c(z_0,z_0'&) +( -2\epc \rho, 2\epc \rho)
\end{align}
for some $z_0$ and $z_0'$. 
The  conditions \eqref{chis0l}--\eqref{chis2l} can be achieved splitting $\chi_l$, $\chi'_l$ into cutoff functions which are supported in balls of radius $c\epc \rho$ with a small $c>0$. 
For the last \eqref{chis3l} we break 
\[\chi_l\LL{\eta_\rho}\chi_l' =\chi_l\LL{\eta_\rho\wt{\eta_\rho}}\chi_l'   + \chi_l\LL{\eta_\rho(1-\wt{\eta_\rho})}\chi_l'\]
where    
$\wt{\eta_\rho}$ is a smooth bump function such that  $\supp \etar\subset \mathfrak S_c(z_0,z'_0) +( -3\epc \rho, 3\epc \rho)$, 
$ \widetilde{\etar}(t)=1$ if $t\in  \mathfrak S_c(z_0,z'_0)  +( -2\epc \rho, 2\epc \rho)$, and $ |\widetilde{\etar}^{(n)}|\le C \rho^{-n}$ for any $n\in\N_0$. 
By the same argument as before we get 
\Be
\label{hohoho} \|\chil\HL{\eta_\rho(1-\wt{\eta_\rho})}\chil' \|_p\lesssim (\lambda\rho)^{-N}\rho^{1+d} \Ee
 for $\rho> \lambda^{-1}.$
 Therefore,  the contribution from $\chi_l\LL{\eta_\rho(1-\wt{\eta_\rho})}\chi_l'$ is negligible. 
 Thus, it is sufficient to deal with $\chi_l\LL{\eta_\rho\wt{\eta_\rho}}\chi_l'$, so  we may assume 
\eqref{chis3l}.  To show \eqref{hohoho}, we note that  $|\partial_t\mathcal P_{\mathcal H}(t,x,y)|\gtrsim 1$ and 
$|\partial^n_t\mathcal P_{\mathcal H}(t,x,y)|\lesssim \rho^{1-n},$  $n\in\N$ if $t\in \supp( \eta_\rho(1-\wt{\eta_\rho}))$ and $(z,z')\in \tsupp$.  
Then, using Lemma \ref{lem:vander} and  Lemma \ref{lem:EFpq}, we obtain the estimate \eqref{hohoho}.

Since  $\mathfrak S_c(z,z')\sim \rho$ for $(z,z')\in \tsupp$, we have
\Be
\label{2dl}
\partial_t^2 \mathcal P_{\!\mathcal L} (\mathfrak S_c(z,z'),z,z') = \frac{|z-z'|^2\cos \mathfrak S_c(z,z') }{2\sin^3 \mathfrak S_c(z,z') }\sim \rho^{-1}
\Ee
for $(z,z')\in \tsupp$.   Recalling \eqref{def:phaseL} it is easy to see 
\Be
\label{phi-dl}
\partial_t^n  \partial_z^\alpha\partial_{z'}^\beta \clp  =O(\rho^{1-n-|\alpha|-|\beta|}),  \quad n\ge 1
\Ee
for $(t, z,z')\in \supp \etar\times \tsupp$.  Since $\sin \mathfrak S_c(z,z')=|z-z'|/2$ and $\partial z  (\mathfrak S_c(z,z'))=  \partial z  (\sin \mathfrak S_c(z,z'))/ \cos \mathfrak S_c(z,z') $, we see that 
  $\partial z  (\mathfrak S_c(z,z'))= (4-|z-z'|^2))^{-1/2} \partial z (|z-z'|). $  
Thus a routine computation   gives 
\Be 
\partial z^\alpha\partial_{z'}^\beta  (\mathfrak S_c(z,z'))= O(|z-z'|^{1-|\alpha|-|\beta|})=O(\rho^{1-|\alpha|-|\beta|}).
\label{scc-d} 
\Ee 
 for $(z,z')\in \tsupp$ since $\rho\sim 2^{-l}$.

\newcommand{\sL}{{\!\mathcal L}}

The derivatives of the phase $\mathcal P_{\cL}$  and the amplitude $\etar$ are not uniformly bounded in $\rho$. 
Following the same approach in the previous section, we can get around this making change of  variables.  
Let us set
\begin{align}
   \phi (t,z,z') &= \mathfrak S_c(\rho z+ z_0,\rho z'+z'_0)+\rho t,  \nonumber
    \\
         \widetilde{\mathcal P}_{\!\cL}  (t,z,z') &= \rho^{-1}\mathcal P_\sL  (\phi (t,z,z'),\rho z +z_0,\rho z'+z'_0), \label{phill}
    \\
    \wchil(z) &=  \chil(\rho z+z_0),  \quad   \wchil'(z')=   \chil'(\rho z'+z'_0),  \nonumber
    \\
     a(t,z,z') &= \rho^{d}(\sin \phi(t,z,z'))^{-d} \eta_\rho(\phi(t,z,z')))  \wchil(z)  \wchil'(z'). \nonumber
\end{align}
Then we also set 
\[  I_\rho^\cL(x,y)= \rho^{1-d} \int a(t,z,z') e^{i\lambda\rho\widetilde{\mathcal P}_{\!\mathcal L}(t,z,z')}dt.\] 
By the change of variables $t\to \phi(t,z,z')$, we write
\Be
\label{integrall}
(\chil \LL{\eta_\rho}\chil') (\rho z +z_0,\rho z'+z'_0) =  
C_d I_\rho^\cL(x,y)
\Ee
where $C_d$ is a constant depending on $d$. 
Then, the support of $a$ is contained in the interval $(-2\epc,2\epc)\times B(0,\epc)\times B(0,\epc)$.
From  \eqref{scc-d} we first note 
\Be
\label{phi-d}
|\partial_t^n \partial_z^\alpha\partial_z^\beta  \phi(t,z,z')|\le C_{n,\alpha, \beta}\, \rho
\Ee
for $(t,z,z')\in \supp a\times \supp \wchil\times \supp \wchil'$. Using this and   \eqref{phi-dl},  one can easily see
\begin{align}
 \label{a-dl}
 |\partial_t^n \partial_z^\alpha\partial_z^\beta  a(t,z,z')|  \le C_{n,\alpha, \beta},  
 \\
 \label{phih-dl}
 |\partial_t^n\partial_z^\alpha\partial_z^\beta \widetilde{\mathcal P}_\sL (t,z,z')|  \le C_{n,\alpha, \beta}
 \end{align}
 for $(t,z,z')\in \supp a\times \supp \wchil\times \supp \wchil'$.
From  \eqref{2dl}  we have 
\begin{align}
   \label{phi-dd}
    \partial_t^2\widetilde{\mathcal P}_h(t,z,z') \sim 1,  \qquad (t,z,z')\in \supp a\times \supp \wchil\times \supp \wchil'.
\end{align}  
and  $\partial_t \widetilde{\mathcal P}_h (0,z,z')=0$.
As before, we now apply the method of stationary phase (\cite[Theorem 7.7.5]{H90}) to the integral  $I_\rho^\cL(x,y)$, and we obtain 
\begin{lem}
Let  $N\in\N$. Then,  for $(z,z')\in\tsupp$, we have
\Be 
\label{for:asym-L}
\begin{aligned}
  I_\rho^\cL(x,y) = \rho^{1-d} \sum_{n=0}^{N-1} (\lambda \rho)^{-\frac12 -n} A_n(z,z') \mspace{1mu}e^{i\lambda\rho \widetilde{\mathcal P}_\sL (0,z,z')} + E_N(z,z'),
  \end{aligned}
  \Ee
where $ A_n,  E_N \in C_c^\infty (B(0,\epc))$ satisfy $|\partial_z^\alpha\partial_z^\beta A_n(z,z')|\le C_{\alpha, \beta}$ and $  \sup_{x,y}|E_N(x,y)|\le C_N\rho^{\frac{2-d}{2}} (\lambda \rho)^{-N} $ 
with $C_{\alpha, \beta}$ and $C_N$ independent of $\lambda,\rho$.
\end{lem}

Let us set
\Be
\label{phil}
\Phi_\sL(z,z'):=\mathcal P_\sL (\mathfrak S_c(z,z'),z,z').
\Ee 
In the next section we investigate the curvature condition of the phase function $\Phi_\sL$.

\subsection{The phase function $\Phi_\sL(z,z')$}
\newcommand{\bv}{\mathbf v}

For $(z,z')\in \supp \wchil\times \supp \wchil'$, we set
\[
\bv(z,z') := z-z',\quad 
\]
and
\[
\mathbf R (z,z') := \cos{\mathfrak S_c(z,z')}\bi_{2d}-\sin{\mathfrak S_c(z,z')}\bs.
\]
It should be noted that $\mathbf R (z,z') \mathbf R (z,z')^\intercal= \bi_{2d} $, so  $\mathbf R (z,z')$  is a rotation matrix.
We occasionally denote $\bv = \bv(z,z')$, $\mathfrak S_c=\mathfrak S_c(z,z')$,  and  $\mathbf R= \mathbf R (z,z')$ for simplicity.

Since $\partial_t\mathcal P_{\cL}(\mathfrak S_c,z,z')=0$, from \eqref{phil} we have $\partial_z \Phi_\sL(z,z')=\partial_z\mathcal P_{\cL}(\mathfrak S_c,z,z')$. Thus, using \eqref{def:phaseL},  we get  $\partial_z \Phi_\sL(z,z')=\frac{(z-z')\cos \mathfrak S_c}{2\sin \mathfrak S_c}+\frac{\bs z'}{2}$.  Now, by \eqref{def:scl}
we have
\begin{align*}
\partial_z \Phi_\sL(z,z') =\frac{\mathbf{R}(z,z')(z-z')}{|z-z'|}+\frac{\bs z}{2}.
\end{align*}
Thus, we see $|\partial_z \Phi_\sL(z,z')-\bs z/2|=1$ for any $z'$. Hence the map $z'\to \partial_z \Phi_\sL(z,z')$ has its rank at most $2d-1$. 
As we shall see later, $ \rank{(\partial_{z'}\partial_z^\intercal \Phi_\sL)}=2d-1$ for any $(z,z')\in\tsupp$.

We compute the mixed Hessian $\pzp\pz\Phi_\cL$. Using  the chain  rule,  we have 
\begin{align}\label{for:hessianL}
\begin{aligned}
    \pzp\pz\Phi_\sL&=\pzp\pz\mathcal P_{\cL}(\mathfrak S_c,z,z')
    +\pzp\partial_t\mathcal P_{\cL}(\mathfrak S_c,z,z') \cdot \pz \mathfrak S_c \\
    &\qquad +\pzp \mathfrak S_c \cdot \pz\partial_t\mathcal P_{\cL}(\mathfrak S_c,z,z') 
    +(\pzp \mathfrak S_c \cdot \pz \mathfrak S_c) \partial_t^2\mathcal P_{\cL}(\mathfrak S_c,z,z').
    \end{aligned}
\end{align}
Here we use  $\ \partial_t\mathcal P_{\cL}(\mathfrak S_c,z,z')=0$. 
From  \eqref{def:phaseL}, one can easily see 
\begin{align*}
    \pzp\pz\mathcal P_{\cL}(\mathfrak S_c,z,z') &= -\frac{\cos \mathfrak S_c}{2\sin \mathfrak S_c}\,\bi-\frac12\,\bs,
    \\
        \pz\partial_t\mathcal P_{\cL} (\mathfrak S_c,z,z') &= -\frac{\bv^\intercal}{2\sin^2 \mathfrak S_c},  
        \\ 
         \pzp\partial_t\mathcal P_{\cL} (\mathfrak S_c,z,z')&= \frac{\bv}{2\sin^2 \mathfrak S_c}.
    \end{align*}
    Since $\sin \mathfrak S_c=|z-z'|/2$,  by \eqref{2dl} we also have 
    \begin{align*}
    \partial_t^2\mathcal P_{\cL} (\mathfrak S_c,z,z')
        = \frac{4\cos \mathfrak S_c}{|z-z'|}, 
        \end{align*}
   Now we note that $ \pz \mathfrak S_c =  {\pz(\sin \mathfrak S_c)}/{\cos \mathfrak S_c}$,  $\pzp \mathfrak S_c= {\pzp(\sin \mathfrak S_c)}/{\cos \mathfrak S_c} $. Thus, differentiating $\sin\mathfrak S_c=|z-z'|/2$, we obtain 
         \Be \label{szp} 
            \begin{aligned} 
               \pz \mathfrak S_c&= \frac{\bv^\intercal}{2\cos \mathfrak S_c |z-z'|},  
               \\
   \pzp \mathfrak S_c&=-\frac{\bv}{2\cos \mathfrak S_c |z-z'|}.
\end{aligned}
\Ee
Combining the identities together with \eqref{for:hessianL}, we obtain 
\begin{align}
\begin{aligned}\label{for:hessianL2}
    \pzp\pz\Phi_\sL 
    =\frac{1}{|z-z'|^3\cos \mathfrak S_c }
    \big(\bv \bv^\intercal -\cos^2 \mathfrak S_c \bv^\intercal \bv \,\bi -\sin \mathfrak S_c\cos \mathfrak S_c \bv^\intercal \bv \,\bs\big).
    \end{aligned}
\end{align}
Since $\bs^2=-\bi$ and $\mathbf R(z,z')^\intercal=\cos \mathfrak S_c\bi+\sin \mathfrak S_c\bs$, it is easy to see the matrix $  \bv \bv^\intercal -\cos^2 \mathfrak S_c \bv^\intercal \bv \,\bi  -\sin \mathfrak S_c\cos \mathfrak S_c \bv^\intercal \bv \,\bs$ 
can be factored as follows: 
\begin{align*}
    (\cos \mathfrak S_c \bv \bv^\intercal - \sin \mathfrak S_c \bv \bv^\intercal \bs - \cos \mathfrak S_c \bv^\intercal \bv \,\bi) \mathbf R(z,z')^\intercal .
\end{align*}
Now we observe 
\begin{align*}
    (\cos \mathfrak S_c \bv \bv^\intercal - \cos \mathfrak S_c \bv^\intercal \bv \,\bi - \sin \mathfrak S_c \bv \bv^\intercal \bs ) v =
    \begin{cases}
    \ \  \  \ \ \ \ 0, &  v=\bv, \\
    - \bv^\intercal \bv\,  \mathbf R(z,z') \bv, &  v=\bs \bv, \\
    -\cos \mathfrak S_c \bv^\intercal \bv \, v, &  v\in  (\text{span}\{\bv,\bs \bv\})^\perp.
    \end{cases}
\end{align*}
For the second case we use the fact that  $\bs$ is skew symmetric,  {\it i.e.,} $\bv^\intercal \bs \bv=0$.
Note that the vectors $\bv^\intercal \bv\,  \mathbf R(z,z') \bv$ and $\cos \mathfrak S_c \bv^\intercal \bv \, v$, $v\neq 0$ are nonzero. 
From this and \eqref{for:hessianL2} we see that the mixed Hessian $\pzp\pz\Phi_\sL$ has  rank $2d-1$ and $\pzp\pz\Phi_\sL\,\mathbf R(z,z') \bv=0$ for any $z'$. 
So, we obtain the following.

\begin{lem} 
\label{hess} For $(z,z') \in \tsupp$ 
the matrix $\partial_{z'}\partial_{z}^\intercal \Phi_\sL(z,z')$ has rank $2d-1$ and its null space is generated by 
the vector 
\[\nu(z,z')=\mathbf R(z,z') \frac{\bv(z,z')}{|\bv(z,z')|}.
\] 
\end{lem}

The vector $\nu(z,z')$ is  the unique (modulo $\pm$) unit vector such that $\pzp\pz\Phi_\sL\,\nu(z,z')=0$. We now  consider the matrix
$\partial_w\partial^\intercal_w \inp{\partial_z\Phi_\sL(z,w)}{\nu(z,z')}|_{w=z}$ for which we can obtain an explicit expression.

\begin{lem}\label{lem:Mform-L}
For $(z,z')\in \tsupp$ we have 
\begin{align}\label{fundamental}
     \partial_w\partial^\intercal_w \inp{\partial_z\Phi_\sL(z,w)}{\nu(z,z')}|_{w=z'} =\frac{-1}{\cos^2 \mathfrak S_c|z-z'|^4}
    \,\mathbf{M}(z,z'),
\end{align}
where the matrix $\mathbf M(z,z')$ is given by
\[
\mathbf M(z,z') = \bv \bv^\intercal-2\cos^2 \mathfrak S_c \bv \bv^\intercal +\cos^2 \mathfrak S_c \bv^\intercal \bv\,\bi + \sin \mathfrak S_c \cos \mathfrak S_c (\bv \bv^\intercal\,\bs-\bs \bv \bv^\intercal).
\]
\end{lem}

\begin{proof}[Proof of Lemma \ref{lem:Mform-L}]
For simplicity we set  
\[\bv_w (z)= \bv(z,w), \quad \mathfrak S_w(z) = \mathfrak S_c(z,w).\]
Using \eqref{for:hessianL2}  and $\mathbf R(z,z')=\cos \mathfrak S_c\bi-\sin \mathfrak S_c\bs$, 
after some computation we have 
\[
\partial_w\pz\Phi_\sL(z,w)\cdot \nu(z,z')  = (\cos \mathfrak S_w|z-w|^3|z-z'|)^{-1}\mathbf k(z,z',w),
\]
where $\mathbf k(z,z',w)\in \R^{2d}$ is given by 
\begin{align*}
\mathbf k(z,&z',w) = \cos \mathfrak S_c\, \bv_w \bv_w^\intercal\bv - \sin \mathfrak S_c\, \bv_w \bv_w^\intercal \bs\bv 
\\
& - \bv_w^\intercal \bv_w \big(\cos \mathfrak S_w\cos(\mathfrak S_w-\mathfrak S_c)\bv+\cos \mathfrak S_w\sin(\mathfrak S_w-\mathfrak S_c)\bs\bv\big).
\end{align*}
Clearly $\mathbf k(z,z',z') = 0$ for any $z'$. Thus, we  have
\[
\partial_w\partial^\intercal_w \inp{\partial_z\Phi_\sL(z,w)}{\nu(z,z')}|_{w=z'} =\frac{1}{\cos \mathfrak S_c|z-z'|^4}\,\partial_w^\intercal\mathbf k(z,z',z').
\]

Now, \eqref{fundamental} follows if we express  the matrix $\partial_w^\intercal\mathbf k(z,z',z')$ in the desired form, i.e., 
\Be  
\label{mmm}
\partial_w^\intercal\mathbf k(z,z',z')= -\frac{1}{\cos \mathfrak S_c} \mathbf M(z,z').
\Ee
Via a routine computation we write
\begin{align*}
          \partial_w^\intercal&\mathbf k(z,z',w)= \cos \mathfrak S_c \partial_w^\intercal(\bv_w \bv_w^\intercal \bv) 
    -\sin \mathfrak S_c \partial_w^\intercal(\bv_w \bv_w^\intercal \bs \bv)
          \\
    & 
    -\cos S_w\cos(S_w-\mathfrak S_c)\bv\cdot\partial_w^\intercal(\bv_w^\intercal\bv_w) +  \bv_w^\intercal \bv_w  \sin S_w\cos(S_w-\mathfrak S_c) \bv\cdot \partial_w^\intercal S_w
    \\
    &
    + \bv_w^\intercal \bv_w \cos S_w\sin(S_w-\mathfrak S_c) \bv\cdot \partial_w^\intercal S_w -\cos S_w\sin(S_w-\mathfrak S_c)\bs\bv\cdot\partial_w^\intercal(\bv_w^\intercal\bv_w) 
    \\
     &
    + \bv_w^\intercal\bv_w \sin S_w \sin(S_w-\mathfrak S_c)\bs\bv\cdot\partial_w^\intercal S_w - \bv_w^\intercal\bv_w\cos S_w \cos(S_w-\mathfrak S_c) \bs\bv\cdot\partial_w^\intercal S_w.
\end{align*}
Here we can discard the $5,6,7$-th terms on the right hand side since they vanish if we put $w=z'$.
Also, we have
\begin{align*}
\partial_w^\intercal(\bv_w \bv_w^\intercal \bv) &= -(\bv_w\bv^\intercal+\bv_w^\intercal\bv\bi), 
\\
\partial_w^\intercal(\bv_w^\intercal\bv_w\bs \bv) &= (\bv_w \bv^\intercal\,\bs- \bv_w^\intercal \bs \bv\,\bi).
\end{align*}
Using these identities, \eqref{szp}, and $4\sin^2{\mathfrak S_c} = \bv^\intercal \bv $ (see \eqref{def:scl}), we now get
\begin{align*}
    \partial_w^\intercal\mathbf k(z,z',z') & = -\cos \mathfrak S_c(\bv \bv^\intercal+\bv^\intercal \bv \bi) -\sin \mathfrak S_c (\bv \bv^\intercal\,\bs- \bv^\intercal \bs \bv\,\bi)
    +2 \cos \mathfrak S_c \,\bv \bv^\intercal 
    \\&     \qquad\qquad -(\cos \mathfrak S_c)^{-1}{\sin^2{\mathfrak S_c} } \bv \bv^\intercal
    +\sin{\mathfrak S_c} \bs \bv \bv^\intercal .
    \end{align*}
Since $\bv^\intercal\bs\bv = 0$, we therefore  obtain  \eqref{mmm}.
\end{proof}

\begin{lem}\label{lem:funda-L} 
Let $(z,z')\in \tsupp $ with $2^{-l}\sim \rho$.
Define the matrix $\mathbf B(z,z')$ by setting
$$\mathbf B(z,z') = \left( \bv_1\dots, \bv_{2d-2}, \frac{\bs\baz}{|\baz|},\,  \frac{\baz}{|\baz|}\right),$$
where $\{\bv_i\}_{i=1}^{2d-2}$ denotes an orthonormal basis of $(\mathrm{span}\{\baz,\bs \baz\})^\perp$.
Then the  $2d\times 2d$  matrix 
    \Be
    \label{fundamental2}
   \mathbf B^\intercal(z,z') \mathbf R(z,z')   \mathbf M(z,z') \mathbf R(z,z')^\intercal\mathbf B(z,z')
    \Ee
    is a diagonal matrix such that  $(2d,2d)$-th entry is zero and the other diagonal entries  $ \lambda_1,\dots,\lambda_{2d-1}$ satisfy  $  \lambda_i\sim \rho^2$, $1\le i\le 2d-1.$

\end{lem}

\begin{proof}[Proof of Lemma \ref{lem:funda-L}]
The matrix $\mathbf M=\mathbf M(z,z')$ in Lemma \ref{lem:Mform-L} can be written as 
\[
\mathbf M = \cos^2 \mathfrak S_c (\bv^\intercal \bv \,\bi + \bs \bv \bv^\intercal \bs - \bv \bv^\intercal)-\mathbf R^\intercal \bs \bv \bv^\intercal \bs\, \mathbf R.
\]
Then, multiplying $\mathbf R=\mathbf R(z,z')$ and $\mathbf R^\intercal=\mathbf R(z,z')^\intercal$ to $\mathbf M(z,z')$, we have 
\begin{align*}
 \mathbf R\mathbf M\mathbf R^\intercal 
= \cos^2 \mathfrak S_c \bv^\intercal \bv \,\bi +
\cos^2 \mathfrak S_c \mathbf R(\bs \bv \bv^\intercal \bs - \bv \bv^\intercal)\mathbf R^\intercal - \bs \bv \bv^\intercal \bs .
\end{align*}
Using $\mathbf R= \cos \mathfrak S_c\bi_{2d}-\sin \mathfrak S_c\bs$, one can easily see  
$ \mathbf R(\bs \bv \bv^\intercal \bs - \bv \bv^\intercal)\mathbf R^\intercal= \bs \bv \bv^\intercal \bs - \bv \bv^\intercal$. Thus we have 
\begin{align*}
\mathbf R\mathbf M\mathbf R^\intercal = \cos^2 \mathfrak S_c (\bv^\intercal \bv \,\bi - \bv \bv^\intercal) - \sin^2 \mathfrak S_c\,\bs \bv \bv^\intercal \bs.
\end{align*}
Then it is easy to verify 
\begin{align*}
  \mathbf R\mathbf M\mathbf R^\intercal  v =
    \begin{cases}
          \bv^\intercal \bv  \cos^2 \mathfrak S_c\,v, &  v\in (\text{span}\{\bv,\bs \bv\})^\perp,\\
           \bv^\intercal \bv \,v, &  v=\bs \bv, \\
    \  \  \ 0, &  v=\bv.
    \end{cases}
\end{align*}
Therefore, $\mathbf R\mathbf M\mathbf R^\intercal $ has the eigenvalues $\bv^\intercal \bv  \cos^2 \mathfrak S_c(z,z'),\,\bv^\intercal \bv ,  0$ with multiplicity $2d-2, 1,1,$ respectively.
It is clear that  $    \bv^\intercal \bv , \, \bv^\intercal \bv  \cos^2 \mathfrak  S_c(z,z') \sim |z-z'|^{2}\sim \rho^2$ because 
$(z,z')\in \tsupp $. Now the elementary (diagonalization)  argument gives \eqref{fundamental2}.
\end{proof}

\subsection{Proof of  Proposition \ref{medrhol}}  We now prove \eqref{eq:medrhol}. 
Let us set 
\[ \mathcal T_\sL f(z)=\int A(z,z')e^{i\lambda\rho\widetilde{\mathcal P}_\sL(0,z,z')} f(z')  dz',\]
where $A\in C_c^\infty( B(0, \epc)\times B(0,\epc))$.  Then we note 
\[\| \ipairl\|_{p}=\rho^{2d} \|(\chil \LL{\eta_\rho}\chil') (\rho\, \cdot+z_0 ,\rho\, \cdot+z'_0 )\|_p .\]
If we use \eqref{integrall} and 
\eqref{for:asym-L} with a large $N$,  for the proof of \eqref{eq:medrhol}  it is sufficient to  show 
\Be
\label{finall}
\|\mathcal T_\sL f\|_p\le C(\lambda\rho)^{-\frac {2d}p} \|f\|_p
\Ee  
for $p_0(2d)< p\le\infty$.  By \eqref{phih-dl} and \eqref{scc-d} it follows that 
\Be
\label{der-bd}
|\partial_z^\alpha \partial_z^\beta \widetilde{\mathcal P}_\sL(0,z,z')|\le C_{\alpha, \beta}.
\Ee

Recalling $\widetilde{\mathcal P}_\sL (0,z,z')=\rho^{-1} \Phi_\sL(\rho z+z_0,\rho z'+ z'_0)$ (see \eqref{phil} and \eqref{phill}), we put 
\[ \widetilde \Phi(z,z')=\rho^{-1} \Phi_\sL(  \rho   \mathbf R_0 z+z_0,  \rho  \mathbf R_0 z'+ z'_0),\]
where 
\[ \mathbf R_0 = \mathbf R(z_0,z'_0) ^\intercal  \mathbf B(z_0,z'_0) . \] 
We define 
\[ \widetilde{\mathcal T}_\sL f(z)=\int A(\mathbf R_0z,\mathbf R_0z')e^{i\lambda\rho  \widetilde \Phi(z,z') } f(z')  dz',\]
Since $\widetilde{\mathcal P}_\sL (0,z,z')=\rho^{-1} \Phi_\sL(\rho z+z_0,\rho z'+ z'_0)$,   \eqref{finall} is equivalent to 
\Be
\label{finalll}
\|\widetilde{\mathcal T}_\sL f\|_p\le C(\lambda\rho)^{-\frac {2d}p} \|f\|_p.
\Ee  
Fixing $z_{2d}\in (-\epc, \epc)$, let us set 
\[  \widetilde{\mathcal T}_\sL^{z_{2d}}  g(z)=\int \widetilde A(z,\zeta) e^{i\lambda\rho  \widetilde \Phi(z,\zeta, z_{2d}) } g(\zeta)  d\zeta, \quad  
\zeta \in \mathbb R^{2d-1},  \]
where $\widetilde A \in C_c^\infty (B(0,\epc)\times B(0,\epc))$.  By the same argument as before, 
the estimate \eqref{finall} follows if we show 
\Be
\label{finall2}  \|\widetilde{\mathcal T}_\sL^{z_{2d}}  g \|_p\le C(\lambda\rho)^{-\frac {2d}p} \|g\|_p, \quad z_{2d}\in  (-\epc, \epc).  
\Ee
To show the estimate it is sufficient to show that the function 
\[   \widetilde \Phi_{z_{2d}}(z,\zeta) :=\widetilde \Phi(z,\zeta,z_{2d}), \quad (z,\zeta)\in \mathbb R^d\times \mathbb R^{d-1}\]
satisfies the Carleson-Sj\"olin condition with ellipticity for ${z_{2d}}\in  (-\epc, \epc)$ on the support of $\widetilde A$ since the desired estimate \eqref{finall2} follows from Theorem \ref{thm:osc}.

Now we note that
\[  \partial_ {z'} \partial_z^\intercal \widetilde \Phi(z,z')= \rho \mathbf R_0^\intercal \partial_ {z'}   \partial_z^\intercal  \Phi_\sL(\rho \mathbf R_0 z+z_0,  \rho \mathbf R_0 z'+z'_0)   \mathbf R_0.\] 
 
By Lemma \ref{hess} we have 
\[ \partial_ {z'} \partial_z^\intercal \widetilde \Phi(z,z') \mathrm N(z,z') =0, \]
where $ \mathrm N(z,z')= \mathbf R_0^\intercal \nu (\rho \mathbf R_0 z+z_0,  \rho \mathbf R_0 z'+z'_0)$. Thus, it is clear that 
\Be
\label{c1-l}
 \partial_\zeta \inpm{\partial_z \widetilde \Phi_{z_{2d}} (z,\zeta)}{\,\mathrm N(z,z') } =0.
\Ee
Since  $\inp{\partial_z \widetilde \Phi(z,z')}{\,\mathrm N(z,z') }=   \partial_z^\intercal  \Phi_\sL(\rho \mathbf R_0 z+z_0,  \rho \mathbf R_0 z'+z'_0) \nu (\rho  \mathbf R_0 z+z_0,  \rho \mathbf R_0  z'+z'_0)$,  we have 
\[ \partial_{z'} \partial_{z'}^\intercal \inpm{\partial_z \widetilde \Phi(0,z')}{\,\mathrm N(0,0) }|_{z'=0} =\rho^2   \mathbf R_0^\intercal  \partial_w\partial^\intercal_w \inp{\partial_z\Phi_\sL(z_0, w)}{\nu(z_0,z'_0)}|_{w=z'_0}  \mathbf R_0.  \]
Thus, using \eqref{fundamental}, we get 
\[ \partial_{z'} \partial_{z'}^\intercal \inpm{\partial_z \widetilde \Phi(0,z')}{\,\mathrm N(0,0) }|_{z'=0}  = \frac{-\rho^2}{\cos^2 \mathfrak S_c(z_0, z_0')  |z_0-z_0'|^4}  \mathbf R_0^\intercal  \mathbf M(z_0, z_0')  \mathbf R_0 . \] 

From Lemma \ref{lem:funda-L},
$\mathbf R_0^\intercal  \mathbf M(z_0, z_0')\mathbf R_0$ is a diagonal matrix with its nonzero diagonal entries  $\sim \rho^2$.  Thus, 
$\mathbf N= \partial_{z'} \partial_{z'}^\intercal \inpm{\partial_z \widetilde \Phi(0,z')}{\,\mathrm N(0,0) }|_{z'=0}$ is a diagonal matrix such that 
$-\mathbf N_{i,i}\sim 1$  for $1\le i\le 2d-1$ and $\mathbf N_{2d, 2d}=0$. Here we use 
$\cos^2 \mathfrak S_c(z_0, z_0')\gtrsim  c_0$ and   $|z_0-z_0'|\sim\rho$. Therefore, the matrix  $\partial_\zeta\partial_\zeta^\intercal \inpm{\partial_z \widetilde \Phi_{z_{2d}}(0,\zeta)}{\,\mathrm N(0,0) }$ is 
a nonsingular diagonal matrix of which diagonal entries are all negative and their absolute values are comparable to $1$. It is clear that $ |\partial_z^\alpha \partial_z^\beta \mathrm N(z,z')|\le C_{\alpha, \beta}$. By  this and \eqref{der-bd} we note that 
$\widetilde \Phi(z,z')$ is a smooth function with bounded derivatives.  
 Thus,  now taking $\epc$ small enough,  we see that the matrix 
\[
\partial_\zeta\partial_\zeta^\intercal \inpm{\partial_z \widetilde \Phi_{z_{2d}}(z,\zeta)}{\,\mathrm N(z,z') },  \quad z'_{2d}\in  (-\epc, \epc) \]
has $2d-1$ negative eigenvalues  of which absolute values are comparable to $1$   for all $(z,\zeta)\in B(0,\epc)\times B(0,\epc)$. 
This and \eqref{c1-l} verify the conditions \eqref{cs1}\,--\,\eqref{elliptic}. Therefore, Theorem \ref{thm:osc} gives  the desired estimate \eqref{finall2}. 
\qed

\section{Lower bound on the  summability index 
of $S_\lambda^\delta(\mathcal H)$
}

In this section we obtain a new lower bound on the summability index $\delta$ for uniform boundedness of $S_\lambda^\delta(\mathcal H)$ on $L^p$.

\begin{prop}\label{prop:necessary}
Let $2<  p\le\infty$. The uniform bound  $\|S_\lambda^\delta(\mathcal H) \|_p\le C$ 
holds only if $\delta> \delta(d, p)$ and 
\[\delta\ge \gamma(d, p):=  -\frac{1}{3p}+\frac d3\Big(\frac12-\frac1p\Big).\]
\end{prop}

In particular, when $d=1$ the uniform bound \eqref{conj:brhm} holds only if  $\delta\ge 0$ for $ 2\le p\le 4$ and  $\delta\ge  -\frac{1}{3p}+\frac{1}{3}(\frac12-\frac1p)$ 
for $p\ge 4$. This coincides with Thangavelu's result \cite[Theorem 2.1]{Th89a}.  In higher dimensions, i.e., $d\ge 2$, it is necessary for \eqref{conj:brhm} that 
$\delta>\delta(d,p)$ if $ \frac{2(2d-1)}{2d-3}\le p\le\infty$;  $\delta\ge \gamma(d,p)>\delta(d,p)$ if $ \frac{2(d+1)}{d}\le p\le \frac{2(2d-1)}{2d-3}$; 
$\delta>0$ if  $2< p\le \frac{2(d+1)}{d}$. (See Figure 1.)
 
The second lower bound  $\delta\ge \gamma(d, p)$ in Proposition \ref{prop:necessary} is a consequence of the following two lemmas. 

\begin{lem}\label{prop:cterexam}
Let $1\le p,q\le\infty$ and $\lambda\gg 1$. Then, for each $\lambda\gg 1$  there exists  $f_\lambda \in\mathcal S(\R^d)$ such that 
\Be
\label{lower}
\|\Pi_\lambda^\cH f_\lambda \|_q\ge C\lambda^{-\frac{1}{3p}+\frac d6(1-\frac1p-\frac1q)}\|f_\lambda\|_p
\Ee
 with $C$ independent of $\lambda$.
\end{lem}

\begin{lem}\label{prop:cterexam1}
Let $1\le p\le\infty$, $\lambda\gg 1$, and let $E,F$ be measurable subsets of $\R^d$. Then,  we have
\[
\| \chi_E \Pi_\lambda^\cH \chi_F \|_p \le C\lambda^\delta  \sup_{t>0}\| \chi_E S^\delta_t(\cH) \chi_F\|_p. 
\]
\end{lem}

\begin{proof}[Proof of Proposition \ref{prop:necessary}] 

Suppose that $\|S_\lambda^\delta(\mathcal H) \|_p\le C$ holds for some $p\in(2,\infty]$ and $\delta$ and  $E$ is a compact subset of $\R^d$ which contains the origin as a point of density. 
Since $\|\chi_E S_\lambda^\delta(\mathcal H)\chi_E\|_{p}\le \|S_\lambda^\delta(\mathcal H)\|_{p}$, 
$
\|\chi_E S_\lambda^\delta(\mathcal H)\chi_E\|_{p}\le C
$
with a constant $C$ independent of $\lambda$. 
Since the principal part of $\cH$ is $-\Delta$, the transplantation theorem of Kenig, Stanton, and Tomas \cite[Theorem 3]{KST82} gives 
 $\|(1-\Delta)_+^\delta\|_{p}\le C$. It is well known that $\|(1-\Delta)_+^\delta\|_{p}\le C$ only if $\delta>\max\{0,-\frac12+d(\frac12-\frac1p)\}$ when $p\neq 2$ (see for example \cite{Herz, F71}).

Thus, it suffices to show that $\delta\ge   \gamma(d, p)$.  From Lemma  \ref{prop:cterexam}, it follows that 
$
    \lambda^{ \gamma(d, p)}\le C\|\Pi_\lambda^\mathcal H\|_{p}.
$
On the other hand, by Lemma \ref{prop:cterexam1}  we have   $\|\Pi_\lambda^\mathcal H\|_{p}\le C\lambda^\delta$. 
Thus, we obtain
$
\lambda^{ \gamma(d, p)}\le C\lambda^\delta
$
for a constant $C$ independent of $\lambda$.
Letting $\lambda$ tend to infinity,  it follows that $\delta\ge  \gamma(d, p)$ as desired.
\end{proof}

Let $\Gamma$ denote the Gamma function and denote  $t_+=\max(t,0)$ and $t_-=- \min(t,0)$ for $t\in \mathbb R$.  We consider the distributions 
\[ 
\chi_\pm ^{\nu}=\frac{x_\pm^{\nu}}{\Gamma(\nu+1)},\quad \nu\in\C
\]
which are given by analytic continuation of the functions $\displaystyle{\frac{x_\pm ^{\nu}}{\Gamma(\nu+1)}}$, ${Re}(\nu)>-1$.  
See  \cite[Section 3.2]{H90} for more about the distribution. 
In order to prove Lemma \ref{prop:cterexam1} we use the identity
\begin{align}\label{for:FH}
    F(\cH) = \int F^{(\nu)}(t)\frac{(t-\cH)_+^{\nu-1}}{{\Gamma(\nu)}}dt 
    =\int F^{(\nu)}(t)t^{\nu-1} \frac{S_t^{\nu-1}(\cH)}{{\Gamma(\nu)}}dt,
\end{align}
which holds for any $\nu\in\C$ and $F\in C_c^\infty([0,\infty))$. Here $F^{(\nu)}$ denotes the  Weyl fractional derivative of $F$ of order $\nu$ which is defined by $F^{(\nu)}=  F\ast \chi_- ^{-\nu-1}$ for $F$ supported in $[0,\infty)$. Note that  $t$ in the integrand can be assumed to be positive  because $(t-\cH)_+^\delta=0$ if $t\le 0$.
The identity \eqref{for:FH} can be shown  using the convolution property of the distribution $\chi_-^{\nu}$, see \cite[pp. 3235-3236]{CHS} for the detail.

\begin{proof}[Proof of Lemma \ref{prop:cterexam1}] 
Now let $\eta$ be a smooth bump function such that $\eta(t) = 1$ if $|t|\le\frac12$ and $\eta(t)=0$ if $|t|\ge 1$ and $\lambda\in 2\N_0+d$.
Since $\lambda\in 2\N_0+d$, taking $\nu=\delta+1$ and $F = \eta(\cdot-\lambda)$ in \eqref{for:FH},  we get 
\[
 \Pi_\lambda^\mathcal H= \eta(H-\lambda)= \frac{1}{{\Gamma(\delta+1)}}\int (\eta(\cdot-\lambda))^{(\delta+1)}(t)t^\delta S_t^\delta(\cH) dt.
\]
A simple calculation shows $|(\eta(\cdot-\lambda))^{(\delta+1)}(t)|\le C (1+|t-\lambda|)^{-(\delta+1)}\chi_{\{t\le\lambda+10\}}$ for a constant $C$.
Thus, combining this with the above identity,  we have 
\begin{align*}
    \|\chi_E \Pi_\lambda^\mathcal H \chi_F \|_{p}&\le C\int |(\eta(\cdot-\lambda))^{(\delta+1)}(t)| t^\delta \|\chi_E S_t^\delta(\cH)\chi_F\|_{p} dt \\
    &\lesssim  \lambda^\delta  \sup_{t>0} \|\chi_E S_t^\delta(\cH)\chi_F\|_{p}.   \qedhere
\end{align*}
\end{proof}

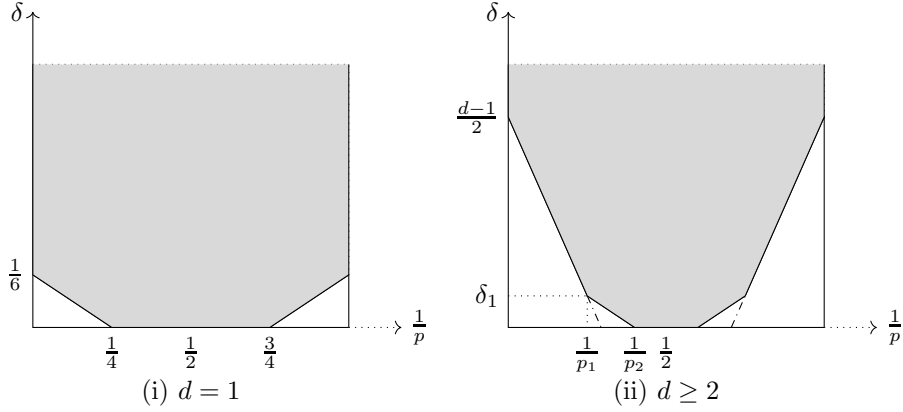
\begin{figure}[!t]
\hspace{-50pt}
\begin{minipage}[t]{.35\textwidth}
\begin{tikzpicture}[scale=0.7]
\draw[<-] (0,6) node[left]{$\delta$}--(0,0)--(6,0);
\draw[dotted, ->] (6,0)--(7,0) node[right]{$\frac1p$};
\draw (0,1) node[left]{$\frac16$}--(1.5,0)--(4.5,0)--(6,1)--(6,0);
\draw[dotted] (0,5)--(6,5)--(6,1);
\node[below] at (3,0) {$\frac12$};
\node[below] at (4.5,0) {$\frac34$};
\node[below] at (1.5,0) {$\frac14$};
\draw[fill=gray! 30]  (0,5)--(0,1)--(1.5,0)--(4.5,0)--(6,1)--(6,5);
\node[below] at (3,-0.8) {(i) $d=1$};
\end{tikzpicture}
\end{minipage}
\quad \quad \qquad 
\begin{minipage}[t]{.35\textwidth}
\begin{tikzpicture}[scale=0.7]
\draw[<-] (0,6) node[left]{$\delta$}--(0,0)--(6,0);
\draw[dotted, ->] (6,0)--(7,0) node[right]{$\frac1p$};
\draw (0,4) node[left]{$\frac{d-1}{2}$}--(1.5,0.6)--(2.4,0)--(3.6,0)--(4.5,0.6)--(6,4)--(6,0);
\node[below] at (3,0) {$\frac12$};
\draw[dotted] (0,5)--(6,5)--(6,4);
\draw[dotted] (0,0.6)--(1.5,0.6);
\draw[dotted] (1.5,0.6)--(1.5,0);
\draw[dash dot] (1.5,0.6)--(30/17,0);
\draw[dash dot] (72/17,0)--(4.5,0.6);
\node[left] at (0,0.6) {$\delta_1$
};
\node[below] at (1.5,0) {$\frac{1}{p_1}$};
\node[below] at (2.4,0) {$\frac{1}{p_2}$};
\draw[fill=gray! 30]  (0,5)--(0,4)--(1.5,0.6)--(2.4,0)--(3.6,0)--(4.5,0.6)--(6,4)--(6,5);
\node[below] at (3,-0.8) {(ii) $d\ge 2$};
\end{tikzpicture}
\end{minipage}
\caption{Lower bound on the summability index $\delta$: $p_1=\frac{2(2d-1)}{2d-3}$, $p_2=\frac{2(d+1)}{d}$, and $\delta_1=\frac{1}{2(2d-1)}$.}
\end{figure}

We now turn to the proof of Lemma  \ref{prop:cterexam} which  is similar to that of  \cite[Proposition 7.10]{JLR1}.
Let $Q$ be the cube which is given  by
\[
Q=\{x\in \R^d : \lambda^\frac12/200\le |x_1|\le  \lambda^\frac12/100, \ |x_i|\le 10^2\lambda^{\frac16},\  2\le i\le d\,\},
\]
and set \[x_\ast=(\lambda^\frac12-10^2\lambda^{-\frac16})\mathbf{e}_1.\]
We  first claim that there is a point $x_0\in B(x_\ast,10^2\lambda^{-\frac16})$ such that 
\begin{align}\label{clm:Qlower}
\int_Q |\Pi_\lambda^\mathcal H(x_0,y)|^2 dy \gtrsim \lambda^{\frac{d-2}{6}}.
\end{align}
Assuming this for the moment, we  consider
\[
f_\lambda (x) = \chi_Q(x) \Pi_\lambda^\mathcal H(x_0,x).
\]
Then, we shall show  the following two estimates:
\begin{align}\label{clm:pflower}
   \|\Pi_\lambda^\mathcal H f_\lambda&\|_{L^q(B(x_\ast, c\lambda^{-\frac16}))}\gtrsim \lambda^{\frac{d-2}{6}-\frac{d}{6q}}, \\
    \label{clm:fupper}
   & \|f_\lambda\|_p\lesssim \lambda^{-\frac13+\frac{d+2}{6p}}
\end{align}
with the implicit constants  independent of $\lambda$ and the constant $c>0$ to be chosen sufficiently small.
Combining  \eqref{clm:pflower} and \eqref{clm:fupper}  immediately yields  the 
desired property \eqref{lower}.  

Thus, it remains  to show  our claim  \eqref{clm:Qlower} and the inequalities   \eqref{clm:pflower} and \eqref{clm:fupper}.  
Compared with the typical construction for the Laplacian,  the proofs of them are somewhat involved.

\subsection{Proof of  \eqref{clm:Qlower}}
To show \eqref{clm:Qlower} we use the following which can be found in  \cite[Lemma 5.1]{KT05}.
\begin{lem}\label{lem:asymh} 
Let $\mu=\sqrt{2k+1}$ and define
\[
    s^-_\mu(t)=\int_0^t\sqrt{|\tau^2-\mu^2|}d\tau\quad \text{and}\quad s^+_\mu(t)=\int^t_\mu\sqrt{|\tau^2-\mu^2|}d\tau. 
\]
Then  the following hold: 
\[    h_{2k}(t)
                =\begin{cases}
                       a_{2k}^-(\mu^2-t^2)^{-\frac14} \big(\!\cos s^-_\mu(t)+\mathcal E\big),     & \qquad\qquad\quad \ |t|<\mu-\mu^{-\frac13}, 
                        \\
                    \qquad \qquad O(\mu^{-\frac16}),                                           & \ \, \mu-\mu^{-\frac13}<|t|<\mu+\mu^{-\frac13}, 
                     \\
                      a^+_{2k}e^{-s^+_\mu(|t|)}(t^2-\mu^2)^{-\frac14}\big(\!1+\mathcal E\big),     & \,  \ \mu+\mu^{-\frac13}<  |t|,
                \end{cases}
\]
\[    h_{2k+1}(t)
              =\begin{cases}
               a_{2k+1}^-(\mu^2-t^2)^{-\frac14}\big(\! \sin\,s^-_\mu(t)+\,\mathcal E\big), & \qquad\qquad \  \, \ |t|<\mu-\mu^{-\frac13},  
               \\
              \qquad \qquad  O(\mu^{-\frac16}),& \mu-\mu^{-\frac13}<|t|<\mu+\mu^{-\frac13},  
              \\
             a^+_{2k+1}e^{-s^+_\mu(|t|)}(t^2-\mu^2)^{-\frac14}\big(\!1+\mathcal E\big),&\mu+\mu^{-\frac13}<|t|,
             \end{cases}
             \]
where
$    |a^{\pm}_k|\sim 1$  and $ \mathcal E=O\big(|t^2-\mu^2|^{-\frac12}\left||t|-\mu\right|^{-1}\big) .$
\end{lem}

We first note that 
\[
\sup_{x\in B(x_\ast,10^2\lambda^{-1/6})} \int_Q \Pi_\lambda(\cH)(x,y)^2 dy = \|\chi_{B(x_\ast,10^2\lambda^{-1/6})} \Pi_\lambda(\cH) \chi_Q\|_{2\to\infty}^2. 
\]
Thus \eqref{clm:Qlower} follows once we find a function  $g$ such that 
\Be\label{desired1}\|\chi_{B(x_\ast,10^2\lambda^{-1/6})} \Pi_\lambda(\cH) \chi_Q g\|_\infty\gtrsim \lambda^{\frac{d-2}{12}}\|g\|_2.\Ee
Let $k:=\frac{\lambda-d}{2}\in\N_0$.
Then, we define the set of indices $J$ by 
\[
J =\Big \{\alpha\in\N_0^d:|\alpha|= k,\, {\lambda^\frac13}/{d}\le \alpha_i\le {2\lambda^{\frac13}}/{d},\, 2\le i\le d\Big\}.
\]
We note that  $\lambda-2\lambda^\frac13\le\alpha_1\le\lambda$ for $\alpha\in J$.  Thus, using Lemma \ref{lem:asymh},  we see
\begin{align*}
    \int_{\lambda^{1/2}-20\lambda^{-1/6}}^{\lambda^{1/2}-10\lambda^{-1/6}}h_{\alpha_1}(t)^2 dt\sim \lambda^{-\frac13},\quad
    \int_{-\frac{\lambda^{-1/6}}{\sqrt{d}}}^{\frac{\lambda^{-1/6}}{\sqrt{d}}} h_{\alpha_i}(t)^2dt\sim \lambda^{-\frac13}, \quad 2\le i\le d
\end{align*}
whenever $\alpha\in J$.
Denote $D=[-\frac{\lambda^{-\frac16}}{\sqrt{d}},\frac{\lambda^{-\frac16}}{\sqrt{d}}]^{d-1}$.
Since $|J|\sim \lambda^{\frac{d-1}{3}}$ and $\Phi_\alpha(x) = \prod_{i=1}^d h_{\alpha_i}(x_i)$, by using the estimates in the above we get
\begin{align*}
     \sum_{\alpha\in J}\int_{\lambda^{1/2}-20\lambda^{-1/6}}^{\lambda^{1/2}-10\lambda^{-1/6}}\int_D\Phi_\alpha(x_1,x')^2 dx'dx_1 
     \sim \lambda^{-\frac13}.
\end{align*}
Thus there exists a point $\widetilde x_*\in [\lambda^{1/2}-20\lambda^{-1/6}, \lambda^{1/2}-10\lambda^{-1/6}]\times D\subset B(x_\ast,10^2\lambda^{-\frac16})$ such that $\sum_{\alpha\in J}\Phi_\alpha^2(\widetilde x_*)\sim \lambda^{\frac{d-2}{6}}$.
Now we set
\[
g(x)= \sum_{\alpha\in J} \Phi_\alpha(\widetilde x_*)\Phi_\alpha(x).
\]
We proceed to prove that $g$ satisfies \eqref{desired1}.
By orthogonality, we have $\|g\|_2 = \big(\sum_{\alpha\in J}\Phi_\alpha(\widetilde x_*)^2\big)^{\frac12}\lesssim\lambda^{\frac{d-2}{12}}$, so \eqref{desired1}
follows if we show 
\Be
\label{desired2}\|\chi_{B(x_\ast,10^2\lambda^{-1/6})} \Pi_\lambda(\cH) \chi_Q g\|_\infty\gtrsim \lambda^{\frac{d-2}{6}}.
\Ee
We write 
\begin{align*}
    \Pi_\lambda(\cH)\chi_Q g(x) &= \sum_{\alpha:|\alpha|=k}\sum_{\beta\in J}\Phi_\alpha(x)\Phi_\beta(\widetilde x_*)\int_Q \Phi_\alpha(y)\Phi_\beta(y) dy \\
    &=:\mathbf{I}(x)+\mathbf{II}(x), 
\end{align*}
where 
\begin{align*}
\mathbf{I}(x)=\sum_{\alpha\in J} \Phi_\alpha(x)\Phi_\alpha(\widetilde x_*)\int_Q \Phi_\alpha(y)\Phi_\alpha(y) dy, \\
\mathbf{II}(x)= \sum_{\alpha\neq\beta} \Phi_\alpha(x)\Phi_\beta(\widetilde x_*)\int_Q \Phi_\alpha(y)\Phi_\beta(y) dy. 
\end{align*}
We first  handle $\mathbf{I}$.
By the choice of $J$,  $\lambda-2\lambda^{\frac13} \le  \alpha_1\le \lambda$ and ${\lambda^\frac13}/{d}\le \alpha_i\le {2\lambda^{\frac13}}/{d}  $ for $i=2,\dots, d$. Thus, using Lemma \ref{lem:asymh} we see that 
\[
    \int_{\lambda^{\frac12}/200\le|t|\le \lambda^{\frac12}/100}h_{\alpha_1}(t)^2 dt\sim 1,\quad
    \int_{|t|\le 10^2\lambda^{\frac16}} h_{\alpha_i}(t)^2 dt\sim 1, \quad 2\le i\le d
\]
for any $\alpha\in J$. Thus, we have $\int_Q\Phi_\alpha(x)^2dx\gtrsim 1$ for  $\alpha\in J$, so $\mathbf{I}(x)$ satisfies the estimate \[|\mathbf I(x_*)|\gtrsim \lambda^{\frac{d-2}{6}}\]
because  $\sum_{\alpha\in J}\Phi_\alpha^2(x_*)\sim \lambda^{\frac{d-2}{6}}$.
For $\mathbf{II}(x)$, we make use of the following formula
\[
\int_{-l}^l h_u(t)h_v(t) dt = \frac{1+(-1)^{u+v}}{\sqrt{2}(u-v)}\big(\sqrt{u+1}\,h_{u+1}(l)h_v(l)-\sqrt{v+1}\,h_u(l)h_{v+1}(l)\big), \quad u\neq v.
\]
See \cite[pp.24--25]{JLR1} for its proof. By Lemma \ref{lem:asymh} it follows that $|h_v(10^2\lambda^{\frac16})|\lesssim e^{-c\lambda^{\frac13}}$ for some $c>0$ when $v\le 2\lambda^{\frac13}d^{-1}$. Also, if $\alpha\neq\beta$,  there exists at least one $i$'s such that $2\le i\le d$ and $\alpha_i\neq\beta_i$ because $|\alpha|=|\beta|=k$. Therefore, we have
\[
\bigg|\int_Q\Phi_\alpha(y)\Phi_\beta(y)dy\bigg| \le \prod_{i=2}^d\int_{-10^2\lambda^{\frac16}}^{10^2\lambda^{\frac16}} h_{\alpha_i}(t)h_{\beta_i}(t)dt\lesssim \lambda^N e^{-c\lambda^\frac13}
\]
for some constants $c>0$ and $N$. From Lemma \ref{lem:asymh} we see $\mathbf{II}(x_*) = O(\lambda^N e^{-c\lambda^\frac13})$.
Combining this and the lower bound for $\mathbf{I}(x_*)$ together with the estimate for $\|g\|_2$,  we obtain
\eqref{desired2}  taking  $\lambda$ large enough.
\qed

\subsection{Proof of   \eqref{clm:pflower} } To show \eqref{clm:pflower},  we make use of the following Lemma, which we prove later.

\begin{lem}
\label{deriv}
    Let $\lambda\in 2\N_0+d$ and $\mu\in[\lambda^{-\frac23},\frac14]$. 
    Suppose that $h\in \mathcal S(\mathbb R^d)$ is a  eigenfunction of $\cH$ with the eigenvalue $\lambda$, {\it i.e.,} $\cH h(x) = \lambda h(x)$. If  
    $
    \lambda^{\frac12}(1-2\mu)\le |y_0|\le \lambda^{\frac12}(1-\mu), 
    $
    then  for any $\alpha\in \N_0^d$ we have 
    \Be
    \label{eq:deriv} |\partial_y^\alpha h(y_0)|\le C (\lambda\mu)^{\frac{|\alpha|}{2}}\|h\|_{L^\infty(B(y_0,2(\lambda\mu)^{-\frac12}))}\Ee 
    with $C$ independent of $\lambda$, $\mu$ and $h$.
\end{lem}

We start by recalling  the estimate
\begin{align}\label{est:bd1inf}
    \|\chi_{A_{\lambda,0}}\Pi_\lambda^\mathcal H\chi_{A_{\lambda,0}}\|_{1\to \infty}\lesssim \lambda^{\frac{d-2}{6}},
\end{align}
which can be found in  (\cite[p.79]{JLR1}, \cite[p.375]{KT05}). Here $A_{\lambda,0}:=\{x\in\R^d : ||x|-\lambda^{\frac12}|\le 10^3\lambda^{-\frac16}\}$.
This is equivalent to the estimate  $|\Pi_\lambda^\mathcal H(x,y)|\lesssim \lambda^{\frac{d-2}{6}}$ for any $x,y\in A_{\lambda,0}$.
Now we note that
\Be
\label{haha}
\Pi_\lambda^\mathcal H f_\lambda (x) = \int_Q \Pi_\lambda^\mathcal H(x_0,y) \Pi_\lambda^\mathcal H(x,y) dy.
\Ee
Using the Cauchy-Schwartz inequality and orthogonality between the Hermite functions $\Phi_\alpha$, we get 
\[
|\Pi_\lambda^\mathcal H  f_\lambda(x)|\le \bigg(\int \Pi_\lambda^\mathcal H(x_0,y)^2 dy \int \Pi_\lambda^\mathcal H(x,y)^2 dy\bigg)^{\frac12}
\le\big(\Pi_\lambda^\mathcal H(x_0,x_0)\Pi_\lambda^\mathcal H(x,x)\big)^\frac12.
\]
By \eqref{est:bd1inf}, we see that $|\Pi_\lambda^\mathcal H f_\lambda(x)|\lesssim \lambda^{\frac{d-2}{6}}$ for every $x\in A_{\lambda,0}$.
Since $\Pi_\lambda^\mathcal H f_\lambda \in \mathcal S(\mathbb R^d)$ and trivially is an eigenfunction of $\mathcal H$ with the eigenvalue $\lambda$, 
by Lemma \ref{deriv} we see 
\[ |\nabla (\Pi_\lambda^\mathcal H f_\lambda)(x)|\lesssim C \lambda^{\frac{d-1}{6}} \]
if $ ||x|-\lambda^{\frac12}|\le 10^2\lambda^{-\frac16}$. 
On the other hand,  from \eqref{clm:Qlower} and \eqref{haha} we have $\Pi_\lambda^\mathcal H f_\lambda (x_0)\gtrsim \lambda^{\frac{d-2}{6}}.$ 
Thus, it follows that 
\[ \Pi_\lambda^\mathcal H f(x)\gtrsim \lambda^{\frac{d-2}{6}}\]
if  $x\in B(x_0,c\lambda^{-\frac16})$ with a sufficiently small $c>0$.  We therefore get 
\eqref{clm:pflower}.   \qed

We now turn to the proof of Lemma \ref{deriv}. 

\begin{proof}[Proof of Lemma \ref{deriv}]
 To show \eqref{eq:deriv} we may assume  $\|h\|_{L^\infty(B(y_0,2(\lambda\mu)^{-\frac12}))} = 1$. 
    Let $\varphi$ be a smooth cutoff function such that $\supp \varphi \subset B(0,2)$ and $\varphi\equiv 1$ on $B(0,1)$. 
    Let us set $\phi_{y_0}(y)= \varphi(\sqrt{\lambda\mu} (y-y_0))$. 
     Then  by inversion  we write
    \begin{align*}
    \partial_x^\alpha h(y_0) = \partial_x^\alpha(\varphi_{y_0}h)(y_0)
    = (2\pi)^{-d} \iint (i\xi)^{\alpha} &\varphi_{y_0}(y)e^{i(y_0-y)\cdot \xi} h(y) dy d\xi.
    \end{align*}
  Setting $\phi_{K}(\xi)=  \varphi({\xi}/{K\sqrt{\lambda\mu}})$ with a large positive constant $K$, we split the integral 
  \begin{align*}
    \partial_x^\alpha h(y_0) = \mathrm I+ \mathrm {I\!I},
    \end{align*}
    where 
    \begin{align*}
   \mathrm I  &:=  (2\pi)^{-d} \iint (i\xi)^{\alpha} {\phi_K}(\xi)  \phi_{y_0}(y) e^{i(y_0-y)\cdot \xi} h(y) dy d\xi ,
    \\
     \mathrm {I\!I}&:= (2\pi)^{-d} \iint (i\xi)^{\alpha} (1-\phi_K(\xi) ) \phi_{y_0}(y)   e^{i(y_0-y)\cdot \xi} h(y) dy d\xi.
    \end{align*}
  For $  \mathrm I$ we have 
    \[
    |{\mathrm I}|\lesssim  \iint |\xi|^{|\alpha|} | \varphi_K(\xi)  \phi_{y_0}(y)  h(y)| dy d\xi\le C(\lambda\mu)^{\frac{|\alpha|}{2}}.
    \]
    Since  $(|\xi|^2+\Delta_y)h(y)=|\xi|^2+{|y|^2}-\lambda$, we may  write
    \[
    \mathrm {I\!I} 
    = C\iint {(i\xi)^\alpha}\frac{(1- \varphi_K(\xi) ) \phi_{y_0}(y) }{|\xi|^2+y^2-\lambda} e^{i(y_0-y)\cdot \xi} (\xi^2+\Delta_y)h(y) dy d\xi.
    \]
    By integration by parts, this is equal to
    \[
    C\iint \xi^{\alpha} (1-\varphi_K(\xi)) A_1(y,\xi) e^{i(y_0-y)\cdot \xi} h(y) dy d\xi,
    \]
    where
    \[
    A_1(y,\xi):=\Delta_y\bigg(\frac{ \phi_{y_0}(y) }{\xi^2+y^2-\lambda}\bigg)-2i\nabla_y\bigg(\frac{ \phi_{y_0}(y)}{\xi^2+y^2-\lambda}\bigg)\cdot\xi.
    \]
    We now note that $\xi^2+y^2-\lambda > \frac{K}{2}(\lambda\mu)^{\frac12}+\frac{\xi^2}{2}$ for $(y,\xi)\in{\supp( \phi_{y_0})\times\supp(\phi_K)}$ with a sufficiently large $K$ and $|\partial_y^\alpha \phi_{y_0}(y)|\lesssim (\lambda\mu)^{\frac{|\alpha|}{2}}$ for any $\alpha\in\N_0^d$. Thus, we have  
   $
    |A_1(y,\xi)|\lesssim (\lambda\mu)^{\frac12}{|\xi|^{-1}}.
    $
    Repeating this by $N$ times, we have
    \[
     \mathrm {I\!I} = C_N\iint  \xi^{\alpha} (1-{\phi_K}(\xi)) A_N(y,\xi) e^{i(y_0-y)\cdot \xi} h(y) dy d\xi,
    \]
    where $A_N$ is a smooth function such that $\supp A_N \subset \supp (\phi_{y_0}\otimes\varphi_K)$ and $|A_N(y,\xi)|\lesssim |\xi|^{-N}(\lambda\mu)^{\frac N2}$.
    Since $\|h\|_{L^\infty(B(y_0,2(\lambda\mu)^{-\frac12}))} = 1$, we have 
    \begin{align*}
    | \mathrm {I\!I} |
   \lesssim (\lambda\mu)^{\frac N2}\int_{K\sqrt{\lambda\mu}\le |\xi|}  |\xi|^{|\alpha|-N}   \int_{|y-y_0|\le 2/\sqrt{\lambda\mu}} dy d\xi \lesssim  (\lambda\mu)^{\frac{|\alpha|}{2}}.
    \end{align*}
    Therefore, we obtain \eqref{eq:deriv}.
\end{proof}

\subsection{Proof of  \eqref{clm:fupper}}  
Let us set $\widetilde x_0=\lambda^{-1/2} x_0$. 
Changing variables, we note that  the estimate \eqref{clm:fupper} is  equivalent to  
\Be
\label{clm:fupper2} 
\|   \chi_{\widetilde Q} \Pi_\lambda^\mathcal H(\lambda^\frac12 \widetilde x_0,  \lambda^\frac12 \cdot) \|_p  \lesssim  \lambda^{-\frac13-\frac{d-1}{3p}},
\Ee 
where 
\[\widetilde Q=\big\{x\in \R^d : 1/200\le |x_1|\le  1/100, \ |x_i|\le 10^2\lambda^{-\frac13},\  2\le i\le d\,\big\}.\]
We recall the  representation formula
\[
\Pi_\lambda^\mathcal H = \frac{1}{2\pi}\int_{-\pi}^{\pi} e^{i(\lambda-\cH)t} dt,
\]
which can be easily shown by utilizing the Hermite expansion and the fact that $\lambda\in 2\N_0+d$.
For example, see \cite[p.11]{JLR1} for the detail.  Using the cutoff functions $\nu^l$ (see \eqref{cutoff-nu}) and scaling, 
we can write 
\[  \Pi_\lambda^\mathcal H(\lambda^\frac12 \widetilde x_0,  \lambda^\frac12 x) =\sum_{l=0}^\infty  \HL{\nu^l}(\widetilde x_0,y).\] 
Thus, the desired estimate \eqref{clm:fupper2} follows if  we show
\begin{align*}
    &\| \HL{\nu^0}(\widetilde x_0,\cdot)\|_{L^p(\widetilde Q)}\lesssim \lambda^{-\frac13-\frac{d-1}{3p}}, 
    \\
    &\|\HL{\nu^l}(\widetilde x_0,\cdot)\|_{L^p(\widetilde Q)}\lesssim \lambda^{-\frac{d-1}{3p}-N}2^{l(\frac{d-2}{2}-N)}, \quad l\ge 1.
\end{align*}
We recall  \eqref{cutoff-nu}. Then, using the properties  
\eqref{for:ihsym} 
and 
\eqref{for:ihsym2}, we see that  the above estimates follow if we obtain 
\begin{align}
\label{est:nece0}
    &\| \HL{\nu^0\chi_{(0,\pi)}}(\widetilde x_0,\cdot)\|_{L^p(\widetilde Q)}\lesssim \lambda^{-\frac13-\frac{d-1}{3p}}, 
    \\
        \label{est:necej}
    &\|\HL{\psi(2^l\cdot)}(\widetilde x_0,\cdot)\|_{L^p(\widetilde Q)}\lesssim \lambda^{-\frac{d-1}{3p}-N}2^{l(\frac{d-2}{2}-N)}, \quad l\ge 1.
\end{align}
Since $\widetilde x_0\in 
B((1-10^2\lambda^{-\frac23})\mathbf{e}_1,10^2\lambda^{-\frac23})$, for $y\in \widetilde Q$  we have
\begin{align}\label{est:detbd}
|\mathcal D(\widetilde x_0,y)| = (1-|\widetilde x_0|^2)(1-|y|^2) + |\widetilde x_0|^2|y|^2- \inp{\widetilde x_0}y^2 
\le 10^{3}\lambda^{-\frac23}.
\end{align}
We also have  $|\inp{\widetilde x_0}{y}|\le 10^{-2}$  for $y\in \widetilde Q$.  Thus, 
from \eqref{eq:derivp-H0}  we have
\[
\partial_t \mathcal P_\cH(t,\widetilde x_0,y) = \frac{(\cos t-\inp {\widetilde x_0}{y})^2-\mathcal D(\widetilde x_0,y)}{2\sin^2 t}\gtrsim 2^{2l}
\]
if $t\in \supp \psi(2^l\cdot)$, $l\ge 1$ since $\lambda$ is assumed to be large.  Recalling \eqref{for:ih0}  we have 
\[\HL{{\psi(2^l\cdot)}}(\widetilde x_0,y) = C_d \int \psi(2^lt)(\sin t)^{-\frac d2} e^{i\lambda\mathcal P_\cH(t,\widetilde x_0,y)} dt. \]
Since ${|\partial_t^n(\psi(2^lt)(\sin t)^{-\frac d2})|\lesssim 2^{l(n+\frac d2)}}$ for any $n\ge 1$, applying Lemma  \ref{lem:vander}
to $\HL{\psi^l}(\widetilde x_0,y)$, we obtain  
\[
|\HL{\psi^l}(\widetilde x_0,y)|\lesssim \lambda^{-N}2^{j(\frac{d-2}{2}-N)}, \quad N\in \N.
\]
Therefore, combining this with $|\wt{Q}|\sim \lambda^{-\frac{(d-1)}{3}}$, we get \eqref{est:necej}.

Let $t_*$ be a number in $[0,\pi]$ such that 
\[\cos t_* = \inp{\widetilde x_0}{y}.\]
The point $t_*$ can be  contained in the support of $\nu^0\chi_{(0,\pi)}$.  
To show the estimate \eqref{est:nece0},  we  make additional decomposition away from $t_\ast$. Let us set 
\[ \psi_l(t)= \psi(2^l(t-t_\ast))+\psi(2^l(t_\ast-t)).\] 
and  also set 
\[  \psi_\lambda (t):=\sum_{2^{-l}\le C\lambda^{-1/3}}  \psi_l(t),\]
 where $C$ is a sufficiently large constant.
We make a further decomposition on $\HL{\nu^0\chi_{(0,\pi)}}(\widetilde x_0,y)$ as follows. Then we may write 
\begin{align*}
\HL{\nu^0\chi_{(0,\pi)}}(\widetilde x_0,y) &= \HL{\psi_\lambda \nu^0\chi_{(0,\pi)}} (\widetilde x_0,y) + \sum_{2^{-l}>C\lambda^{-1/3}}  \HL{\psi_l \nu^0\chi_{(0,\pi)}} (\widetilde x_0,y).
\end{align*}
Since $\sin t\sim 1$ on $t\in \supp(\nu^0\chi_{(0,\pi)})\subset[2^{-2}, \pi-2^{-2}]$, 
\[|\HL{\psi_\lambda \nu^0\chi_{(0,\pi)}} (\widetilde x_0,y) |\lesssim\int  |\psi_\lambda (t)|dt\lesssim \lambda^{-\frac13}.\]
Now we note that $(\cos t-\cos t_\ast)^2\sim 2^{-2l}$ on the support of $\psi_l \nu^0\chi_{(0,\pi)}$. 
Since 
$\partial_t \mathcal P_\cH(t,\widetilde x_0,y) = \frac{(\cos t-\cos t_\ast)^2-\mathcal D(\widetilde x_0,y)}{2\sin^2 t}$, 
by \eqref{est:detbd}
we have   
\[  \partial_t \mathcal P_\cH(t,\widetilde x_0,y)\gtrsim 2^{-2l}\] 
for $t\in \supp(\psi_l \nu^0\chi_{(0,\pi)})$.  So, Lemma \ref{lem:vander} gives  
$ \HL{\psi_l \nu^0\chi_{(0,\pi)}} (\widetilde x_0,y) \lesssim \lambda^{-N}2^{l(-1+3N)}$, for any $N\in\N$. Thus, 
 $\sum_{2^{-l}>C\lambda^{-1/3}}| \HL{\psi_l \nu^0\chi_{(0,\pi)}} (\widetilde x_0,y)|\lesssim \lambda^{-1/3}$. 
Combining these estimate, we now get 
\[|\HL{\nu^0\chi_{(0,\pi)}}(\widetilde x_0,y)|\lesssim \lambda^{-\frac13}.\] 
Therefore, we get  \eqref{est:nece0} because $|\wt{Q}|\sim \lambda^{-\frac{(d-1)}{3}}$. This completes the proof of  \eqref{clm:fupper}.
\qed

\subsection*{Acknowledgement} 
This work was supported by the National Research Foundation of Korea (NRF) grant number  2021R1A2B5B02001786.

\end{document}